\documentclass[final,3p,times]{elsarticle}

\newtheorem{theorem}{Theorem}[section]
\newtheorem{lemma}[theorem]{Lemma}
\newtheorem{corollary}[theorem]{Corollary}
\newtheorem{proposition}[theorem]{Proposition}

\newtheorem{result}[theorem]{Result}

\newdefinition{definition}[theorem]{Definition}
\newdefinition{example}[theorem]{Example}
\newdefinition{xca}[theorem]{Exercise}

\newdefinition{remark}[theorem]{Remark}
\newproof{proof}{Proof}
\newproof{vrfctn}{Verification}

\usepackage{times}
\usepackage{amssymb}
\usepackage{amsfonts}
\usepackage{amsmath}
\usepackage{euscript}
\usepackage{color}
\usepackage{graphicx}
\usepackage{rotating}
\usepackage{bm}
\usepackage{curves}
\usepackage[svgnames,dvipsnames]{pstricks}
\usepackage{lettrine}
\usepackage{mathrsfs}
\usepackage{booktabs}
\DeclareGraphicsExtensions{{},.eps}
\providecommand{\tsc}[1]{{\text{\sc#1}}}
\providecommand{\tscn}[2]{{\text{{\sc#1}{\small#2}}}}

\providecommand{\etal}{et al.}

\providecommand{\figref}[1]{{\textup{(Fig.~\ref{#1})}}}
\providecommand{\ffref}[1]{{\textup{(fn\ref{#1})}}}
\providecommand{\ffpref}[1]{{\textup{(fn\ref{#1}, p. \pageref{#1})}}}
\providecommand{\thmref}[1]{{\textup{(Theorem~\ref{#1})}}}
\providecommand{\lemref}[1]{{\textup{(Lemma~\ref{#1})}}}
\providecommand{\crlref}[1]{{\textup{(Corollary~\ref{#1})}}}
\providecommand{\prpref}[1]{{\textup{(Proposition~\ref{#1})}}}
\providecommand{\defref}[1]{{\textup{(Definition~\ref{#1})}}}

\providecommand{\rmkref}[1]{{\textup{(Remark~\ref{#1})}}}
\providecommand{\rstref}[1]{{\textup{(Result~\ref{#1})}}}

\providecommand{\figrefnp}[1]{{\textup{Fig.~\ref{#1}}}}

\providecommand{\thmrefnp}[1]{{\textup{Theorem~\ref{#1}}}}
\providecommand{\lemrefnp}[1]{{\textup{Lemma~\ref{#1}}}}
\providecommand{\crlrefnp}[1]{{\textup{Corollary~\ref{#1}}}}
\providecommand{\prprefnp}[1]{{\textup{Proposition~\ref{#1}}}}
\providecommand{\defrefnp}[1]{{\textup{Definition~\ref{#1}}}}

\providecommand{\rmkrefnp}[1]{{\textup{Remark~\ref{#1}}}}

\providecommand{\xmprefnp}[1]{{\textup{Example~\ref{#1}}}}
\providecommand{\eqrefsatob} [2]{\textup{(\ref{#1}--\ref{#2})}}
\providecommand{\eqrefsab}   [2]{\textup{(\ref{#1}, \ref{#2})}}
\providecommand{\eqrefsabc}  [3]{\textup{(\ref{#1}, \ref{#2}, \ref{#3})}}

\providecommand{\figrefsab}   [2]{{\textup{(Figs.~\ref{#1}, \ref{#2})}}}

\providecommand{\const}{{\rm const}}
\providecommand{\strlngfk}[2]{\genfrac{[}{]}{0pt}{}{#1}{#2}}
\newcommand{\abs}[1]{\lvert#1\rvert}

\begin{document}

\begin{frontmatter}

\title{Representation of the Lagrange reconstructing polynomial by combination of substencils}

\author{G.A. Gerolymos}
\ead{georges.gerolymos@upmc.fr}
\address{Universit\'e Pierre-et-Marie-Curie \textup{(}\tsc{upmc}\textup{)}, 4 place Jussieu, 75005 Paris, France}
\journal{J. Comp. Appl. Math.}


\begin{abstract}
The Lagrange reconstructing polynomial [Shu C.W.: {\em SIAM Rev.} {\bf 51} (2009) 82--126] of a function $f(x)$
on a given set of equidistant ($\Delta x=\const$) points $\bigl\{x_i+\ell\Delta x;\;\ell\in\{-M_-,\cdots,+M_+\}\bigr\}$
is defined as the polynomial whose sliding (with $x$) averages on $[x-\tfrac{1}{2}\Delta x,x+\tfrac{1}{2}\Delta x]$
are equal to the Lagrange interpolating polynomial of $f(x)$ on the same stencil [Gerolymos G.A.: {\em J. Approx. Theory} {\bf 163} (2011) 267--305].
We first study the fundamental functions of Lagrange reconstruction, show that these polynomials have only real and distinct roots,
which are never located at the cell-interfaces (half-points) $x_i+n\tfrac{1}{2}\Delta x$ ($n\in\mathbb{Z}$), and obtain several identities.
Using these identities, we show that there exists a unique representation of the Lagrange reconstructing polynomial on $\{i-M_-,\cdots,i+M_+\}$
as a combination of the Lagrange reconstructing polynomials on Neville substencils [Carlini E., Ferretti R., Russo G.: {\em SIAM J. Sci. Comp.} {\bf 27} (2005) 1071--1091],
with weights which are rational functions of $\xi$ ($x=x_i+\xi\Delta x$) [Liu Y.Y., Shu C.W., Zhang M.P.: {\em Acta Math. Appl. Sinica} {\bf 25} (2009) 503--538],
and give an analytical recursive expression of the weight-functions.
We show that all of the poles of the rational weight-functions are real, and that there can be no poles at half-points.
We then use the analytical expression of the weight-functions, combined with the factorization of the fundamental functions of Lagrange reconstruction,
to obtain a formal proof of convexity (positivity of the weight-functions) in the neighborhood of $\xi=\tfrac{1}{2}$, iff all of the substencils
contain either point $i$ or point $i+1$ (or both).
\end{abstract}

\end{frontmatter}

%
%
%
%
%
%
%
%
%
\section{Introduction}\label{RLRPCS_s_I}
%
%
%
%
%
%
%
%
%

Polynomial interpolation and/or polynomial reconstruction are the basic numerical approximation operations involved in the development of \tsc{weno} schemes \cite{Liu_Osher_Chan_1994a,
                                                                                                                                                                   Jiang_Shu_1996a},
which are widely used \cite{Shu_2009a} for the discretization of (hyperbolic) \tsc{pde}s, particularly when the solution contains discontinuities.
Following Godunov's theorem \cite{vanLeer_2006a}, these schemes introduce nonlinearity in the approximation (with respect to the reconstructed function $h(x)$ or to its cell-averages $f(x)$),
to combine high-order with monotonicity. Central to the development of these methods \cite{Shu_1998a,
                                                                                           Shu_2009a}
is the underlying linear approximation, where the interpolating \cite{Carlini_Ferretti_Russo_2005a,
                                                                      Shu_2009a}
and/or the reconstructing \cite{Shu_1998a,
                                Shu_2009a}
polynomial on a given stencil is represented by a combination of the corresponding (interpolating or reconstructing) polynomials on substencils.
We introduce the following definitions
%
\begin{definition}[{\rm Stencil {\cite[Definition 4.1, p. 283]{Gerolymos_2011a}}}]
\label{Def_AELRP_s_EPR_ss_PR_001}
Consider a 1-D homogeneous computational mesh
\begin{subequations}
                                                                                                       \label{Eq_Def_AELRP_s_EPR_ss_PR_001_001}
\begin{equation}
x_i=x_1+(i-1)\Delta x\qquad\qquad \Delta x = {\rm const}\in\mathbb{R}_{>0}
                                                                                                       \label{Eq_Def_AELRP_s_EPR_ss_PR_001_001a}
\end{equation}
Assume
\begin{equation}
M:=M_-+M_+\geq0
                                                                                                       \label{Eq_Def_AELRP_s_EPR_ss_PR_001_001b}
\end{equation}
The set of contiguous points
\begin{equation}
\tsc{s}_{i,M_-,M_+}:=\left\{i-M_-,\cdots,i+M_+\right\}
                                                                                                       \label{Eq_Def_AELRP_s_EPR_ss_PR_001_001c}
\end{equation}
is defined as the discretization-stencil in the neighborhood of $i$, with $M_-$ neighbors to the left and $M_+$ neighbors to the right.
The stencil $\tsc{s}_{i,M_-,M_+}$~\eqref{Eq_Def_AELRP_s_EPR_ss_PR_001_001c} contains $M+1>0$ points and has a length of $M$ intervals.
If $M_\pm\geq0$ then the stencil contains the pivot-point $i$. If $M_-M_+<0$ then the stencil does not contain the pivot-point $i$.
We will note
\begin{equation}
[\tsc{s}_{i,M_-,M_+}]:=[x_{i-M_-},x_{i+M_+}]\;\subset\;{\mathbb R}
                                                                                                       \label{Eq_Def_AELRP_s_EPR_ss_PR_001_001d}
\end{equation}
the interval defined by the edge-points of the stencil.
\end{subequations}
\qed
\end{definition}
%
%
\begin{definition}[{\rm Neville substencils}]
\label{Def_RLRPCS_s_I_002}
Let $\tsc{s}_{i,M_-,M_+}$ be a discretization stencil on a homogeneous grid \defref{Def_AELRP_s_EPR_ss_PR_001} with
\begin{subequations}
                                                                                                       \label{Eq_Def_RLRPCS_s_I_002_001}
\begin{equation}
M:=M_-+M_+\geq2
                                                                                                       \label{Eq_Def_RLRPCS_s_I_002_001a}
\end{equation}
Assume
\begin{equation}
K_{\rm s}\leq M-1\qquad;\qquad K_{\rm s}\in{\mathbb N}_0
                                                                                                       \label{Eq_Def_RLRPCS_s_I_002_001b}
\end{equation}
The $K_{\rm s}+1\geq1$ substencils
\begin{equation}
\tsc{s}_{i,M_--k_{\rm s},M_+-K_{\rm s}+k_{\rm s}}:=\left\{i-M_-+k_{\rm s},\cdots,i+M_+-K_{\rm s}+k_{\rm s}\right\}\;\forall\;k_{\rm s}\in\{0,\cdots,K_{\rm s}\}
                                                                                                       \label{Eq_Def_RLRPCS_s_I_002_001c}
\end{equation}
each of which contains $M-K_{\rm s}+1$ points and which satisfy
\begin{equation}
\bigcup_{k_{\rm s}=0}^{K_{\rm s}}\tsc{s}_{i,M_--k_{\rm s},M_+-K_{\rm s}+k_{\rm s}}=\tsc{s}_{i,M_-,M_+}
                                                                                                       \label{Eq_Def_RLRPCS_s_I_002_001d}
\end{equation}
\begin{equation}
\ell_{\rm s}\neq m_{\rm s} \iff \left\{\begin{array}{l}\tsc{s}_{i,M_--\ell_{\rm s},M_+-K_{\rm s}+\ell_{\rm s}}\nsubset\tsc{s}_{i,M_--m_{\rm s},M_+-K_{\rm s}+m_{\rm s}}\\
                                                       \tsc{s}_{i,M_--\ell_{\rm s},M_+-K_{\rm s}+\ell_{\rm s}}\neq    \tsc{s}_{i,M_--m_{\rm s},M_+-K_{\rm s}+m_{\rm s}}\\\end{array}\right.\quad
\forall \ell_{\rm s},m_{\rm s}\in\{0,\cdots,K_{\rm s}\}
                                                                                                       \label{Eq_Def_RLRPCS_s_I_002_001e}
\end{equation}
\begin{equation}
\tsc{s}_{i-M_-+k_\mathrm{s}+1,i+M_+-K_\mathrm{s}+k_\mathrm{s}+1}=\Big(\tsc{s}_{i-M_-+k_\mathrm{s},i+M_+-K_\mathrm{s}+k_\mathrm{s}}\setminus\{i-M_-+k_\mathrm{s}\}\Big)\cup\{i+M_+-K_\mathrm{s}+k_\mathrm{s}+1\}
\quad\forall k_\mathrm{s}\in\{0,\cdots,K_{\rm s}-1\}
                                                                                                       \label{Eq_Def_RLRPCS_s_I_002_001f}
\end{equation}
are the $(M-K_{\rm s}+1)$-order\footnote{\label{ff_Def_RLRPCS_s_I_002_001}In the sense that the Lagrange interpolating and reconstructing
                                                                          polynomials on each of the substencils \eqref{Eq_Def_RLRPCS_s_I_002_001c}
                                                                          are $O(\Delta x^{M-K_{\rm s}+1})$-accurate approximations \cite[Proposition 4.6, p. 289]{Gerolymos_2011a}.
                                        }
substencils of $\tsc{s}_{i,M_-,M_+}$, corresponding to the $K_{\rm s}$-level subdivision of $\tsc{s}_{i,M_-,M_+}$.
\end{subequations}
\qed
\end{definition}
%
%
\begin{definition}[{\rm Reconstruction pair {\cite[Definition 2.1, p. 270]{Gerolymos_2011a}}}]
\label{Def_AELRP_s_RPERR_ss_RP_001}
Assume that $\Delta x\in{\mathbb R}_{>0}$ is a \underline{constant} length, and that the functions
$f:I\longrightarrow{\mathbb R}$ and $h:I\longrightarrow{\mathbb R}$ are defined on the interval $I=[a-\frac{1}{2}\Delta x,b+\frac{1}{2}\Delta x]\subset{\mathbb R}$,
satisfying everywhere
\begin{subequations}
                                                                                                       \label{Eq_Def_AELRP_s_RPERR_ss_RP_001_001}
\begin{equation}
f(x)=\dfrac{1}{\Delta x}\int_{x-\frac{1}{2}\Delta x}^{x+\frac{1}{2}\Delta x}{h(\zeta)d\zeta}\quad\forall x\in[a,b]
                                                                                                       \label{Eq_Def_AELRP_s_RPERR_ss_RP_001_001a}
\end{equation}
assuming the existence of the integral in~\eqref{Eq_Def_AELRP_s_RPERR_ss_RP_001_001a}.
We will note the functions $f(x)$ and $h(x)$ related by~\eqref{Eq_Def_AELRP_s_RPERR_ss_RP_001_001a}
\begin{alignat}{6}
h=&R_{(1;\Delta x)}(f)
                                                                                                       \label{Eq_Def_AELRP_s_RPERR_ss_RP_001_001b}\\
f=&R^{-1}_{(1;\Delta x)}(h)
                                                                                                       \label{Eq_Def_AELRP_s_RPERR_ss_RP_001_001c}
\end{alignat}
\end{subequations}
and will call $f$ and $h$
a reconstruction pair on $[a,b]$,
in view of the computation of the 1-derivative.\footnote{\label{ff_Def_AELRP_s_RPERR_ss_RP_001_01}
                                                         By \cite[Lemma 2.2, p. 271]{Gerolymos_2011a}, $\eqref{Eq_Def_AELRP_s_RPERR_ss_RP_001_001a}\;\Longrightarrow\;
                                                         f^{(n)}(x)=\dfrac{h^{(n-1)}(x+\frac{1}{2}\Delta x)-h^{(n-1)}(x-\frac{1}{2}\Delta x)}{\Delta x}\;\forall x\in[a,b]\;\forall n\in\{1,\cdots,N\}$, exactly,
                                                         assuming $f(x)$ and $h(x)$ are of class $C^N[a-\tfrac{1}{2}\Delta x,b+\tfrac{1}{2}\Delta x]$.}\qed
\end{definition}
%
%
\begin{definition}[{\rm Lagrange reconstructing polynomial {\cite[Definition 2.3, p. 271]{Gerolymos_2011a}}}]
\label{Def_AELRP_s_RPERR_ss_RP_002}
Let $p_{  I,M_-,M_+}(x_i+\xi\Delta x;x_i,\Delta x;f)$ denote the Lagrange interpolating polynomial \cite[pp. 186--189]{Henrici_1964a}
of the real function $f:\mathbb{R}\longrightarrow\mathbb{R}$ on the stencil $\tsc{s}_{i,M_-,M_+}$ \defref{Def_AELRP_s_EPR_ss_PR_001}.
Its reconstruction pair \defref{Def_AELRP_s_RPERR_ss_RP_001}
\begin{equation}
p_{R_1,M_-,M_+}(x_i+\xi\Delta x;x_i,\Delta x;f):=[R_{(1;\Delta x)}(p_{I,M_-,M_+})](x_i+\xi\Delta x;x_i,\Delta x;f)
                                                                                                       \label{Eq_Def_AELRP_s_RPERR_ss_RP_002_001}
\end{equation}
will be called the Lagrange reconstructing polynomial on the stencil $\tsc{s}_{i,M_-,M_+}$.\qed
\end{definition}
%

We study representations where the polynomial approximation on $\tsc{s}_{i,M_-,M_+}$ \defref{Def_AELRP_s_EPR_ss_PR_001} is expressed as a weighted sum of the corresponding
polynomial approximations on the $K_\mathrm{s}+1$ substencils \defref{Def_RLRPCS_s_I_002}
\begin{subequations}
                                                                                                       \label{Eq_RLRPCS_s_I_001}
\begin{alignat}{6}
p_{R_1,M_-,M_+}(x_i+\xi\Delta x;x_i,\Delta x;f)=&\sum_{k_\mathrm{s}=0}^{K_\mathrm{s}} \sigma_{R_1,M_-,M_+,K_\mathrm{s},k_\mathrm{s}}(\xi)\;p_{R_1,M_--k_\mathrm{s},M_+-K_\mathrm{s}+k_\mathrm{s}}(x_i+\xi\Delta x;x_i,\Delta x;f)
                                                                                                       \label{Eq_RLRPCS_s_I_001a}\\
p_{  I,M_-,M_+}(x_i+\xi\Delta x;x_i,\Delta x;f)=&\sum_{k_\mathrm{s}=0}^{K_\mathrm{s}} \sigma_{  I,M_-,M_+,K_\mathrm{s},k_\mathrm{s}}(\xi)\;p_{  I,M_--k_\mathrm{s},M_+-K_\mathrm{s}+k_\mathrm{s}}(x_i+\xi\Delta x;x_i,\Delta x;f)
                                                                                                       \label{Eq_RLRPCS_s_I_001b}
\end{alignat}
\end{subequations}
with weight-functions ($\sigma_{  I,M_-,M_+,K_\mathrm{s},k_\mathrm{s}}(\xi)$ in the interpolation case or $\sigma_{R_1,M_-,M_+,K_\mathrm{s},k_\mathrm{s}}(\xi)$ in the reconstruction case)
which are independent of the approximated function ($f(x)$ in the interpolation case or $h(x)$ in the reconstruction case).
The subscripts $M_\pm$, $K_\mathrm{s}$ and $k_\mathrm{s}$ in \eqref{Eq_RLRPCS_s_I_001} indicate that the weight-functions depend on the stencil ($M_\pm$),
on the level of subdivision ($K_\mathrm{s}$) and on the particular substencil ($k_\mathrm{s}$).
Because the weight-functions are independent of the approximated function ($f(x)$ or $h(x)$) they are usually called {\em linear} weights \cite{Shu_2009a}. Alternatively, since the weights combine the interpolating (or
reconstructing) polynomials on the substencils to exactly the interpolating (or reconstructing) polynomial on the entire stencil, they recover the highest possible accuracy (between weighted combinations
of the substencils) and, for this reason, they are alternatively called {\em optimal} (in the sense of accuracy) weights \cite{Jiang_Shu_1996a,
                                                                                                                               Balsara_Shu_2000a}.

The underlying linear interpolation or reconstruction used in \tsc{weno} \cite{Shu_1998a,
                                                                               Shu_2009a}
schemes on the general stencil $\{i-M_-,\cdots,i+M_+\}$ \defref{Def_AELRP_s_EPR_ss_PR_001}
can be obtained by writing the approximation error \cite[(56a), p. 292]{Gerolymos_2011a}
for the $K_\mathrm{s}+1$ ($k_\mathrm{s}\in\{0,\cdots,K_\mathrm{s}\}$) substencils $\{i-M_-+k_\mathrm{s},\cdots,i+M_+-K_\mathrm{s}+k_\mathrm{s}\}$ \defref{Def_AELRP_s_EPR_ss_PR_001},
each of which has an error of $O(\Delta x^{M-K_\mathrm{s}+1})$ \cite[Proposition 4.7, p. 292]{Gerolymos_2011a}.
At any \underline{fixed point $x_i+\xi\Delta x$}, we can in this way construct a $(K_\mathrm{s}+1)\times(K_\mathrm{s}+1)$ linear system
({\em eg} \cite[(13), p. 8489]{Gerolymos_Senechal_Vallet_2009a}) for the weights which linearly combine the approximated values on the substencils
to obtain an $O(\Delta x^{M+1})$-accurate approximation, recovering the accuracy (and indeed the exact value \cite{Shu_1998a,
                                                                                                                   Shu_2009a})
of the entire stencil $\{i-M_-,\cdots,i+M_+\}$, at the chosen fixed point $x_i+\xi\Delta x$.
It is known by numerical experiment \cite{Shu_1998a,
                                          Shu_2009a},
that, for stencils symmetric around $x_i$ ({\em ie} $M_-=M_+$) these linear or optimal weights
for the ($K_\mathrm{s}=M_-=M_+$)-subdivision,
can be calculated at the fixed point $\xi=\tfrac{1}{2}$ ({\em eg} \cite[Tab. 3, p. 8484]{Gerolymos_Senechal_Vallet_2009a}),
{\em ie} for this choice of $\{M\pm,K_\mathrm{s},\xi\}$ the linear system \cite[(13), p. 8489]{Gerolymos_Senechal_Vallet_2009a} is not singular.
Shu \cite{Shu_1998a} has given examples of other choices of $\{M\pm,K_\mathrm{s},\xi\}$ for which the linear system is singular.
Obviously the weights are functions of $\xi$, parametrized by $\{M_\pm,K_{\rm s},k_\mathrm{s}\}$.

The Neville-Aitken algorithm \cite[pp. 204--209]{Henrici_1964a} constructs the interpolating polynomial on $\{i-M_-,\cdots,i+M_+\}$, by recursive combination of the
interpolating polynomials on substencils, with weights which are also polynomials of $x$ \cite[pp. 204--209]{Henrici_1964a}.
Carlini~\etal~\cite{Carlini_Ferretti_Russo_2005a}, working on the Lagrange interpolating polynomial in the context of centered (central) \tsc{weno} schemes,
recognized the connexion between the Neville-Aitken algorithm \cite[pp. 204--209]{Henrici_1964a} and the determination of the optimal weights,
and gave the explicit expression \cite[(3.6,4.10), pp. 1074--1079]{Carlini_Ferretti_Russo_2005a} of the polynomial weight-functions $\sigma_{I,r-1,r,r-1,k_\mathrm{s}}(\xi)$ which combine the Lagrange interpolating polynomials
on the $K_\mathrm{s}+1=(r-1)+1$ substencils $\{i-(r-1)+k_\mathrm{s},\cdots,i+r-(r-1)+k_\mathrm{s}\}$ to obtain the Lagrange interpolating polynomial
on the big stencil $\{i-(r-1),\cdots,i+r\}$ which contains an odd number of $M=2r-1$ intervals and an even number of $M+1=2r$ points.
This result was also confirmed by
Liu~\etal~\cite[(2.2), p. 506]{Liu_Shu_Zhang_2009a} who further gave the analytical expression \cite[(2.18), p. 511]{Liu_Shu_Zhang_2009a}
for the polynomial weight-functions $\sigma_{I,r,r,r,k_\mathrm{s}}(\xi)$ which combine the Lagrange interpolating polynomials
on the $K_\mathrm{s}+1=r+1$ substencils $\{i-r+k_\mathrm{s},\cdots,i+r-r+k_\mathrm{s}\}$ to obtain the Lagrange interpolating polynomial
on the big stencil $\{i-r,\cdots,i+r\}$ which contains an even number of $M=2r$ intervals and an odd number of $M+1=2r+1$ points.
For both cases it is shown \cite{Carlini_Ferretti_Russo_2005a,
                                 Liu_Shu_Zhang_2009a}
that $\forall\xi\in[-1,1]$ the linear weights are positive, and as a consequence the above combination of substencils is convex $\forall\xi\in[-1,1]$.
In a recent work \cite{Gerolymos_2011a_news} we extended these results for the general $K_\mathrm{s}$-level subdivision of an arbitrary stencil
$\tsc{x}_{i-M_-,i+M_+}:=\{x_{i-M_-},\cdots,x_{i+M_+}\}\subset\mathbb{R}$ of $M+1:=M_-+M_++1$ distinct ordered points on an inhomogeneous grid
to $K_\mathrm{s}+1\leq M$ substencils $\tsc{x}_{i-M_-+k_\mathrm{s},i+M_+-K_\mathrm{s}+k_\mathrm{s}}:=\left\{x_{i-M_-+k_\mathrm{s}},\cdots,x_{i+M_+-K_\mathrm{s}+k_\mathrm{s}}\right\}$
($k_\mathrm{s}\in\{0,\cdots,K_{\rm s}\}$), and used a general recurrence relation \cite[(4e), Lemma 2.1]{Gerolymos_2011a_news}
to obtain a simplified expression \cite[Proposition 3.1]{Gerolymos_2011a_news} of the weight-functions for the Lagrange interpolating polynomial \eqref{Eq_RLRPCS_s_I_001b}, and to prove positivity
in the interval $x\in[x_{i-M_-+K_\mathrm{s}-1},x_{i+M_+-K_\mathrm{s}+1}]$ which contains at least 1 cell (at least 2 grid-points) iff  $K_\mathrm{s}\leq\left\lceil\frac{M}{2}\right\rceil$
\cite[Proposition 3.2]{Gerolymos_2011a_news}.

Looking more carefully into \eqref{Eq_RLRPCS_s_I_001b} one notices that it is directly related to M\"uhlbach's theorem \cite[Theorem 2.1, p. 100]{Muhlbach_1978a},
corresponding to \cite[(2.2,2.3), p. 100]{Muhlbach_1978a}. M\"uhlbach \cite{Muhlbach_1978a} expresses the coefficient $\sigma_{  I,M_-,M_+,K_\mathrm{s},k_\mathrm{s}}(\xi)$ \eqref{Eq_RLRPCS_s_I_001b}
in terms of quotients of determinants of interpolation-error functions, directly obtained by the Cramer solution \cite[Proposition 5.1.1, p. 72]{Allaire_Kaber_2008a} of error-eliminating linear systems.
M\"uhlbach \cite{Muhlbach_1978a} studies Chebyshev-systems satisfying interpolatory conditions. In the reconstructing polynomial case \eqref{Eq_RLRPCS_s_I_001a},
the usual linear system approach \cite[(13), p. 8489]{Gerolymos_Senechal_Vallet_2009a} is equivalent to the algorithm of M\"uhlbach \cite[Theorem 2.1, p. 100]{Muhlbach_1978a}
with the important difference that in \eqref{Eq_RLRPCS_s_I_001a} we study polynomials $p_{R_1}(x)$ whose linear functionals $p_I(x)=[R^{-1}_{(1;\Delta x)}(p_{R_1})](x)$ \defref{Def_AELRP_s_RPERR_ss_RP_001}
satisfy interpolatory conditions, so that the existence and uniqueness proofs in \cite[Theorem 2.1, p. 100]{Muhlbach_1978a} are not directly applicable.
Nonetheless, the general recurrence relation for weight-functions proven in \cite[(4e), Lemma 2.1]{Gerolymos_2011a_news}, only requires that the ($K_\mathrm{s}=1$)-level subdivision
can be defined. Therefore, finding a general expression for the weight-functions $\sigma_{R_1,M_-,M_+,K_\mathrm{s},k_\mathrm{s}}(\xi)$ \eqref{Eq_RLRPCS_s_I_001a},
is tantamount to solving the problem of the $(K_\mathrm{s}=1)$-level subdivision for the Lagrange reconstructing polynomial.

Although the reconstructing polynomial \cite{Shu_1998a,
                                             Shu_2009a,
                                             Gerolymos_2011a}
is even more widely used in \tsc{weno} discretizations, the development of practical \tsc{weno} schemes \cite{Jiang_Shu_1996a,
                                                                                                              Balsara_Shu_2000a,
                                                                                                              Gerolymos_Senechal_Vallet_2009a},
invariably followed the aforementioned linear system approach \cite[(13), p. 8489]{Gerolymos_Senechal_Vallet_2009a}, using symbolic calculation.
There is little analytical work on the weight-functions $\sigma_{R_1,M_-,M_+,K_\mathrm{s},k_\mathrm{s}}(\xi)$ which combine the Lagrange reconstructing polynomials
on the $K_\mathrm{s}+1\leq M$ substencils $\{i-M_-+k_\mathrm{s},\cdots,i+M_+-K_\mathrm{s}+k_\mathrm{s}\}$ to obtain the Lagrange reconstructing polynomial
on the big stencil $\{i-M_-,\cdots,i+M_+\}$. Only recently, Liu~\etal~\cite{Liu_Shu_Zhang_2009a} studied
particular families of stencils and subdivisions, using symbolic computation.\footnote{\label{ff_RLRPCS_s_I_001}
                                                                                          Liu~\etal~\cite{Liu_Shu_Zhang_2009a} have examined, using symbolic calculation,
                                                                                          the computation and positivity of linear (optimal) weight-functions
                                                                                          in \tsc{weno} interpolation, reconstruction and integration.
                                                                                         }
Liu~\etal~\cite{Liu_Shu_Zhang_2009a} have concentrated on the usual \tsc{weno} substencils.\footnote{\label{ff_RLRPCS_s_I_002}In the nomenclature of Shu \cite{Shu_1998a,
                                                                                                                                                               Shu_2009a},
                                                                                                     used in Liu~\etal~\cite{Liu_Shu_Zhang_2009a}, stencils are defined in terms of cell-interfaces (half-points),
                                                                                                     and the term nodes in \cite{Liu_Shu_Zhang_2009a} denotes cells,
                                                                                                     so that the stencil $\{i-(r-1),\cdots,i+(r-1)\}$ is defined in \cite[Tab. 3.2, p. 516]{Liu_Shu_Zhang_2009a}
                                                                                                     as $\{i-r+\tfrac{1}{2},\cdots,i+r-\tfrac{1}{2}\}$,
                                                                                                     and the stencil $\{i-(r-1),\cdots,i+r\}$ is defined in \cite[Tab. 3.5, p. 518]{Liu_Shu_Zhang_2009a}
                                                                                                     as $\{i-r+\tfrac{1}{2},\cdots,i+r+\tfrac{1}{2}\}$
                                                                                                    }
In the reconstruction case, it was shown by construction \cite{Liu_Shu_Zhang_2009a} that the optimal weight-functions are not polynomials,
as in the interpolation case \cite{Carlini_Ferretti_Russo_2005a,
                                   Liu_Shu_Zhang_2009a,
                                   Gerolymos_2011a_news},
but, instead, rational functions of $\xi$ ($x=x_i+\xi\Delta x$),
implying that at the poles of these rational functions the weight-functions cannot be defined.
For upwind-biased schemes \cite{Jiang_Shu_1996a,
                                Balsara_Shu_2000a,
                                Gerolymos_Senechal_Vallet_2009a}
the big stencil $\{i-(r-1),\cdots,i+(r-1)\}$ ($r\in{\mathbb N}_{\geq2}$) which is centered around the point $i$, and upwind-biased with respect to the cell-face $i+\tfrac{1}{2}$, is
subdivided \cite[Tab. 3.2, p. 516]{Liu_Shu_Zhang_2009a} into $K_\mathrm{s}+1=r$ substencils $\{i-(r-1)+k_\mathrm{s},\cdots,i+k_\mathrm{s}\}$, $k_\mathrm{s}\in\{0,\cdots,K_\mathrm{s}\}$.
For centered schemes, the big stencil $\{i-(r-1),\cdots,i+r\}$ which is centered with respect to the cell-face $i+\tfrac{1}{2}$ \cite{Shu_1998a,
                                                                                                                                      Qiu_Shu_2002a,
                                                                                                                                      Shu_2009a},
and as a consequence downwind-biased with respect to the point $i$, is subdivided into $K_\mathrm{s}+1=r+1$ substencils $\{i-(r-1)+k_\mathrm{s},\cdots,i+k_\mathrm{s}\}$, $k_\mathrm{s}\in\{0,\cdots,K_\mathrm{s}\}$.
In \cite[(3.2--3.4), p. 514]{Liu_Shu_Zhang_2009a} an algorithm is sketched for computing the rational weight-functions,
which are tabulated up to $r=7$ \cite[Tab. 3.2, p. 516]{Liu_Shu_Zhang_2009a} for the upwind-biased case (even number of intervals)
and up to $r=6$ \cite[Tab. 3.5, p. 518]{Liu_Shu_Zhang_2009a} for the centered case (odd number of intervals).
We remarked in \cite[p. 298]{Gerolymos_2011a} that both these families can be grouped together as the subdivision of the general stencil
$\{i-\lfloor\frac{M}{2}\rfloor,\cdots,i+M-\lfloor\frac{M}{2}\rfloor\}$ into $K_\mathrm{s}+1=\left\lceil\frac{M}{2}\right\rceil+1$ substencils, in the range $M\in\{2\cdots,11\}$.
These weights were further analyzed to determine the regions of convexity of the representation (positivity of the weight-functions).
These important results \cite{Liu_Shu_Zhang_2009a} include explicit expressions of the weight-functions for the particular stencils which were studied, but a general analytical expression
of the optimal weight-functions for the representation of the Lagrange reconstructing polynomial by combination of substencils is not yet available, contrary to the interpolating polynomial case \cite{Carlini_Ferretti_Russo_2005a,
                                                                                                                                                                                                         Liu_Shu_Zhang_2009a,
                                                                                                                                                                                                         Gerolymos_2011a_news}.
The work of Liu~\etal~\cite{Liu_Shu_Zhang_2009a} is based on the reconstruction via primitive approach \cite[pp. 243--244]{Harten_Engquist_Osher_Chakravarthy_1987a},
as developed in \cite{Shu_1998a,
                      Shu_2009a},
where the integral (primitive) $\int_{0}^{x} h(\zeta)\;d\zeta$ of the function $h(x)$, which is reconstructed from its sliding averages $f(x)$ \defref{Def_AELRP_s_RPERR_ss_RP_001}, is used.

Despite the enormous successes of the reconstruction via primitive approach \cite[pp. 243--244]{Harten_Engquist_Osher_Chakravarthy_1987a} in designing and analyzing practical \tsc{weno} schemes \cite{Shu_1998a,
                                                                                                                                                                                                        Shu_2009a,
                                                                                                                                                                                                        Liu_Shu_Zhang_2009a}
the reconstruction via deconvolution approach \cite[244--246]{Harten_Engquist_Osher_Chakravarthy_1987a} is conceptually more straightforward, since it directly uses the unknown function which is reconstructed from
cell-averages, and sometimes simplifies analytical work. In a recent work \cite[Lemma 2.5, p. 272]{Gerolymos_2011a} we have provided the analytical solution
of the deconvolution problem \cite[(3.13b), p. 244]{Harten_Engquist_Osher_Chakravarthy_1987a}, which expresses the unknown function $h(x)$, which is reconstructed from its sliding averages, as a series of the
derivatives of the sliding averages $f^{(n)}(x)$ \cite[(10b), p. 272]{Gerolymos_2011a}. This analytical solution of the deconvolution problem \cite[(3.13b), p. 244]{Harten_Engquist_Osher_Chakravarthy_1987a}
allows the analytical computation of the approximation error of the Lagrange reconstructing polynomial \cite[Proposition 4.7, p. 292]{Gerolymos_2011a},
which would have been necessary to build the general linear system \cite[(13), p. 8489]{Gerolymos_Senechal_Vallet_2009a} for the weight-functions $\sigma_{R_1,M_-,M_+,K_\mathrm{s},k_\mathrm{s}}(\xi)$
in \eqref{Eq_RLRPCS_s_I_001a}.

In the present work\footnote{\label{ff_RLRPCS_s_I_003}
                             In \cite[\S6.1, pp. 297--300]{Gerolymos_2011a}, we had sketched, without giving any proof or analysis,
                             some of the problems which are solved in the present paper. Furthermore, at that time, we had not proven the conjectured convexity \cite{Shu_1998a,
                                                                                                                                                                      Shu_2009a,
                                                                                                                                                                      Liu_Shu_Zhang_2009a},
                             in the neighborhood of $\xi=\tfrac{1}{2}$.
                            }
we use relations and concepts developed in \cite{Gerolymos_2011a}, along with the general recurrence relation for the generation of weight-functions
proven in \cite[Lemma 2.1]{Gerolymos_2011a_news},
to extend the analysis of Liu~\etal~\cite{Liu_Shu_Zhang_2009a}, both by providing general analytical expressions (and existence and uniqueness proofs) of the rational weight-functions,
but also by studying the general case of the subdivision of an arbitrarily biased stencil on a homogeneous grid, $\{i-M_-,\cdots,i+M_+\}$ ($M_\pm\in{\mathbb Z}\;:\;M:=M_-+M_+\geq 2$) containing $M$ intervals,
into $K_\mathrm{s}+1\leq M$ substencils \defref{Def_RLRPCS_s_I_002} of equal length of $M-K_\mathrm{s}$ intervals,
each shifted by 1 cell with respect to its neighbors ($K_\mathrm{s}$ is free to take all possible values $\in\{1,\cdots,M-1\}$).
We also prove several relations concerning the Lagrange reconstructing polynomial.

In \S\ref{RLRPCS_s_RB} we very briefly summarize those results for the Lagrange reconstructing polynomial and its approximation error obtained in \cite{Gerolymos_2011a}
which are necessary in the present work.\footnote{\label{ff_RLRPCS_s_I_004}
                                                  In the present work we also make extensive use of relations concerning reconstruction pairs \cite[Definition 2.1, p. 270]{Gerolymos_2011a},
                                                  and in particular polynomial reconstruction pairs \cite[Theorem 5.1, p. 296]{Gerolymos_2011a}.
                                                 }

In \S\ref{RLRPCS_s_FPLIR} we study the basis polynomials $\alpha_{R_1,M_-,M_+,\ell}(\xi)$ ($\ell\in\{-M_-,\cdots,M_+\}$ which \cite[Proposition 4.5, p. 287]{Gerolymos_2011a}
represent the Lagrange reconstructing polynomial on $\tsc{s}_{i,M_-,M_+}:=\{i-M_-,\cdots,i+M_+\}$, with coordinates the values $f_{i+\ell}:=f(x_i+\ell\Delta x)$ of the cell-averages
of the reconstructed function. These results, which include an analysis of the roots of $\alpha_{R_1,M_-,M_+,\ell}(\xi)$ and relations with the polynomials
$\lambda_{R_1,M_-,M_+,n}(\xi)$ appearing in the expression of the approximation error of the Lagrange reconstructing polynomial \cite[Proposition 4.7, p. 292]{Gerolymos_2011a}
are the starting point for the construction of the Lagrange reconstructing polynomial as a combination of the Lagrange reconstructing polynomials on substencils.

In \S\ref{RLRPCS_s_RCsSs} we use the results of \S\ref{RLRPCS_s_FPLIR} to establish a $1$-level subdivision rule \lemref{Lem_RLRPCS_s_RCsSs_ss_Ks1_001},
by which, applying \cite[Lemma 2.1]{Gerolymos_2011a_news}, we construct \prpref{Prp_RLRPCS_s_RCsSs_ss_Ks_001}
an analytical recursive expression of the weight-functions for a general subdivision of an
arbitrarily biased stencil on a homogeneous grid. We prove the uniqueness of the rational weight-functions \prpref{Prp_RLRPCS_s_RCsSs_ss_Ks_002},
and we show by studying their poles (all of which are real) that it is always possible to define the weight-functions at half-nodes ($\xi=n+\tfrac{1}{2}\;,\;n\in{\mathbb Z}$).
Finally, we prove \thmref{Thm_RLRPCS_s_RCsSs_ss_C_001} the convexity of the representation \eqref{Eq_RLRPCS_s_I_001a} in the neighborhood of $\xi=\tfrac{1}{2}$,
for all subdivisions \defref{Def_RLRPCS_s_I_002} for which all of the substencils contain either point $i$ or point $i+1$ (or both).

%
%
%
%
%
%
%
%
%
\section{Reconstruction background}\label{RLRPCS_s_RB}
%
%
%
%
%
%
%
%
%

In a recent work \cite{Gerolymos_2011a} we have studied the exact and approximate reconstruction of a function $h(x)$.
We have obtained the general analytical solution of the deconvolution of Taylor-series problem \cite[(3.13), pp. 244--254]{Harten_Engquist_Osher_Chakravarthy_1987a},
and used this solution in developing analytical relations for the approximation error of
polynomial reconstruction on an arbitrary stencil in a homogeneous grid \cite{Gerolymos_2011a}. We briefly summarize those results of \cite{Gerolymos_2011a} which are the starting point of
the analysis presented in the present work, and which are necessary for completeness.
\begin{lemma}[{\rm Derivatives of reconstruction pairs}]
\label{Lem_RLRPCS_s_RB_001}
Let $h=R_{(1;\Delta x)}(f)$ be a reconstruction pair \defref{Def_AELRP_s_RPERR_ss_RP_001}, and assume that $f(x)$ and $h(x)$ are of class $C^N[a-\tfrac{1}{2}\Delta x,b+\tfrac{1}{2}\Delta x]$.
Then
\begin{alignat}{6}
h=R_{(1;\Delta x)}(f)\Longrightarrow h^{(n)}=R_{(1;\Delta x)}(f^{n})\qquad\forall n\in\{1,\cdots,N\}
                                                                                                       \label{Eq_Lem_RLRPCS_s_RB_001_001}
\end{alignat}
\end{lemma}
%
%
\begin{proof}
We have by direct integration
\begin{alignat}{6}
\dfrac{1}{\Delta x}\int_{x-\frac{1}{2}\Delta x}^{x+\frac{1}{2}\Delta x}{h'(\zeta)d\zeta}=\dfrac{h(x+\frac{1}{2}\Delta x)-h(x-\frac{1}{2}\Delta x)}{\Delta x}
                                                                                        \stackrel{\text{\cite[(9)]{Gerolymos_2011a}}}{=}f'(x) \quad\forall x\in[a,b]
                                                                                                       \label{Eq_Lem_RLRPCS_s_RB_001_002}
\end{alignat}
from the fundamental property of reconstruction pairs \cite[Lemma 2.2, p. 271]{Gerolymos_2011a}, proving \eqref{Eq_Lem_RLRPCS_s_RB_001_001} for $n=1$, and by induction $\forall n\in\{1,\cdots,N\}$.
\qed
\end{proof}

Expressions for the Lagrange interpolating polynomial and its approximation error are widely available in the literature \cite[pp. 186--189]{Henrici_1964a},
and they are only included in the following to define notation, for completeness, but also to highlight analogies and differences
between Lagrange interpolation and Lagrange reconstruction. In \cite[Propositions 4.5, 4.6, 4.7]{Gerolymos_2011a}
we developed corresponding analytical expressions for the Lagrange reconstructing polynomial, which can be summarized as
%
\begin{proposition}[{\rm Lagrange polynomial interpolation and reconstruction on $\tsc{s}_{i,M_-,M_+}$ \cite{Gerolymos_2011a}}]
\label{Prp_RLRPCS_s_RB_001}
Assume $M_\pm\in{\mathbb Z}:M:=M_-+M_+\geq0$ \eqref{Eq_Def_AELRP_s_EPR_ss_PR_001_001b}.
Let $h=R_{(1;\Delta x)}(f)$ be a reconstruction pair \defref{Def_AELRP_s_RPERR_ss_RP_001}.
Then the Lagrange interpolating ($p_{I,M_-,M_+}$) and reconstructing ($p_{R_1,M_-,M_+}$) polynomials of $f(x)$ \defref{Def_AELRP_s_RPERR_ss_RP_002} on $\tsc{s}_{i,M_-,M_+}$ \defref{Def_AELRP_s_EPR_ss_PR_001} are
\begin{subequations}
                                                                                                       \label{Eq_Prp_RLRPCS_s_RB_001_001}
\begin{alignat}{6}
p_{R_1,M_-,M_+}(x_i+\xi\Delta x;x_i,\Delta x;f)=&\sum_{\ell=-M_-}^{M_+}    \alpha_{R_1,M_-,M_+,\ell}(\xi)\;f(x_i+\ell\Delta x)
                                                                                                       \label{Eq_Prp_AELRP_s_EPR_ss_PR_001_001d}\\
                                               =&h(x_i+\xi\Delta x)+\sum_{n=M+1}^{N_\tsc{tj}}    \mu_{R_1,M_-,M_+,n}(\xi)\;\Delta x^n\;f^{(n)}(x_i)+O(\Delta x^{N_\tsc{tj}+1})
                                                                                                       \label{Eq_Prp_AELRP_s_EPR_ss_AELPR_001_001c}\\
                                               =&h(x_i+\xi\Delta x)+\sum_{n=M+1}^{N_\tsc{tj}}\lambda_{R_1,M_-,M_+,n}(\xi)\;\Delta x^n\;h^{(n)}(x_i+\xi\Delta x)+O(\Delta x^{N_\tsc{tj}+1})
                                                                                                       \label{Eq_Prp_AELRP_s_EPR_ss_AELPR_002_001a}
\end{alignat}
\end{subequations}
\begin{subequations}
                                                                                                       \label{Eq_Prp_RLRPCS_s_RB_001_002}
\begin{alignat}{6}
p_{I  ,M_-,M_+}(x_i+\xi\Delta x;x_i,\Delta x;f)=&\sum_{\ell=-M_-}^{M_+}    \alpha_{I  ,M_-,M_+,\ell}(\xi)\;f(x_i+\ell\Delta x)
                                                                                                       \label{Eq_Prp_AELRP_s_EPR_ss_PR_001_001e}\\
                                               =&f(x_i+\xi\Delta x)+\sum_{n=M+1}^{N_\tsc{tj}}    \mu_{I  ,M_-,M_+,n}(\xi)\;\Delta x^n\;f^{(n)}(x_i)+O(\Delta x^{N_\tsc{tj}+1})
                                                                                                       \label{Eq_Prp_AELRP_s_EPR_ss_AELPR_001_001e}\\
                                               =&f(x_i+\xi\Delta x)+\sum_{n=M+1}^{N_\tsc{tj}}\lambda_{I  ,M_-,M_+,n}(\xi)\;\Delta x^n\;f^{(n)}(x_i+\xi\Delta x)+O(\Delta x^{N_\tsc{tj}+1})
                                                                                                       \label{Eq_Prp_AELRP_s_EPR_ss_AELPR_002_001b}
\end{alignat}
\end{subequations}
and, provided that $h(x)$ is sufficiently smooth $\forall x\in[x_{i-M_-}-\tfrac{1}{2}\Delta x,x_{i+M_+}+\tfrac{1}{2}\Delta x]$,
their approximation errors are defined by
\eqrefsab{Eq_Prp_AELRP_s_EPR_ss_AELPR_001_001c}{Eq_Prp_AELRP_s_EPR_ss_AELPR_001_001e},
or equivalently by \eqrefsab{Eq_Prp_AELRP_s_EPR_ss_AELPR_002_001a}{Eq_Prp_AELRP_s_EPR_ss_AELPR_002_001b}.
The functions
$\alpha_{R_1,M_-,M_+,\ell}(\xi)$, $\mu_{R_1,M_-,M_+,n}(\xi)$,    $\lambda_{R_1,M_-,M_+,n}(\xi)$,
$\alpha_{  I,M_-,M_+,\ell}(\xi)$, $\mu_{  I,M_-,M_+,n}(\xi)$ and $\lambda_{  I,M_-,M_+,n}(\xi)$ are polynomials in
\begin{alignat}{6}
\xi:=\frac{x-x_i}{\Delta x}
                                                                                                       \label{Eq_Prp_AELRP_s_EPR_ss_PR_001_001f}
\end{alignat}
with coefficients depending only on $M_\pm$ and are defined by
\begin{subequations}
                                                                                                       \label{Eq_Prp_RLRPCS_s_RB_001_003}
\begin{alignat}{6}
\alpha_{R_1,M_-,M_+,\ell}(\xi):=&\sum_{m=0}^{M}\left(\sum_{k=0}^{\lfloor\frac{M-m}{2}\rfloor}\frac{\tau_{2k}(m+2k)!}
                                                                                                {m!              }({^{M_+}_{M_-}V}^{-1})_{m+2k+1,\ell+M_-+1}\right)\xi^m
&\qquad;\qquad&-M_-\leq\ell\leq+M_+
                                                                                                       \label{Eq_Prp_AELRP_s_EPR_ss_PR_001_001g}\\
\mu_{R_1,M_-,M_+,s}(\xi):=&\sum_{k=0}^{\lfloor\frac{s}{2}\rfloor}\frac{-\tau_{2k}}
                                                                      {(s-2k)!   }\xi^{s-2k}
                         + \sum_{m=0}^{M}\left(\sum_{k=0}^{\lfloor\frac{M-m}{2}\rfloor}\tau_{2k}\nu_{M_-,M_+,m+2k,s}\frac{(m+2k)!}{s!\;m!}\right)\xi^m
&\qquad;\qquad& s\geq M+1
                                                                                                       \label{Eq_Prp_AELRP_s_EPR_ss_AELPR_001_001f}\\
\lambda_{R_1,M_-,M_+,n}(\xi):=&\sum_{\ell=0}^{n-M-1}\mu_{R_1,M_-,M_+,n-\ell}(\xi)
                                                    \dfrac{(-1)^{\ell+1}}
                                                          {(\ell+1)!    }\left((\xi-\tfrac{1}{2})^{\ell+1}
                                                                              -(\xi+\tfrac{1}{2})^{\ell+1}\right)
&\qquad;\qquad& n\geq M+1
                                                                                                       \label{Eq_Prp_AELRP_s_EPR_ss_AELPR_002_001c}
\end{alignat}
\end{subequations}
\begin{subequations}
                                                                                                       \label{Eq_Prp_RLRPCS_s_RB_001_004}
\begin{alignat}{6}
\alpha_{I  ,M_-,M_+,\ell}(\xi):=&\sum_{m=0}^{M}({^{M_+}_{M_-}V}^{-1})_{m+1,\ell+M_-+1}\;\xi^m
&\qquad;\qquad&-M_-\leq\ell\leq+M_+
                                                                                                       \label{Eq_Prp_AELRP_s_EPR_ss_PR_001_001h}\\
\mu_{I  ,M_-,M_+,s}(\xi):=&\frac{ 1}
                                {s!}\left(-\xi^s+\sum_{m=0}^{M}\nu_{M_-,M_+,m,s}\xi^m\right)
&\qquad;\qquad& s\geq M+1
                                                                                                       \label{Eq_Prp_AELRP_s_EPR_ss_AELPR_001_001g}\\
\lambda_{I  ,M_-,M_+,n}(\xi):=&\sum_{\ell=0}^{n-M-1}\dfrac{(-\xi)^\ell}
                                                          {\ell!      }\mu_{I,M_-,M_+,n-\ell}(\xi)
&\qquad;\qquad& n\geq M+1
                                                                                                       \label{Eq_Prp_AELRP_s_EPR_ss_AELPR_002_001d}
\end{alignat}
\end{subequations}
where $({^{M_+}_{M_-}V}^{-1})_{ij}$ are the elements of the inverse Vandermonde matrix on $\tsc{s}_{i,M_-,M_+}$ \textup{\cite[Definition 4.3, p. 283]{Gerolymos_2011a}},
expressed by \textup{\cite[(43a,43b), pp. 283--284]{Gerolymos_2011a}},\footnote{\label{ff_Prp_RLRPCS_s_RB_001_001}
                                                                                By \cite[Lemma 4.4, pp. 283--284]{Gerolymos_2011a}
                                                                                \begin{alignat}{6}
                                                                                ({^{M_+}_{M_-}V}^{-1})_{ij}=\sum_{n=0}^{M+1-i}(M_-)^n\;\binom{n+i-1}{n}\;({^{M}_{0}V}^{-1})_{i+n,j}\qquad\begin{array}{l}\forall i,j\in\{1,\cdots,M+1\}\\
                                                                                                                                                                                                         M:=M_-+M_+\\\end{array}
                                                                                                                    \notag
                                                                                \end{alignat}
                                                                                \begin{alignat}{6}
                                                                                ({^{M}_{0}V}^{-1})_{ij}=(-1)^{i+j}\sum_{k=1}^{M+1}\frac{     1}
                                                                                                                                       {(k-1)!}\binom{k-1}
                                                                                                                                                     {j-1}\strlngfk{k-1}
                                                                                                                                                                   {i-1}
                                                                                                                                   \qquad\forall i,j\in\{1,\cdots,M+1\}
                                                                                                                    \notag
                                                                                \end{alignat}
                                                                               }
$\nu_{M_-,M_+,m,s}$ are defined by\footnote{\label{ff_Prp_RLRPCS_s_RB_001_002}
                                            By \cite[Lemma 4.4, pp. 283--284]{Gerolymos_2011a}
                                            \begin{alignat}{6}
                                            \nu_{M_-,M_+,m,k}=
                                            \sum_{\ell=-M_-}^{M_+}({^{M_+}_{M_-}V}^{-1})_{m+1,\ell+M_-+1}\;\ell^k=\delta_{mk}\qquad\left.\begin{array}{c} 0 \leq k \leq M\\
                                                                                                                                                          0 \leq m \leq M\\\end{array}\right.
                                                                                                       \notag
                                            \end{alignat}
                                            \begin{alignat}{6}
                                            \sum_{m=0}^{M}\nu_{M_-,M_+,m,k}\;\ell^m=\ell^k                                   \qquad\left.\begin{array}{l} \forall   k\in{\mathbb N}_0      \\
                                                                                                                                                          \forall\ell\in\{-M_-,\cdots,M_+\}\\\end{array}\right.
                                                                                                       \notag
                                            \end{alignat}
                                           }
\begin{alignat}{6}
\nu_{M_-,M_+,m,k}:=\sum_{\ell=-M_-}^{M_+}({^{M_+}_{M_-}V}^{-1})_{m+1,\ell+M_-+1}\;\ell^k
\qquad;\qquad k\in\mathbb{N}_0
                                                                                                       \label{Eq_Lem_AELRP_s_EPR_ss_PR_001_001c}
\end{alignat}
and the numbers $\tau_{n}$ satisfy\footnote{\label{ff_Prp_RLRPCS_s_RB_001_003}
                                            By \cite[Theorem 2.9, pp. 275--276]{Gerolymos_2011a}
                                            the numbers $\tau_n$ can be defined as $\tau_{n}:=\dfrac{1}{n!}g_\tau^{(n)}(0)$
                                            from the generating function $g_\tau(x):=\dfrac{\tfrac{1}{2} x      }
                                                                                           {\sinh{\tfrac{1}{2}x}}$}
\begin{subequations}
                                                                                                       \label{Eq_Prp_RLRPCS_s_RB_001_005}
\begin{alignat}{6}
\tau_0 &=& 1
&\qquad;\qquad&
\tau_{2k} &=& \sum_{s=0}^{k-1}\frac{-\tau_{2s}           }
                                   {2^{2k-2s}\;(2k-2s+1)!}
          &=& \sum_{s=1}^{k}  \frac{-\tau_{2k-2s}  }
                                   {2^{2s}\;(2s+1)!} &\qquad k > 0
                                                                                                       \label{Eq_Lem_AELRP_s_RPERR_ss_D_001_001c}\\
&&&&
\tau_{2n+1} &=& 0&&&\qquad n\geq0
                                                                                                       \label{Eq_Thm_AELRP_s_RPERR_ss_GFtaunRPexp_001_001d}
\end{alignat}
\end{subequations}
\qed
\end{proposition}
\begin{remark}[{\rm Alternative expressions for $\alpha_{I,M_-,M_+,\ell}(\xi)$ \eqref{Eq_Prp_AELRP_s_EPR_ss_PR_001_001h} and $\alpha_{R_1,M_-,M_+,\ell}(\xi)$ \eqref{Eq_Prp_AELRP_s_EPR_ss_PR_001_001g}}]
\label{Rmk_RLRPCS_s_RB_001}
\begin{subequations}
                                                                                                       \label{Eq_Rmk_RLRPCS_s_RB_001_001}
The polynomials $\alpha_{I,M_-,M_+,\ell}(\xi)$ are the fundamental functions \cite[pp. 183--197]{Henrici_1964a}
of Lagrange interpolation on the stencil $\tsc{s}_{i,M_-,M_+}$,
and can also be expressed, upon replacing $x_\ell=x_i+\ell\Delta x$ in \cite[(9.4), p. 184]{Henrici_1964a}, as
\begin{alignat}{6}
\alpha_{I,M_-,M_+,\ell}(\xi) =\dfrac{\displaystyle\prod_{\substack{k=-M_-\\
                                                                   k\neq\ell}}^{M_+}(\xi-k)}
                                    {\displaystyle\prod_{\substack{k=-M_-\\
                                                                   k\neq\ell}}^{M_+}(\ell-k)}
                                                                                                       \label{Eq_Rmk_RLRPCS_s_RB_001_001a}
\end{alignat}
Shu \cite[(2.19), p. 336]{Shu_1998a} has shown, using the reconstruction via primitive approach \cite[pp. 243--244]{Harten_Engquist_Osher_Chakravarthy_1987a}, that the fundamental functions \eqref{Eq_Prp_AELRP_s_EPR_ss_PR_001_001g}
of the Lagrange reconstructing polynomial on $\tsc{s}_{i,M_-,M_+}$, $p_{R_1,M_-,M_+}(x_i+\xi\Delta x;x_i,\Delta x;f)$ \eqref{Eq_Prp_AELRP_s_EPR_ss_PR_001_001d},
can equivalently be expressed as\footnote{\label{ff_Rmk_RLRPCS_s_RB_001_001}
                                         The correspondence of the present indicial notation with the one used by Shu \cite[(2.19), p. 336]{Shu_1998a} is
                                         $r_{\rm Shu}=M_-$, $j_{\rm Shu}=\ell+M_-$, $k_{\rm Shu}=M+1=M_-+M_++1$,
                                         and we use again $x=x_i+\xi\Delta x$ \eqref{Eq_Prp_AELRP_s_EPR_ss_PR_001_001f}.
                                         }
\begin{alignat}{6}
\alpha_{R_1,M_-,M_+,\ell}(\xi)\stackrel{\text{\cite[(2.19)]{Shu_1998a}}}{=}
                                 \sum_{m=\ell+M_-+1}^{M+1}\dfrac{\displaystyle\sum_{\substack{p=0\\p\neq m}}^{M+1}
                                                                              \prod_{\substack{q=0\\q\neq m\\q\neq p}}^{M+1}(\xi-q+\tfrac{1}{2})}
                                                                {\displaystyle\prod_{\substack{p=0\\p\neq m}}^{M+1}(m-p)}
                                                                                                       \label{Eq_Rmk_RLRPCS_s_RB_001_001b}
\end{alignat}
\end{subequations}
\qed
\end{remark}
%
%
\begin{remark}[{\rm Mapping $R_{(1;\Delta x)}$}]
\label{Rmk_RLRPCS_s_RB_002}
By \cite[Theorem 5.1, p. 296]{Gerolymos_2011a} the mapping $R_{(1;\Delta x)}$ is a bijection of the $(M+1)$-dimensional space $\mathbb{R}_M[x]$ of polynomials of degree $\leq M$ onto itself.
This implies that polynomial reconstruction pairs are unique \cite[Lemma 3.1, p. 296]{Gerolymos_2011a}.
Let $p(x)\in \mathbb{R}_M[x]$ be a polynomial of degree $\leq M$.
Then by \cite[Theorem 5.1, p. 296]{Gerolymos_2011a}
\begin{alignat}{6}
\forall p(x)\in \mathbb{R}_M[x]\;\Longrightarrow\;\left\{\begin{array}{c}q(x):=[R_{(1;\Delta x)}(p)](x)\in \mathbb{R}_M[x]\\
                                                             {\rm deg}(q)={\rm deg}(p)\\
                                                             {\rm coeff}[x^{{\rm deg}(p)},p(x)]={\rm coeff}[x^{{\rm deg}(p)},q(x)]\\\end{array}\right.
                                                                                                       \label{Eq_Rmk_RLRPCS_s_RB_002_001}
\end{alignat}
where the property that the linear operator $R_{(1;\Delta x)}:\mathbb{R}_M[x]\longrightarrow \mathbb{R}_M[x]$ conserves the degree of the polynomial follows directly from \cite[Lemma 3.1, p. 277]{Gerolymos_2011a}.
Furthermore, the coefficients of the leading power of $p(x)\in \mathbb{R}_M[x]$ and of $q(x):=[R_{(1;\Delta x)}(p)](x)\in \mathbb{R}_M[x]$
can be shown to be equal, by straightforward application of the
expression \cite[(26f), Lemma 3.1, p. 277]{Gerolymos_2011a}.\footnote{\label{ff_Rmk_RLRPCS_s_RB_002_001}
                                                                      By \cite[Lemma 3.1, p. 277]{Gerolymos_2011a} if ${\rm deg}(p)=M$, and $p(x_i+\xi\;\Delta x)=\sum_{m=0}^{M}c_{p_m}\;\xi^m$ then
                                                                      $q(x_i+\xi\;\Delta x):=[R_{(1;\Delta x)}(p)](x_i+\xi\;\Delta x)=\sum_{m=0}^{M}c_{q_m}\;\xi^m$ and by \cite[(26f), p. 277]{Gerolymos_2011a}
                                                                      \begin{equation} c_{q_M}\stackrel{\text{\cite[(26f)]{Gerolymos_2011a}}}{=}
                                                                                              \dfrac{1}{M!}\sum_{k=0}^{\lfloor\frac{M-M}{2}\rfloor}\tau_{2k}\;c_{p_{M+2k}}\;(M+2k)!
                                                                                              \stackrel{\eqref{Eq_Lem_AELRP_s_RPERR_ss_D_001_001c}}{=}c_{p_M}\notag
                                                                      \end{equation}
                                                                      }
One consequence of these properties is that several relations obtained for the interpolating polynomial have their direct analogues for the reconstructing polynomial and {\em vice-versa}.
\qed
\end{remark}
%
%
%
%
%
%
%
%
%
%
\section{Fundamental polynomials of Lagrange interpolation and reconstruction}\label{RLRPCS_s_FPLIR}
%
%
%
%
%
%
%
%
%

The construction of a recursive formulation (\S\ref{RLRPCS_s_RCsSs}) for the linear weight-functions \eqref{Eq_RLRPCS_s_I_001a}
is based on the representations of the Lagrange reconstructing polynomial \eqref{Eq_Prp_AELRP_s_EPR_ss_PR_001_001d} and of its approximation error \eqref{Eq_Prp_AELRP_s_EPR_ss_AELPR_002_001a}.
It is therefore necessary to gain some insight on the fundamental functions \eqref{Eq_Prp_AELRP_s_EPR_ss_PR_001_001g} of Lagrange reconstruction,
and on the truncation-error polynomials \eqref{Eq_Prp_AELRP_s_EPR_ss_AELPR_002_001c}.

%
%
%
%
%
\subsection{Reconstruction pairs of fundamental polynomials}\label{RLRPCS_s_FPLIR_ss_RPsFPs}
%
%
%
%
%

Each of the fundamental polynomials of Lagrange reconstruction $\alpha_{R_1,M_-,M_+,\ell}(\xi)$ \eqref{Eq_Prp_AELRP_s_EPR_ss_PR_001_001g} is intimately related
to the corresponding fundamental polynomial of Lagrange interpolation $\alpha_{I,M_-,M_+,\ell}(\xi)$ \eqref{Eq_Prp_AELRP_s_EPR_ss_PR_001_001h}, as can be seen
by using \eqrefsab{Eq_Prp_AELRP_s_EPR_ss_PR_001_001e}{Eq_Prp_AELRP_s_EPR_ss_PR_001_001d} in the reconstruction-pair-defining relation \eqref{Eq_Def_AELRP_s_RPERR_ss_RP_001_001a}.
%
\begin{proposition}[{\rm Reconstruction pairs $\alpha_{R_1,M_-,M_+,\ell}=R_{(1;1)}(\alpha_{I,M_-,M_+,\ell})$}]
\label{Prp_RLRPCS_s_FPLIR_ss_RPsFPs_001}
Assume $M_\pm\in{\mathbb Z}:M:=M_-+M_+\geq0$ \eqref{Eq_Def_AELRP_s_EPR_ss_PR_001_001b}.
The polynomial $\alpha_{R_1,M_-,M_+,\ell}(\xi)$ \eqref{Eq_Prp_AELRP_s_EPR_ss_PR_001_001g} appearing in the representation \eqref{Eq_Prp_AELRP_s_EPR_ss_PR_001_001d} of the reconstructing polynomial
on the stencil $\tsc{s}_{i,M_-,M_+}$ \defref{Def_AELRP_s_EPR_ss_PR_001} is the reconstruction pair\footnote{\label{ff_Prp_RLRPCS_s_FPLIR_ss_RPsFPs_001_001}on a unit-spacing grid, $\Delta x=1$}
of the corresponding polynomial $\alpha_{I,M_-,M_+,\ell}(\xi)$~\eqref{Eq_Prp_AELRP_s_EPR_ss_PR_001_001h}
appearing in the representation \eqref{Eq_Prp_AELRP_s_EPR_ss_PR_001_001e} of the interpolating polynomial on the same stencil
\begin{alignat}{6}
\alpha_{R_1,M_-,M_+,\ell}(\xi)=[R_{(1;1)}(\alpha_{I,M_-,M_+,\ell})](\xi)\iff\alpha_{I,M_-,M_+,\ell}(\xi)=\int_{\xi-\frac{1}{2}}^{\xi+\frac{1}{2}}\alpha_{R_1,M_-,M_+,\ell}(\eta)\;d\eta
\quad\left\{\begin{array}{l}\forall\;\ell\in\{-M_-,\cdots,M_+\}\\
                            \forall\xi\in{\mathbb R}\\\end{array}\right.
                                                                                                       \label{Eq_Prp_RLRPCS_s_FPLIR_ss_RPsFPs_001_001}
\end{alignat}
\end{proposition}
%
%
\begin{proof}
\begin{subequations}
                                                                                                       \label{Eq_Prp_RLRPCS_s_FPLIR_ss_RPsFPs_001_002}
By \defrefnp{Def_AELRP_s_RPERR_ss_RP_002} of the reconstructing polynomial, we have, using \eqref{Eq_Def_AELRP_s_RPERR_ss_RP_001_001a}
\begin{alignat}{6}
p_{I,M_-,M_+}(x_i+\xi\Delta x;x_i,\Delta x;f)=\dfrac{1}{\Delta x}\int_{x_i+\xi\Delta x-\frac{1}{2}\Delta x}^{x_i+\xi\Delta x+\frac{1}{2}\Delta x}p_{R_1,M_-,M_+}(\zeta;x_i,\Delta x;f)\;d\zeta
                                                                                                       \label{Eq_Prp_RLRPCS_s_FPLIR_ss_RPsFPs_001_002a}
\end{alignat}
and using the representation \eqref{Eq_Prp_AELRP_s_EPR_ss_PR_001_001d} of the reconstructing polynomial
and the representation \eqref{Eq_Prp_AELRP_s_EPR_ss_PR_001_001e} of the interpolating polynomial,
we readily obtain by \eqref{Eq_Prp_RLRPCS_s_FPLIR_ss_RPsFPs_001_002a}
\begin{alignat}{6}
\sum_{\ell=-M_-}^{M_+}\left(\alpha_{I,M_-,M_+,\ell}(\xi)-\int_{\xi-\frac{1}{2}}^{\xi+\frac{1}{2}}\alpha_{R_1,M_-,M_+,\ell}(\eta)\;d\eta\right)\;f(x_i+\ell\;\Delta x)
\stackrel{\eqrefsabc{Eq_Prp_AELRP_s_EPR_ss_PR_001_001d}
                    {Eq_Prp_AELRP_s_EPR_ss_PR_001_001e}
                    {Eq_Prp_RLRPCS_s_FPLIR_ss_RPsFPs_001_002a}}{=}0\quad\left\{\begin{array}{l}\forall \xi\in{\mathbb R}\\
                                                                                                                                                     \forall x_i\in{\mathbb R}\\
                                                                                                                                                     \forall \Delta x\in{\mathbb R}_{>0}\\
                                                                                                                                                     \forall f:{\mathbb R}\longrightarrow{\mathbb R}\\\end{array}\right.
                                                                                                       \label{Eq_Prp_RLRPCS_s_FPLIR_ss_RPsFPs_001_002b}
\end{alignat}
Since \eqref{Eq_Prp_RLRPCS_s_FPLIR_ss_RPsFPs_001_002b} is valid $\forall f:{\mathbb R}\longrightarrow{\mathbb R}$, it proves,
using the definitions \eqrefsab{Eq_Def_AELRP_s_RPERR_ss_RP_001_001a}{Eq_Def_AELRP_s_RPERR_ss_RP_001_001b},
\eqref{Eq_Prp_RLRPCS_s_FPLIR_ss_RPsFPs_001_001}.
\end{subequations}
\qed
\end{proof}
%
%
\begin{lemma}[{\rm $\alpha_{R_1,M_-,M_+,\ell}(\xi)\neq 0_{\mathbb{R}_M[\xi]}\neq\alpha_{I,M_-,M_+,\ell}(\xi)$}]
\label{Lem_RLRPCS_s_FPLIR_ss_RPsFPs_001}
\begin{subequations}
                                                                                                       \label{Eq_Lem_RLRPCS_s_FPLIR_ss_RPsFPs_001_001}
Assume $M_\pm\in{\mathbb Z}:M:=M_-+M_+\geq0$ \eqref{Eq_Def_AELRP_s_EPR_ss_PR_001_001b} and let
\begin{alignat}{6}
0_{\mathbb{R}_M[\xi]}(\xi):=\sum_{m=0}^{M}0\;\xi^m=0
                                                                                                       \label{Eq_Lem_RLRPCS_s_FPLIR_ss_RPsFPs_001_001a}
\end{alignat}
denote the 0-element of the space $\mathbb{R}_M[\xi]$ of all polynomials of degree $\leq M$.
None of the polynomials \eqrefsab{Eq_Prp_AELRP_s_EPR_ss_PR_001_001g}{Eq_Prp_AELRP_s_EPR_ss_PR_001_001h}, of degree $M$ in $\xi$
is identically 0
\begin{alignat}{6}
\alpha_{R_1,M_-,M_+,\ell}(\xi)\neq0_{\mathbb{R}_M[\xi]}(\xi)\neq \alpha_{I,M_-,M_+,\ell}(\xi)\qquad\forall\;\ell\in\{-M_-,\cdots,M_+\}
                                                                                                       \label{Eq_Lem_RLRPCS_s_FPLIR_ss_RPsFPs_001_001b}
\end{alignat}
Furthermore
\begin{alignat}{6}
{\rm deg}[\alpha_{R_1,M_-,M_+,\ell}(\xi)]={\rm deg}[\alpha_{I,M_-,M_+,\ell}(\xi)]=M
                                                                                                       \label{Eq_Lem_RLRPCS_s_FPLIR_ss_RPsFPs_001_001c}
\end{alignat}
\begin{alignat}{6}
{\rm coeff}[\xi^M,\alpha_{R_1,M_-,M_+,\ell}(\xi)]=
{\rm coeff}[\xi^M,\alpha_{I  ,M_-,M_+,\ell}(\xi)]=(-1)^{\ell+M_+}\frac{ 1}
                                                                      {M!}\binom{M       }
                                                                                {\ell+M_-}\neq0\qquad\;\forall\;\ell\in\{-M_-,\cdots,M_+\}
                                                                                                       \label{Eq_Lem_RLRPCS_s_FPLIR_ss_RPsFPs_001_001d}
\end{alignat}
\end{subequations}
\end{lemma}
%
%
\begin{proof}
It is well known, and also obvious from the expression \eqref{Eq_Rmk_RLRPCS_s_RB_001_001a}, that the fundamental polynomials $\alpha_{I,M_-,M_+,\ell}(\xi)$ of Lagrange interpolation on a stencil of $M+1$ equidistant points
are $\neq0_{\mathbb{R}_M[\xi]}(\xi)$. Since by \eqref{Eq_Prp_RLRPCS_s_FPLIR_ss_RPsFPs_001_001} $\alpha_{I,M_-,M_+,\ell}(\xi)\neq0_{\mathbb{R}_M[\xi]}(\xi)$ is equal to the definite integral of $\alpha_{R_1,M_-,M_+,\ell}(\xi)$
in the interval $[\xi-\tfrac{1}{2},\xi+\tfrac{1}{2}]$, it follows that $\alpha_{R_1,M_-,M_+,\ell}(\xi)\neq0_{\mathbb{R}_M[\xi]}(\xi)$.
By \eqrefsab{Eq_Prp_AELRP_s_EPR_ss_PR_001_001g}{Eq_Prp_AELRP_s_EPR_ss_PR_001_001h} we have \eqref{Eq_Lem_RLRPCS_s_FPLIR_ss_RPsFPs_001_001c}.
It is easy to show by direct computation\footnote{\label{ff_Lem_RLRPCS_s_FPLIR_ss_RPsFPs_001_001} By \prprefnp{Prp_RLRPCS_s_RB_001}
\begin{alignat}{6}
  {\rm coeff}[\xi^M,\alpha_{R_1,M_-,M_+,\ell}(\xi)]
\stackrel{\eqref{Eq_Prp_AELRP_s_EPR_ss_PR_001_001g}}{=}&
  \sum_{k=0}^{\lfloor\frac{M-M}{2}\rfloor}\frac{\tau_{2k}\;(M+2k)!}
                                               {M!                }({^{M_+}_{M_-}V}^{-1})_{M+2k+1,\ell+M_-+1}
=                                         \frac{\tau_0   \;     M!}
                                               {M!                }({^{M_+}_{M_-}V}^{-1})_{M+1,\ell+M_-+1}
\stackrel{\eqref{Eq_Lem_AELRP_s_RPERR_ss_D_001_001c}}{=}({^{M_+}_{M_-}V}^{-1})_{M+1,\ell+M_-+1}
                                                                                                       \notag\\
\stackrel{\text{\cite[(43a)]{Gerolymos_2011a}}}{=}&\sum_{n=0}^{M+1-M-1}(M_-)^n\;\binom{n+M+1-1}{n}\;({^{M}_{0}V}^{-1})_{M+1+n,\ell+M_-+1}
                                                =  (M_-)^0\;\binom{M}{0}\;({^{M}_{0}V}^{-1})_{M+1,\ell+M_-+1}
                                                                                                       \notag\\
                                                = &                       ({^{M}_{0}V}^{-1})_{M+1,\ell+M_-+1}
\stackrel{\text{\cite[(43b)]{Gerolymos_2011a}}}{=} (-1)^{M+1+\ell+M_-+1}\sum_{k=1}^{M+1}\frac{     1}
                                                                                             {(k-1)!}\binom{k-1         }
                                                                                             {\ell+M_-+1-1}\strlngfk{k-1  }
                                                                                                                    {M+1-1}
                                                                                                       \notag\\
                                                = &(-1)^{M+\ell+M_-}\frac{ 1}
                                                                         {M!}\binom{M       }
                                                                                   {\ell+M_-}\strlngfk{M}
                                                                                                      {M}
                                                =  (-1)^{M+\ell+M_-}\frac{ 1}
                                                                         {M!}\binom{M       }
                                                                                   {\ell+M_-}
                                                                                                       \notag\\
                          \stackrel{M:=M_-+M_+}{=}& (-1)^{M_++\ell+2M_-}\frac{ 1}
                                                                             {M!}\binom{M       }
                                                                                       {\ell+M_-}
                                                =  (-1)^{\ell+M_+}\frac{ 1}
                                                                       {M!}\binom{M       }
                                                                                 {\ell+M_-}\neq0\qquad\;\forall\;\ell\in\{-M_-,\cdots,M_+\}\stackrel{\eqref{Eq_Def_AELRP_s_EPR_ss_PR_001_001b}}{\Longrightarrow}\;\ell+M_-\leq M
                                                                                                       \notag
\end{alignat}
where we used the expressions \cite[(43a,43b), pp. 283--284]{Gerolymos_2011a} for the elements of the inverse of the Vandermonde matrix \ffpref{ff_Prp_RLRPCS_s_RB_001_001},
and well known properties of the unsigned Stirling numbers of the first kind \cite[Tab. 264, p. 264]{Graham_Knuth_Patashnik_1994a}, $\displaystyle m<n\neq0\Longrightarrow\strlngfk{m}{n}=0\;\forall\;m,n\in{\mathbb N}$ and
                                                                                                                                    $\displaystyle \strlngfk{n}{n}=1\;\forall\;n\in{\mathbb N}_0$.
                                                       }
that ${\rm coeff}[\xi^M,\alpha_{R_1,M_-,M_+,\ell}(\xi)]$ is given by \eqref{Eq_Lem_RLRPCS_s_FPLIR_ss_RPsFPs_001_001d}.
By \eqref{Eq_Rmk_RLRPCS_s_RB_002_001} this is also the coefficient of $\xi^M$ of the polynomial $\alpha_{I,M_-,M_+,\ell}(\xi)$ \eqref{Eq_Prp_AELRP_s_EPR_ss_PR_001_001h},
which \prpref{Prp_RLRPCS_s_FPLIR_ss_RPsFPs_001} is the reconstruction pair of $\alpha_{R_1,M_-,M_+,\ell}(\xi)$, proving \eqref{Eq_Lem_RLRPCS_s_FPLIR_ss_RPsFPs_001_001d}.
\qed
\end{proof}
%

%
\begin{proposition}[{\rm Basis $\left\{\alpha_{R_1,M_-,M_+,\ell}(\xi),\ell\in\{-M_-,\cdots,M_+\}\right\}$}]
\label{Prp_RLRPCS_s_FPLIR_ss_RPsFPs_002}
Assume $M_\pm\in{\mathbb Z}:M:=M_-+M_+\geq0$ \eqref{Eq_Def_AELRP_s_EPR_ss_PR_001_001b}. The $(M+1)$ unique polynomials $\alpha_{R_1,M_-,M_+,\ell}(\xi)\neq0_{\mathbb{R}_M[\xi]}$ \eqref{Eq_Prp_AELRP_s_EPR_ss_PR_001_001g},
$\ell\in\{-M_-,\cdots,M_+\}$,
constitute a basis of the $(M+1)$-dimensional space $\mathbb{R}_M[\xi]$ of all polynomials of degree $\leq M$.
\end{proposition}
%
%
\begin{proof}
\begin{subequations}
                                                                                                       \label{Eq_Prp_RLRPCS_s_FPLIR_ss_RPsFPs_002_002}
This can be proved either by \prprefnp{Prp_RLRPCS_s_FPLIR_ss_RPsFPs_001} or directly.

\underline{Proof by \prprefnp{Prp_RLRPCS_s_FPLIR_ss_RPsFPs_001}:}
It is a well-known fact \cite{Henrici_1964a}, and also obvious from \eqref{Eq_Rmk_RLRPCS_s_RB_001_001a}, that the $(M+1)$ unique polynomials $\left\{\alpha_{I,M_-,M_+,\ell}(\xi),\ell\in\{-M_-,\cdots,M_+\}\right\}$
are linearly independent and span the $(M+1)$-dimensional space $\mathbb{R}_M[\xi]$ of all polynomials of degree $\leq M$ in $\xi$. By \lemrefnp{Prp_RLRPCS_s_FPLIR_ss_RPsFPs_001} each
polynomial $\alpha_{R_1,M_-,M_+,\ell}=R_{(1;\Delta x)}(\alpha_{I,M_-,M_+,\ell})$ \eqref{Eq_Prp_RLRPCS_s_FPLIR_ss_RPsFPs_001_001}, and by \cite[Theorem 5.1, p. 296]{Gerolymos_2011a} the
mapping $R_{(1;\Delta x)}:\mathbb{R}_M[\xi]\to\mathbb{R}_M[\xi]$ is a bijection. Hence the image of $\left\{\alpha_{I,M_-,M_+,\ell}(\xi),\ell\in\{-M_-,\cdots,M_+\}\right\}$,
$\left\{\alpha_{R_1,M_-,M_+,\ell}(\xi),\ell\in\{-M_-,\cdots,M_+\}\right\}$ is also a basis of $\mathbb{R}_M[\xi]$.

\underline{Direct proof:}
Existence of the polynomials $\alpha_{R_1,M_-,M_+,\ell}(\xi)$ satisfying \eqref{Eq_Prp_AELRP_s_EPR_ss_PR_001_001d} was proved by construction
\cite[Proposition 4.5, p. 287]{Gerolymos_2011a} yielding \eqref{Eq_Prp_AELRP_s_EPR_ss_PR_001_001g}.
Recall that by \cite[Theorem 5.1, p. 296]{Gerolymos_2011a} the linear operator $R_{(1;\Delta x)}$ \defref{Def_AELRP_s_RPERR_ss_RP_001} is a bijection of the vector
space $\mathbb{R}_M[x]$ of all polynomials of degree $\leq M$ onto itself.
Obviously, by \cite[Lemma 3.1, p. 277]{Gerolymos_2011a} the same properties apply to the inverse operator $R^{-1}_{(1;\Delta x)}$.
Since, $\forall p(x)\in \mathbb{R}_M[x]\Longrightarrow p^{(s)}(x)=0\;\forall s\geq M+1$,
the reconstructing polynomial \defref{Def_AELRP_s_RPERR_ss_RP_002} of $p(x)$ on the stencil $\tsc{s}_{i,M_-,M_+}:=\{i-M_-,\cdots,i+M_+\}$ \defref{Def_AELRP_s_EPR_ss_PR_001}
is exactly equal to the reconstruction pair of $p(x)$ \defref{Def_AELRP_s_RPERR_ss_RP_002}, $q(x):=[R_{(1;\Delta x)}(p)](x)$, by \eqref{Eq_Prp_AELRP_s_EPR_ss_AELPR_002_001a}.
By \eqref{Eq_Prp_AELRP_s_EPR_ss_PR_001_001d} we have
\begin{alignat}{6}
q(x) \stackrel{\eqrefsab{Eq_Prp_AELRP_s_EPR_ss_PR_001_001d}
                        {Eq_Prp_AELRP_s_EPR_ss_AELPR_002_001a}}{=} \sum_{\ell=-M_-}^{M_+}    \alpha_{R_1,M_-,M_+,\ell}\left(\dfrac{x-x_i}{\Delta x}\right)\;[R^{-1}_{(1;\Delta x)}(q)](x_i+\ell\;\Delta x)
                                                         &\qquad\qquad\left\{\begin{array}{l}\forall q(x)\in \mathbb{R}_M[x]\\
                                                                                             \forall x  \in{\mathbb R}\\
                                                                                             \forall x_i\in{\mathbb R}\\
                                                                                             \forall \Delta x\in{\mathbb R}_{>0}\\\end{array}\right.
                                                                                                       \label{Eq_Prp_RLRPCS_s_FPLIR_ss_RPsFPs_002_002a}
\end{alignat}
Since \eqref{Eq_Prp_RLRPCS_s_FPLIR_ss_RPsFPs_002_002a} holds $\forall x_i\in{\mathbb R}$ and $\forall \Delta x\in{\mathbb R}_{>0}$
we may set $x_i=0$ and $\Delta x=1$ in \eqref{Eq_Prp_RLRPCS_s_FPLIR_ss_RPsFPs_002_002a} to obtain
\begin{alignat}{6}
q(x) = \sum_{\ell=-M_-}^{M_+}    \alpha_{R_1,M_-,M_+,\ell}(x)\;[R^{-1}_{(1;1)}(q)](    \ell          )
                                                         &\qquad\qquad\left\{\begin{array}{l}\forall q(x)\in \mathbb{R}_M[x]\\
                                                                                             \forall x  \in{\mathbb R}\\\end{array}\right.
                                                                                                       \label{Eq_Prp_RLRPCS_s_FPLIR_ss_RPsFPs_002_002b}
\end{alignat}
By \eqref{Eq_Prp_RLRPCS_s_FPLIR_ss_RPsFPs_002_002b}, the $M+1$ polynomials $\alpha_{R_1,M_-,M_+,\ell}(x), \ell\in\{-M_-,\cdots,M_+\}$ span $\mathbb{R}_M[x]$, and
since ${\rm dim}(\mathbb{R}_M[x])=M+1$ they form a basis of $\mathbb{R}_M[x]$.
They are therefore linearly independent \cite{Roman_2005a}, and as a consequence $\neq0_{\mathbb{R}_M[\xi]}$ \eqref{Eq_Lem_RLRPCS_s_FPLIR_ss_RPsFPs_001_001b},
a fact already proven in \lemrefnp{Lem_RLRPCS_s_FPLIR_ss_RPsFPs_001}.
\end{subequations}
\qed
\end{proof}
%

%
%
%
%
%
\subsection{Roots of fundamental polynomials}\label{RLRPCS_s_FPLIR_ss_RFPs}
%
%
%
%
%

Because of \eqref{Eq_Prp_RLRPCS_s_FPLIR_ss_RPsFPs_001_001} for every value returned by the polynomial $\alpha_{I,M_-,M_+,\ell}(\xi_I)$ at point $\xi_I\in\mathbb{R}$,
there exists a nearby point $\xi_{R_1}\in\mathbb{R}$ such that $\alpha_{R_1,M_-,M_+,\ell}(\xi_{R_1})=\alpha_{I,M_-,M_+,\ell}(\xi_I)$, the distance between the 2 points being
$\abs{\xi_{R_1}-\xi_I}<\tfrac{1}{2}$. This can be formalized as
%
\begin{lemma}[{\rm ${\alpha_{I,M_-,M_+,\ell}\Big([\xi_1,\xi_2]\Big)\subseteq\alpha_{R_1,M_-,M_+,\ell}\Big((\xi_1-\tfrac{1}{2},\xi_2+\tfrac{1}{2})\Big)\;\forall\xi_1,\xi_2\in\mathbb{R}\;:\;\xi_1\leq\xi_2}$}]
\label{Lem_RLRPCS_s_FPLIR_ss_RFPs_001}
\begin{subequations}
                                                                                                       \label{Eq_Lem_RLRPCS_s_FPLIR_ss_RFPs_001_001}
Assume $M_\pm\in{\mathbb Z}:M:=M_-+M_+\geq0$ \eqref{Eq_Def_AELRP_s_EPR_ss_PR_001_001b}. Then
\begin{equation}
\forall\ell\in\{-M_-,\cdots,M_+\}\qquad
\forall\xi_I\in\mathbb{R}\quad\exists\;\xi_{R_1}\in(\xi_I-\tfrac{1}{2},\xi_I+\tfrac{1}{2})\subset\mathbb{R}\,:\,
                                                                     \alpha_{R_1,M_-,M_+,\ell}(\xi_{R_1})=\alpha_{I,M_-,M_+,\ell}(\xi_I)
                                                                                                       \label{Eq_Lem_RLRPCS_s_FPLIR_ss_RFPs_001_001a}
\end{equation}
where $\alpha_{R_1,M_-,M_+,\ell}(\xi)$ \eqref{Eq_Prp_AELRP_s_EPR_ss_PR_001_001g} 
and $\alpha_{I,M_-,M_+,\ell}(\xi)$ \eqref{Eq_Prp_AELRP_s_EPR_ss_PR_001_001h} are the fundamental polynomials of Lagrange reconstruction and interpolation,
respectively \prpref{Prp_RLRPCS_s_RB_001}, implying that
\begin{equation}
\forall\ell\in\{-M_-,\cdots,M_+\}\qquad
\alpha_{I,M_-,M_+,\ell}\Big([\xi_1,\xi_2]\Big)\subseteq\alpha_{R_1,M_-,M_+,\ell}\Big((\xi_1-\tfrac{1}{2},\xi_2+\tfrac{1}{2})\Big)\;\forall\xi_1,\xi_2\in\mathbb{R}\;:\;\xi_1\leq\xi_2
                                                                                                       \label{Eq_Lem_RLRPCS_s_FPLIR_ss_RFPs_001_001b}
\end{equation}
\end{subequations}
\end{lemma}
%
%
\begin{proof}
The proof follows immediately from \prprefnp{Prp_RLRPCS_s_FPLIR_ss_RPsFPs_001}. By \eqref{Eq_Prp_RLRPCS_s_FPLIR_ss_RPsFPs_001_001}
\begin{alignat}{6}
\forall\xi_I\in\mathbb{R}\qquad
\alpha_{I,M_-,M_+,\ell}(\xi_I)=\int_{\xi_I-\frac{1}{2}}^{\xi_I+\frac{1}{2}}\alpha_{R_1,M_-,M_+,\ell}(\eta)\;d\eta
\qquad\forall\ell\in\{-M_-,\cdots,M_+\}
                                                                                                       \label{Eq_Lem_RLRPCS_s_FPLIR_ss_RFPs_001_002}
\end{alignat}
Using the mean value theorem for the definite integral \cite[p. 352]{Zorich_2004a} in \eqref{Eq_Lem_RLRPCS_s_FPLIR_ss_RFPs_001_002} yields \eqref{Eq_Lem_RLRPCS_s_FPLIR_ss_RFPs_001_001a},
from which \eqref{Eq_Lem_RLRPCS_s_FPLIR_ss_RFPs_001_001b} is easily proved by contradiction.
\qed
\end{proof}
%

The fundamental polynomials of the Lagrange interpolating polynomial $\alpha_{I,M_-,M_+,\ell}(\xi)$ \eqref{Eq_Prp_AELRP_s_EPR_ss_PR_001_001h} are polynomials of degree $M$ in $\xi$ \eqref{Eq_Prp_AELRP_s_EPR_ss_PR_001_001h},
and it is well known \cite{Henrici_1964a}
and obvious from their expression \eqref{Eq_Rmk_RLRPCS_s_RB_001_001a} that their $M$ roots are the integer nodes $\{-M_-,\cdots,M_+\}\setminus\{\ell\}$
\begin{subequations}
                                                                                                       \label{Eq_RLRPCS_s_FPLIR_ss_RFPs_001}
\begin{alignat}{6}
\alpha_{I,M_-,M_+,\ell}(n)=&0     &\qquad\forall n\in\{-M_-,\cdots,M_+\}\setminus\{\ell\}                                 &\qquad\forall\ell\in\{-M_-,\cdots,M_+\}
                                                                                                       \label{Eq_RLRPCS_s_FPLIR_ss_RFPs_001a}\\
\alpha_{I,M_-,M_+,\ell}(\xi)\neq&0&\qquad\forall \xi\in\mathbb{R}\setminus\Big\{\{-M_-,\cdots,M_+\}\setminus\{\ell\}\Big\}&\qquad\forall\ell\in\{-M_-,\cdots,M_+\}
                                                                                                       \label{Eq_RLRPCS_s_FPLIR_ss_RFPs_001b}
\end{alignat}
\end{subequations}
The fundamental polynomials of the Lagrange reconstructing polynomial $\alpha_{R_1,M_-,M_+,\ell}(\xi)$ \eqrefsab{Eq_Prp_AELRP_s_EPR_ss_PR_001_001g}{Eq_Rmk_RLRPCS_s_RB_001_001b}
are also polynomials of degree $M$ in $\xi$ \eqref{Eq_Prp_AELRP_s_EPR_ss_PR_001_001g},
but the expressions \eqrefsab{Eq_Prp_AELRP_s_EPR_ss_PR_001_001g}{Eq_Rmk_RLRPCS_s_RB_001_001b} are too complicated to directly give information about their roots.
It is nonetheless easy, using \lemrefnp{Lem_RLRPCS_s_FPLIR_ss_RFPs_001}, to show that
%
\begin{proposition}[{\rm Roots of $\alpha_{R_1,M_-,M_+,\ell}(\xi)$ \eqref{Eq_Prp_AELRP_s_EPR_ss_PR_001_001g}}]
\label{Prp_RLRPCS_s_FPLIR_ss_RFPs_001}
Assume $M_\pm\in{\mathbb Z}:M:=M_-+M_+\geq0$ \eqref{Eq_Def_AELRP_s_EPR_ss_PR_001_001b}.
The $M$ roots of the degree $M$ in $\xi$ polynomials $\alpha_{R_1,M_-,M_+,\ell}(\xi)$ \eqref{Eq_Prp_AELRP_s_EPR_ss_PR_001_001g} are all distinct and real, and there is exactly 1 root
in each open interval $(n-\tfrac{1}{2},n+\tfrac{1}{2})$ $\forall\;n\in\{-M_-,\cdots,M_+\}\setminus\{\ell\}$, {\em ie}
\begin{subequations}
                                                                                                       \label{Eq_Prp_RLRPCS_s_FPLIR_ss_RFPs_001_001}
\begin{equation}
\left.\begin{array}{l}\forall   n\in\{-M_-,\cdots,M_+\}\setminus\{\ell\}\\
                      \forall\ell\in\{-M_-,\cdots,M_+\}\\\end{array}\right\}\exists!\;\xi_{R_1,M_-,M_+,\ell,n}\in(n-\tfrac{1}{2},n+\tfrac{1}{2})\subset\mathbb{R}\,:\,\alpha_{R_1,M_-,M_+,\ell}(\xi_{R_1,M_-,M_+,\ell,n})=0
                                                                                                       \label{Eq_Prp_RLRPCS_s_FPLIR_ss_RFPs_001_001a}
\end{equation}
\begin{equation}
\alpha_{R_1,M_-,M_+,\ell}(\xi)\neq0\qquad\forall\xi\in\mathbb{R}\setminus\bigg\{\xi_{R_1,M_-,M_+,\ell,n};\;n\in\{-M_-,\cdots,M_+\}\setminus\{\ell\}\bigg\}\qquad\forall\ell\in\{-M_-,\cdots,M_+\}
                                                                                                       \label{Eq_Prp_RLRPCS_s_FPLIR_ss_RFPs_001_001b}
\end{equation}
\end{subequations}
\end{proposition}
%
%
\begin{proof}
The proof follows immediately from \lemrefnp{Lem_RLRPCS_s_FPLIR_ss_RFPs_001}, by writing \eqref{Eq_Lem_RLRPCS_s_FPLIR_ss_RFPs_001_001a}
at each of the $M$ roots \eqref{Eq_RLRPCS_s_FPLIR_ss_RFPs_001a} of $\alpha_{I,M_-,M_+,\ell}(\xi)$.
By \prprefnp{Lem_RLRPCS_s_FPLIR_ss_RPsFPs_001} the polynomial $\alpha_{R_1,M_-,M_+,\ell}(\xi)\neq0_{\mathbb{R}_M[\xi]}(\xi)$.
Furthermore ${\rm deg}[\alpha_{R_1,M_-,M_+,\ell}(\xi)]=M$ \eqref{Eq_Lem_RLRPCS_s_FPLIR_ss_RPsFPs_001_001c},
and since there are exactly $M:=M_-+M_+$ elements in $\bigg\{\{-M_-,\cdots,M_+\}\setminus\{\ell\}\bigg\}$,
the roots \eqref{Eq_Prp_RLRPCS_s_FPLIR_ss_RFPs_001_001a} are,
by the fundamental theorem of algebra and its corollaries \cite[pp. 282--289]{Zorich_2004a},
the only roots of $\alpha_{R_1,M_-,M_+,\ell}(\xi)$, which proves \eqref{Eq_Prp_RLRPCS_s_FPLIR_ss_RFPs_001_001b},
and uniqueness ($\exists!$) in \eqref{Eq_Prp_RLRPCS_s_FPLIR_ss_RFPs_001_001a}, by contradiction.
\qed
\end{proof}
%
%
\begin{remark}[{\rm Extrema of $\alpha_{R_1,M_-,M_+,\ell}(\xi)$ \eqref{Eq_Prp_AELRP_s_EPR_ss_PR_001_001g}}]
\label{Rmk_RLRPCS_s_FPLIR_ss_RFPs_001}
It is straightforward to show that each fundamental polynomial of Lagrange reconstruction $\alpha_{R_1,M_-,M_+,\ell}(\xi)$ \eqref{Eq_Prp_AELRP_s_EPR_ss_PR_001_001g} has
$M-1$ extrema, where $\alpha'_{R_1,M_-,M_+,\ell}(\xi)=0$, one in each interval between 2 consecutive roots $\xi_{R_1,M_-,M_+,\ell,n}$ \eqref{Eq_Prp_RLRPCS_s_FPLIR_ss_RFPs_001_001a}.
Indeed, for any nonzero polynomial $p(\xi)\in\mathbb{R}_M[\xi]$ with $M$ distinct real roots we know, by Rolle's theorem \cite[pp. 215--216]{Zorich_2004a},
that there is a point where $p'(\xi)=0$ in each of the $M-1$ intervals between 2 consecutive distinct real roots, these $M-1$ points being exactly
the $M-1$ roots of $p'(\xi)\in\mathbb{R}_{M-1}[\xi]$. Both $\alpha_{R_1,M_-,M_+,\ell}(\xi)\in\mathbb{R}_M[\xi]$ \eqref{Eq_Prp_AELRP_s_EPR_ss_PR_001_001g}, by \prprefnp{Prp_RLRPCS_s_FPLIR_ss_RFPs_001},
and $\alpha_{I  ,M_-,M_+,\ell}(\xi)\in\mathbb{R}_M[\xi]$ \eqref{Eq_Prp_AELRP_s_EPR_ss_PR_001_001h}, by \eqref{Eq_RLRPCS_s_FPLIR_ss_RFPs_001}, have $M$ real distinct roots.
Therefore, $\alpha'_{R_1,M_-,M_+,\ell}(\xi)\in\mathbb{R}_{M-1}[\xi]$ and $\alpha'_{I,M_-,M_+,\ell}(\xi)\in\mathbb{R}_{M-1}[\xi]$ have $M-1$ real distinct roots, corresponding to the $M-1$ extrema of
$\alpha_{R_1,M_-,M_+,\ell}(\xi)$ and $\alpha_{I,M_-,M_+,\ell}(\xi)$, respectively. Since \prpref{Prp_RLRPCS_s_FPLIR_ss_RPsFPs_001} $\alpha_{R_1,M_-,M_+,\ell}(\xi)=[R_{(1;1)}(\alpha_{I  ,M_-,M_+,\ell})](\xi)$,
by \lemrefnp{Lem_RLRPCS_s_RB_001}, $\alpha'_{R_1,M_-,M_+,\ell}(\xi)=[R_{(1;1)}(\alpha'_{I  ,M_-,M_+,\ell})](\xi)$, so that their corresponding $M-1$ distinct real roots,
which are also the corresponding extrema of $\alpha_{R_1,M_-,M_+,\ell}(\xi)=[R_{(1;1)}(\alpha_{I  ,M_-,M_+,\ell})](\xi)$, are distant by $<\tfrac{1}{2}$ \lemref{Lem_RLRPCS_s_FPLIR_ss_RFPs_001}.
\qed
\end{remark}
%
%
\begin{proposition}[{\rm Factorization of $\alpha_{I,M_-,M_+,\ell}(\xi)$ \eqref{Eq_Prp_AELRP_s_EPR_ss_PR_001_001h} and $\alpha_{R_1,M_-,M_+,\ell}(\xi)$ \eqref{Eq_Prp_AELRP_s_EPR_ss_PR_001_001g}}]
\label{Prp_RLRPCS_s_FPLIR_ss_RFPs_002}
Assume $M_\pm\in{\mathbb Z}:M:=M_-+M_+\geq0$, and $\ell\in\{-M_-,\cdots,M_+\}$.
Then the fundamental polynomials of Lagrange reconstruction $\alpha_{R_1,M_-,M_+,\ell}(\xi)$ \eqref{Eq_Prp_AELRP_s_EPR_ss_PR_001_001g}
and interpolation $\alpha_{I,M_-,M_+,\ell}(\xi)$ \eqref{Eq_Prp_AELRP_s_EPR_ss_PR_001_001h} on the stencil $\tsc{s}_{i,M_-,M_+}$ \defref{Def_AELRP_s_EPR_ss_PR_001}
can be factorized as
\begin{subequations}
                                                                                                       \label{Eq_Prp_RLRPCS_s_FPLIR_ss_RFPs_002_001}
\begin{alignat}{6}
\alpha_{R_1,M_-,M_+,\ell}(\xi)=&(-1)^{\ell+M_+}\frac{ 1}
                                                    {M!}\binom{M       }
                                                              {\ell+M_-}\prod_{\substack{n=-M_-\\
                                                                                         n\neq\ell}}^{M_+}(\xi-\xi_{R_1,M_-,M_+,\ell,n})
                                                                                                       \label{Eq_Prp_RLRPCS_s_FPLIR_ss_RFPs_002_001a}\\
\alpha_{I,M_-,M_+,\ell}(\xi)  =&(-1)^{\ell+M_+}\frac{ 1}
                                                    {M!}\binom{M       }
                                                              {\ell+M_-}\prod_{\substack{n=-M_-\\
                                                                                         n\neq\ell}}^{M_+}(\xi-n)
                                                                                                       \label{Eq_Prp_RLRPCS_s_FPLIR_ss_RFPs_002_001b}
\end{alignat}
where $\xi_{R_1,M_-,M_+,\ell,n}\in\mathbb{R}$ \textup{(}$n\in\{-M_-,\cdots,M_+\}\setminus\{\ell\}$\textup{)} are the $M$ real and distinct roots of $\alpha_{R_1,M_-,M_+,\ell}(\xi)$ \prpref{Prp_RLRPCS_s_FPLIR_ss_RFPs_001}.
\end{subequations}
\end{proposition}
%
%
\begin{proof}
Every polynomial $p(x)\in\mathbb{R}[x]$ can be factorized as $p(x)={\rm coeff}[x^{{\rm deg}(p)},p(x)]\prod_{n=1}^{{\rm deg}(p)}(x-x_{p_n})$,
where $x_{p_n}\in\mathbb{C}\;(n\in\{1,\cdots,{\rm deg}(p)\})$ are its ${\rm deg}(p)\in\mathbb{N}$ roots \cite[pp. 284--285]{Zorich_2004a}.
We know that ${\rm deg}(\alpha_{R_1,M_-,M_+,\ell})={\rm deg}(\alpha_{I,M_-,M_+,\ell})=M_-+M_+=M$ \eqref{Eq_Lem_RLRPCS_s_FPLIR_ss_RPsFPs_001_001c}.
The $M$ roots of $\alpha_{I,M_-,M_+,\ell}(\xi)$ are $n\in\big(\{-M_-,\cdots,M_+\}\setminus\{\ell\}\big)\subset\mathbb{Z}$ \eqref{Eq_RLRPCS_s_FPLIR_ss_RFPs_001},
and the $M$ roots of $\alpha_{R_1,M_-,M_+,\ell}(\xi)$, $\xi_{R_1,M_-,M_+,\ell,n}\in\mathbb{R}\;(n\in\{-M_-,\cdots,M_+\}\setminus\{\ell\})$ are real \prpref{Prp_RLRPCS_s_FPLIR_ss_RFPs_001}.
Furthermore ${\rm coeff}[\xi^M,\alpha_{R_1,M_-,M_+,\ell}(\xi)]={\rm coeff}[\xi^M,\alpha_{I  ,M_-,M_+,\ell}(\xi)]$ are given by \eqref{Eq_Lem_RLRPCS_s_FPLIR_ss_RPsFPs_001_001d}.
These facts prove \eqref{Eq_Prp_RLRPCS_s_FPLIR_ss_RFPs_002_001}.\footnote{\label{ff_Prp_RLRPCS_s_FPLIR_ss_RFPs_002_001}
                                                                        Notice that by comparison of \eqref{Eq_Prp_RLRPCS_s_FPLIR_ss_RFPs_002_001b} with \eqref{Eq_Rmk_RLRPCS_s_RB_001_001a}
                                                                        \begin{equation} \dfrac{1}{\displaystyle\prod_{\substack{k=-M_-\\
                                                                                                                                 k\neq\ell}}^{M_+}(\ell-k)}=(-1)^{\ell+M_+}\frac{ 1}
                                                                                                                                                                                {M!}\binom{M       }
                                                                                                                                                                                          {\ell+M_-}\notag\end{equation}
                                                                       as can be easily verified by direct computation.
                                                                       }
\qed
\end{proof}
%
%
\begin{example}[{\rm Fundamental polynomials $\alpha_{I,M_-,M_+,\ell}(\xi)$ \eqref{Eq_Prp_AELRP_s_EPR_ss_PR_001_001h} and $\alpha_{R_1,M_-,M_+,\ell}(\xi)$ \eqref{Eq_Prp_AELRP_s_EPR_ss_PR_001_001g}}]
\label{Xmp_RLRPCS_s_FPLIR_ss_RFPs_001}
Consider the fundamental polynomials of Lagrange reconstruction, $\alpha_{R_1,M_-,M_+,\ell}(\xi)$ \eqref{Eq_Prp_AELRP_s_EPR_ss_PR_001_001g},
and the corresponding fundamental polynomials of Lagrange interpolation, $\alpha_{I,M_-,M_+,\ell}(\xi)$ \eqref{Eq_Prp_AELRP_s_EPR_ss_PR_001_001h},
on the stencils \defref{Def_AELRP_s_EPR_ss_PR_001} $\tsc{s}_{i,3,3}$ \figref{Fig_Xmp_RLRPCS_s_FPLIR_ss_RFPs_001_001} and $\tsc{s}_{i,3,4}$ \figref{Fig_Xmp_RLRPCS_s_FPLIR_ss_RFPs_001_002}.
We know that the corresponding polynomials,
$\alpha_{R_1,M_-,M_+,\ell}(\xi)\in\mathbb{R}_M[\xi]$ \eqref{Eq_Prp_AELRP_s_EPR_ss_PR_001_001g} and $\alpha_{I,M_-,M_+,\ell}(\xi)\in\mathbb{R}_M[\xi]$ \eqref{Eq_Prp_AELRP_s_EPR_ss_PR_001_001h},
$\forall\ell\in\{-M_-,\cdots,M_+\}$,
have $M:=M_-+M_+$ distinct real roots \prpref{Prp_RLRPCS_s_FPLIR_ss_RFPs_001},
each root of the fundamental polynomial of Lagrange reconstruction $\xi_{R_1,M_-,M_+,\ell,n}$ \eqref{Eq_Prp_RLRPCS_s_FPLIR_ss_RFPs_001_001a}
being close to the root $n\in\{-M_-,\cdots,M_+\}\setminus\{\ell\}$ of the corresponding fundamental polynomial of Lagrange interpolation \eqref{Eq_RLRPCS_s_FPLIR_ss_RFPs_001a},
{\em viz} $\abs{\xi_{R_1,M_-,M_+,\ell,n}-n}<\tfrac{1}{2}$ \eqref{Eq_Prp_RLRPCS_s_FPLIR_ss_RFPs_001_001a}. Furthermore \rmkref{Rmk_RLRPCS_s_FPLIR_ss_RFPs_001}
both $\alpha_{I,M_-,M_+,\ell}(\xi)$ \eqref{Eq_Prp_AELRP_s_EPR_ss_PR_001_001h} and $\alpha_{R_1,M_-,M_+,\ell}(\xi)$ \eqref{Eq_Prp_AELRP_s_EPR_ss_PR_001_001g} have $M-1$ corresponding extrema,
again distant $<\tfrac{1}{2}$. For these reasons the shapes of $\alpha_{I,M_-,M_+,\ell}(\xi)$ \eqref{Eq_Prp_AELRP_s_EPR_ss_PR_001_001h} and $\alpha_{R_1,M_-,M_+,\ell}(\xi)$ \eqref{Eq_Prp_AELRP_s_EPR_ss_PR_001_001g}
are quite similar \figrefsab{Fig_Xmp_RLRPCS_s_FPLIR_ss_RFPs_001_001}{Fig_Xmp_RLRPCS_s_FPLIR_ss_RFPs_001_002}.
For the stencil $\tsc{s}_{i,3,3}$ \figref{Fig_Xmp_RLRPCS_s_FPLIR_ss_RFPs_001_001} which is symmetric around $\xi=0$,
we observe that $\xi_{R_1,3,3,\ell,n}\notin\mathbb{Z}\;\forall\ell\in\{-3,+3\}\;\forall n\in\{-3,\cdots,+3\}\setminus\{\ell\}$.
On the contrary, for the stencil $\tsc{s}_{i,3,4}$ \figref{Fig_Xmp_RLRPCS_s_FPLIR_ss_RFPs_001_002} which is symmetric around $\xi=\tfrac{1}{2}$,
we observe that there are two integer roots, $\xi_{R_1,3,4,-3,1}=+1\in\mathbb{Z}$ and $\xi_{R_1,3,4,+4,0}=0\in\mathbb{Z}$.
Although we have not worked out a formal proof concerning integer roots, we can formulate the following conjecture \rstref{Rst_RLRPCS_s_FPLIR_ss_RFPs_001},
obtained using symbolic computation.
\qed
\end{example}
%
%
\begin{figure}[ht!]
\begin{picture}(500,450)
\put(30,-10){\includegraphics[angle=0,width=400pt]{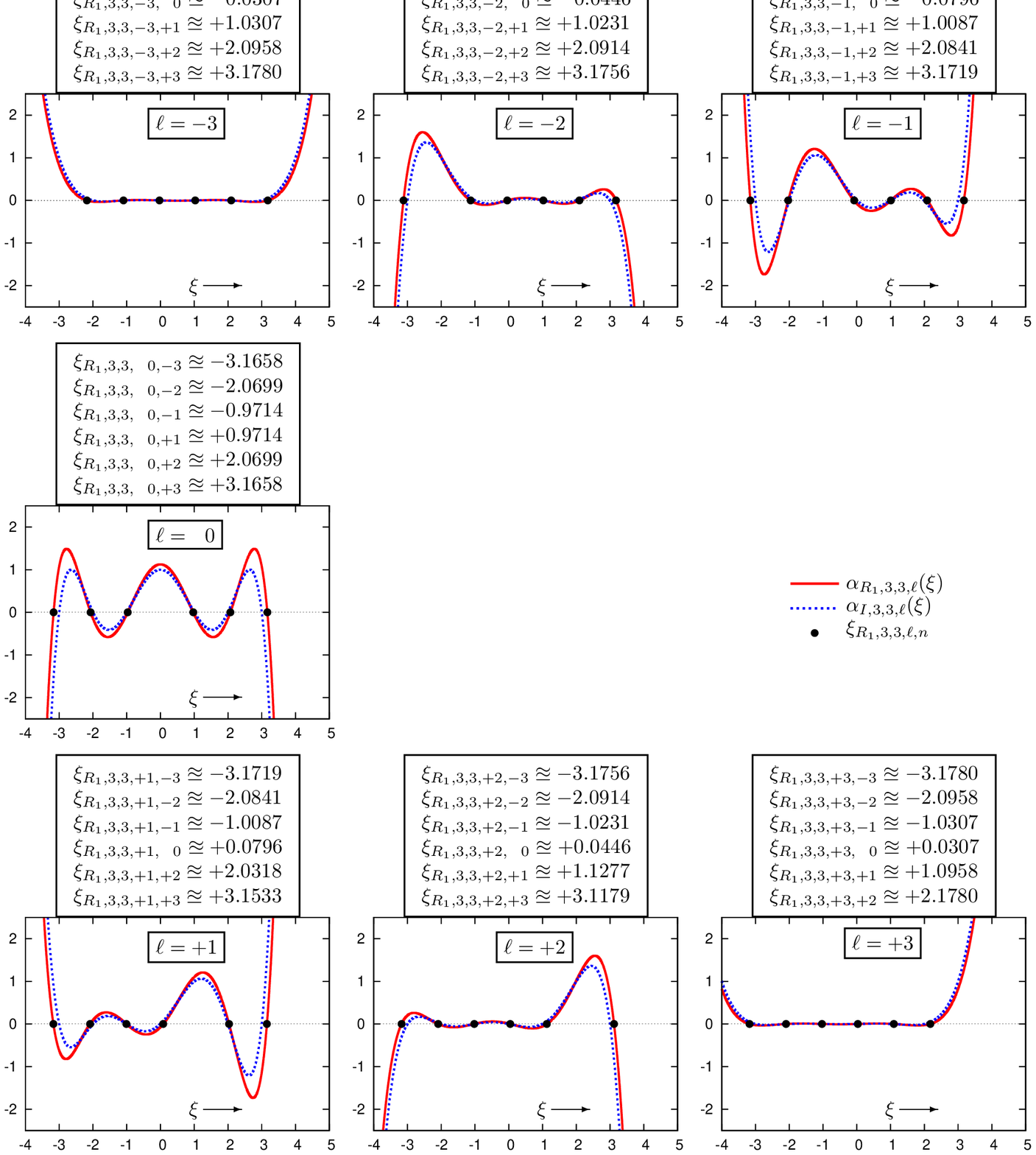}}
\end{picture}
\caption{Fundamenatal polynomials \prpref{Prp_RLRPCS_s_RB_001} of Lagrange interpolation, $\alpha_{I  ,M_-,M_+,\ell}(\xi)$ \eqref{Eq_Prp_AELRP_s_EPR_ss_PR_001_001h},
and reconstruction, $\alpha_{R_1,M_-,M_+,\ell}(\xi)$ \eqref{Eq_Prp_AELRP_s_EPR_ss_PR_001_001g}, on the stencil $\tsc{s}_{i,3,3}$ ($\ell\in\{-3,\cdots,+3\}$),
and locations of the 6 real roots of each $\alpha_{R_1,3,3,\ell}(\xi)\in\mathbb{R}_6[\xi]$ \prpref{Prp_RLRPCS_s_FPLIR_ss_RFPs_001},
$\xi_{R_1,M_-,M_+,\ell,n}$ ($n\in\{-M_-,\cdots,M_+\}\setminus\{\ell\}$), appearing in the factorization \eqref{Eq_Prp_RLRPCS_s_FPLIR_ss_RFPs_002_001a}
of $\alpha_{R_1,M_-,M_+,\ell}(\xi)$ \prpref{Prp_RLRPCS_s_FPLIR_ss_RFPs_002};
notice that in the present case ($M:=M_-+M_+=6$) $\xi_{R_1,3,3,\ell,n}\notin\mathbb{Z}\;\forall\ell\in\{-3,+3\}\;\forall n\in\{-3,\cdots,+3\}\setminus\{\ell\}$ \rstref{Rst_RLRPCS_s_FPLIR_ss_RFPs_001}.}
\label{Fig_Xmp_RLRPCS_s_FPLIR_ss_RFPs_001_001}
\end{figure}
%
%
\begin{figure}[ht!]
\begin{picture}(500,470)
\put(30,-10){\includegraphics[angle=0,width=400pt]{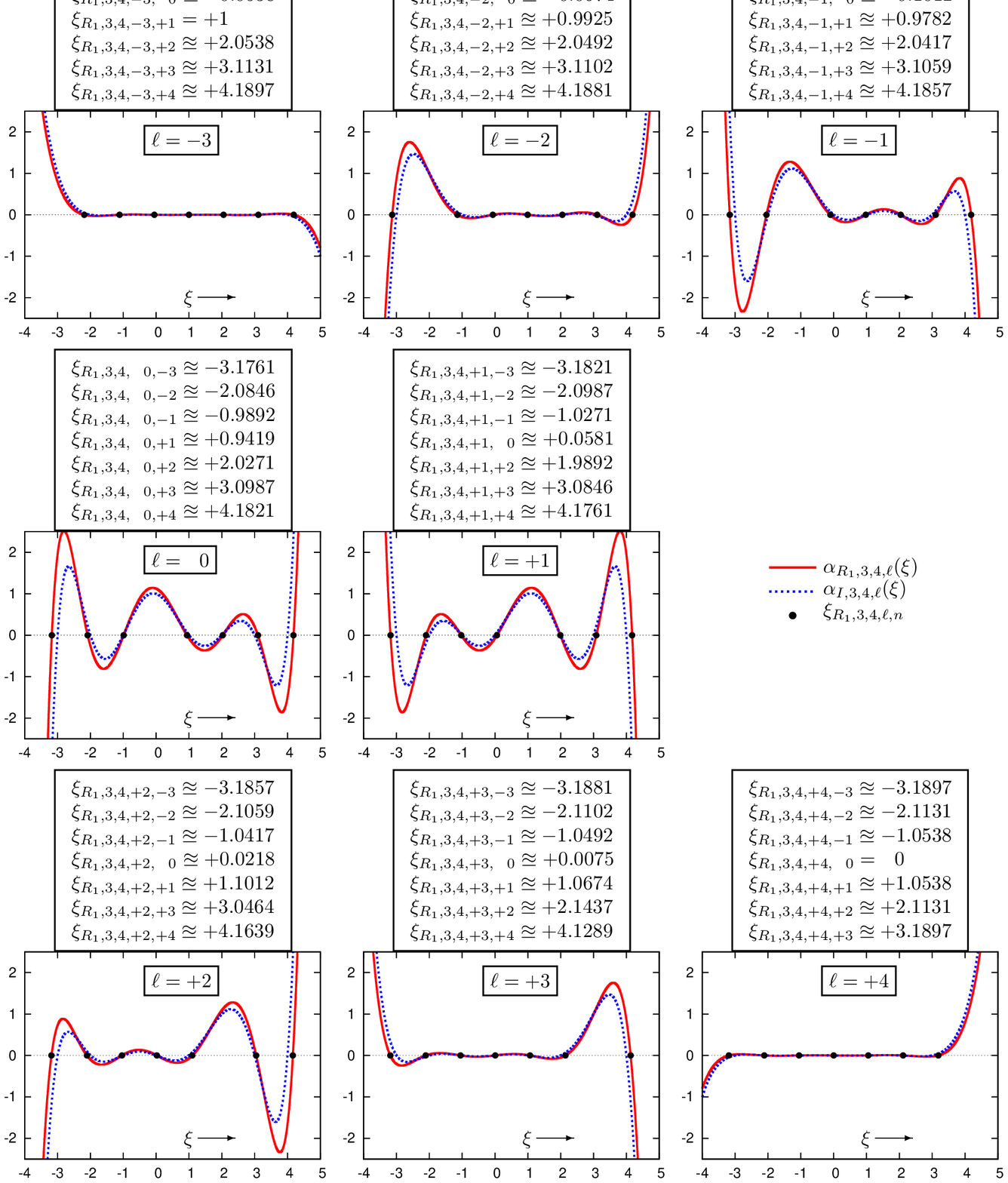}}
\end{picture}
\caption{Fundamenatal polynomials \prpref{Prp_RLRPCS_s_RB_001} of Lagrange interpolation, $\alpha_{I  ,M_-,M_+,\ell}(\xi)$ \eqref{Eq_Prp_AELRP_s_EPR_ss_PR_001_001h},
and reconstruction, $\alpha_{R_1,M_-,M_+,\ell}(\xi)$ \eqref{Eq_Prp_AELRP_s_EPR_ss_PR_001_001g}, on the stencil $\tsc{s}_{i,3,4}$ ($\ell\in\{-3,\cdots,+4\}$),
and locations of the 7 real roots of each $\alpha_{R_1,3,4,\ell}(\xi)\in\mathbb{R}_7[\xi]$ \prpref{Prp_RLRPCS_s_FPLIR_ss_RFPs_001},
$\xi_{R_1,M_-,M_+,\ell,n}$ ($n\in\{-M_-,\cdots,M_+\}\setminus\{\ell\}$), appearing in the factorization \eqref{Eq_Prp_RLRPCS_s_FPLIR_ss_RFPs_002_001a}
of $\alpha_{R_1,M_-,M_+,\ell}(\xi)$ \prpref{Prp_RLRPCS_s_FPLIR_ss_RFPs_002};
notice that in the present case ($M:=M_-+M_+=7$) $\xi_{R_1,3,4,-3,1}=+1\in\mathbb{Z}$ and $\xi_{R_1,3,4,+4,0}=0\in\mathbb{Z}$ \rstref{Rst_RLRPCS_s_FPLIR_ss_RFPs_001}.}
\label{Fig_Xmp_RLRPCS_s_FPLIR_ss_RFPs_001_002}
\end{figure}
%
%
\begin{result}[{\rm Integer roots of $\alpha_{R_1,M_-,M_+,\ell}(\xi)$ \eqref{Eq_Prp_AELRP_s_EPR_ss_PR_001_001g}}]
\label{Rst_RLRPCS_s_FPLIR_ss_RFPs_001}
\begin{subequations}
                                                                                                       \label{Eq_Rst_RLRPCS_s_FPLIR_ss_RFPs_001_001}
Let
\begin{alignat}{6}
M_{\pm}\in\left\{-20,\cdots,+20\right\}\;:\;M:=M_-+M_+\geq1
                                                                                                       \label{Eq_Rst_RLRPCS_s_FPLIR_ss_RFPs_001_001a}
\end{alignat}
Then for $M$ even
\begin{alignat}{6}
M=2k\;;\;k\in{\mathbb N}\;\Longrightarrow\;\xi_{R_1,M_-,M_+,\ell,n}\notin{\mathbb Z}\qquad\left\{\begin{array}{l} \forall\;\ell\in\{-M_-,\cdots,M_+\}\\
                                                                                                                                                     \\
                                                                                                    \forall\;n\in\{-M_-,\cdots,M_+\}\setminus\{\ell\}\\\end{array}\right.
                                                                                                       \label{Eq_Rst_RLRPCS_s_FPLIR_ss_RFPs_001_001b}
\end{alignat}
and for $M$ odd
\begin{alignat}{6}
M=2k+1\;;\;k\in{\mathbb N}_0\;\Longrightarrow\;\left\{\begin{array}{l}\xi_{R_1,M_-,M_+,-M_-,-M_-+\left\lceil\frac{M}{2}\right\rceil}=-M_-+\left\lceil\frac{M}{2}\right\rceil=\dfrac{M_+-M_-+1}{2}\\
                                                                                                                                                                          \\
                                                                      \xi_{R_1,M_-,M_+,+M_+,+M_+-\left\lceil\frac{M}{2}\right\rceil}=+M_+-\left\lceil\frac{M}{2}\right\rceil=\dfrac{M_+-M_--1}{2}\\
                                                                                                                                                                          \\
                                                                       \xi_{R_1,M_-,M_+,\ell,n}\notin{\mathbb Z}\qquad\left\{\begin{array}{l} \forall\;\ell\in\{-M_-+1,\cdots,M_+-1\}\\
                                                                                                                                                                                     \\
                                                                                                                                    \forall\;n\in\{-M_-,\cdots,M_+\}\setminus\{\ell\}\\\end{array}\right.\end{array}\right.
                                                                                                       \label{Eq_Rst_RLRPCS_s_FPLIR_ss_RFPs_001_001c}
\end{alignat}
\end{subequations}
\end{result}
%
%
\begin{vrfctn}
By \prprefnp{Prp_RLRPCS_s_FPLIR_ss_RFPs_001}, we know that all of the roots of the basis polynomials $\alpha_{R_1,M_-,M_+,\ell}(\xi)$ are real.
Since by \eqref{Eq_Prp_AELRP_s_EPR_ss_PR_001_001g} ${\rm deg}(\alpha_{R_1,M_-,M_+,\ell})=M$ ($M:=M_-+M_+$),
there are $M$ real roots, with exactly 1 root in each of the $M$ open intervals $(n-\tfrac{1}{2},n+\tfrac{1}{2})$ $\forall\;n\in\{-M_-,\cdots,M_+\}\setminus\{\ell\}$ \eqref{Eq_Prp_RLRPCS_s_FPLIR_ss_RFPs_001_001a}.
Hence, if $\alpha_{R_1,M_-,M_+,\ell}(\xi)$ has integer roots, these must belong to the set $\left\{\{-M_-,\cdots,M_+\}\setminus\{\ell\}\right\}\subset{\mathbb Z}$, {\em ie}
if $\alpha_{R_1,M_-,M_+,\ell}(\xi)$ has integer roots these must lie on the points of the stencil $\tsc{s}_{i,M_-,M_+}:=\left\{i-M_-,\cdots,i+M_+\right\}$ \eqref{Eq_Def_AELRP_s_EPR_ss_PR_001_001c},
except the point $i+\ell$ itself. As a consequence,
the result was obtained by direct calculation of $\alpha_{R_1,M_-,M_+,\ell}(\xi)$ \eqref{Eq_Prp_AELRP_s_EPR_ss_PR_001_001g} $\forall \ell,n\in\{-M_-,\cdots,M_+\}$ for the range of stencils studied.
\qed
\end{vrfctn}
%

%
%
%
%
%
\subsection{Some identities concerning the fundamental polynomials of Lagrange reconstruction}\label{RLRPCS_s_FPLIR_ss_SICFPLR}
%
%
%
%
%

%
%
%
%
%
%
%
%
%
%
%
%
%
%
%
%
%
%

To build (\S\ref{RLRPCS_s_RCsSs}) the recursive construction of the weight-functions \cite{Shu_1998a,
                                                                                           Shu_2009a,
                                                                                           Liu_Shu_Zhang_2009a}
$\sigma_{R_1,M_-,M_+,K_{\rm s},k_{\rm s}}(\xi)$ for the combination \eqref{Eq_RLRPCS_s_I_001a}
of the polynomial reconstructions $p_{R_1,M_--k_{\rm s},M_+-K_{\rm s}+k_{\rm s}}(x_i+\xi\Delta x;x_i,\Delta x;f)$ on the substencils \defref{Def_RLRPCS_s_I_002} of $\tsc{s}_{i,M_-,M_+}$ \defref{Def_AELRP_s_EPR_ss_PR_001}
to the polynomial reconstruction $p_{R_1,M_-,M_+}(x_i+\xi\Delta x;x_i,\Delta x;f)$ on the big stencil $\tsc{s}_{i,M_-,M_+}$,
we will first examine \lemref{Lem_RLRPCS_s_RCsSs_ss_Ks1_001} the elementary subdivision of $\tsc{s}_{i,M_-,M_+}$ \defref{Def_AELRP_s_EPR_ss_PR_001} into the
substencils $\tsc{s}_{i,M_--1,M_+}$ (which omits the leftmost point $i-M_-$) and $\tsc{s}_{i,M_-,M_+-1}$ (which omits the rightmost point $i+M_+$).
To obtain the general result for $\sigma_{R_1,M_-,M_+,1,0}(\xi)$ and $\sigma_{R_1,M_-,M_+,1,1}(\xi)$ \lemref{Lem_RLRPCS_s_RCsSs_ss_Ks1_001}
we need to show that the leading terms of the approximation error \eqref{Eq_Prp_AELRP_s_EPR_ss_AELPR_002_001a} of the Lagrange reconstructing polynomial,
on 2 overlapping stencils of equal length, but shifted by 1 cell, are different.
Since the error-expansion \eqref{Eq_Prp_AELRP_s_EPR_ss_AELPR_002_001a} polynomials \eqref{Eq_Prp_AELRP_s_EPR_ss_AELPR_002_001c}, $\lambda_{R_1,M_--1,M_+,M}(\xi)$ and $\lambda_{R_1,M_-,M_+-1,M}(\xi)$,
are of degree $M$ \cite[fn8, p. 294, Proposition 4.7]{Gerolymos_2011a},
they can be projected on the basis $\left\{\alpha_{R_1,M_-,M_+,\ell}(\xi),\ell\in\{-M_-,\cdots,M_+\}\right\}$ of $\mathbb{R}_M[\xi]$ \prpref{Prp_RLRPCS_s_FPLIR_ss_RPsFPs_002},
and the same projection is possible for the polynomials
$\left\{\alpha_{R_1,M_--1,M_+,\ell}(\xi),\ell\in\{-M_-+1,\cdots,M_+\}\right\}\in\mathbb{R}_{M-1}[\xi]\subset\mathbb{R}_M[\xi]$ and
$\left\{\alpha_{R_1,M_-,M_+-1,\ell}(\xi),\ell\in\{-M_-,\cdots,M_+-1\}\right\}\in\mathbb{R}_{M-1}[\xi]\subset\mathbb{R}_M[\xi]$.
%
\begin{proposition}[{\rm Identities on $\alpha_{R_1,M_-,M_+,\ell}(\xi)$ \eqref{Eq_Prp_AELRP_s_EPR_ss_PR_001_001g} and $\lambda_{R_1,M_-,M_+,n}(\xi)$ \eqref{Eq_Prp_AELRP_s_EPR_ss_PR_001_001g}}]
\label{Prp_RLRPCS_s_FPLIR_ss_SICFPLR_001}
Assume the conditions and definitions of \prprefnp{Prp_RLRPCS_s_RB_001},
and consider the stencil $\tsc{s}_{i,M_-,M_+}$ \defref{Def_AELRP_s_EPR_ss_PR_001} and its substencils \defref{Def_RLRPCS_s_I_002}
$\tsc{s}_{i,M_-,M_+-1}$ and $\tsc{s}_{i,M_--1,M_+}$.
\begin{subequations}
                                                                                                       \label{Eq_Prp_RLRPCS_s_FPLIR_ss_SICFPLR_001_001}
The following identities hold $\forall \xi\in{\mathbb R}$
\begin{alignat}{6}
\alpha_{R_1,M_--1,M_+,M_+}(\xi)=&(-1)^{M-1}\;\alpha_{R_1,M_-,M_+-1,-M_-}(\xi)
                                                                                                       \label{Eq_Prp_RLRPCS_s_FPLIR_ss_SICFPLR_001_001a}\\
\alpha_{R_1,M_-,M_+-1,\ell}(\xi)\neq&\alpha_{R_1,M_--1,M_+,\ell}(\xi)\qquad\forall\;\ell\in\{-M_-+1,\cdots,M_+-1\}
                                                                                                       \label{Eq_Prp_RLRPCS_s_FPLIR_ss_SICFPLR_001_001b}\\
\lambda_{R_1,M_--1,M_+,M}(\xi)=&(-1)^{M-1}\;\alpha_{R_1,M_-,M_+,-M_-}(\xi)
                                                                                                       \label{Eq_Prp_RLRPCS_s_FPLIR_ss_SICFPLR_001_001c}\\
\lambda_{R_1,M_-,M_+-1,M}(\xi)=&-\alpha_{R_1,M_-,M_+,M_+}(\xi)
                                                                                                       \label{Eq_Prp_RLRPCS_s_FPLIR_ss_SICFPLR_001_001d}\\
0_{\mathbb{R}_{M-1}[\xi]}(\xi)\stackrel{\eqref{Eq_Lem_RLRPCS_s_FPLIR_ss_RPsFPs_001_001b}}{\neq}\alpha_{R_1,M_--1,M_+,M_+}(\xi)=&\lambda_{R_1,M_--1,M_+,M}(\xi)-\lambda_{R_1,M_-,M_+-1,M}(\xi)
                                                                                                       \label{Eq_Prp_RLRPCS_s_FPLIR_ss_SICFPLR_001_001e}
\end{alignat}
\end{subequations}
\end{proposition}
%
%
\begin{proof}
Let $p(x)\in \mathbb{R}_M[x]$ be a polynomial of degree $\leq M$.
Then by \cite[Theorem 5.1, p. 296]{Gerolymos_2011a} its reconstruction pair $q(x):=[R_{(1;\Delta x)}(p)](x)\in \mathbb{R}_M[x]$,
and ${\rm deg}(q)={\rm deg}(p)$ \rmkref{Rmk_RLRPCS_s_RB_002}.

\begin{subequations}
                                                                                                       \label{Eq_Prp_RLRPCS_s_FPLIR_ss_SICFPLR_001_002}
\underline{Proof of \eqref{Eq_Prp_RLRPCS_s_FPLIR_ss_SICFPLR_001_001a}:} By \prprefnp{Prp_RLRPCS_s_RB_001}, $\forall p(x)\in \mathbb{R}_{M-1}[x]$,
its reconstructing polynomials \defref{Def_AELRP_s_RPERR_ss_RP_002} on the 2 stencils
$\tsc{s}_{i,M_-,M_+-1}:=\{i-M_-,\cdots,i+M_+-1\}$ and $\tsc{s}_{i,M_--1,M_+}:=\{i-M_-+1,\cdots,i+M_+\}$, which contain the same number of $M$ points
but are shifted by 1 cell, are exactly equal to the reconstruction pair of $p(x)$ \defref{Def_AELRP_s_RPERR_ss_RP_002} $q(x):=[R_{(1;\Delta x)}(p)](x)$,
because of \eqref{Eq_Prp_AELRP_s_EPR_ss_AELPR_001_001c}, since $p(x)\in \mathbb{R}_{M-1}[x]\Longrightarrow p^{(s)}(x)=0_{\mathbb{R}_{M-1}[x]}(x)\;\forall s\geq M$.
Hence, by \eqref{Eq_Prp_AELRP_s_EPR_ss_PR_001_001d},
\begin{alignat}{6}
 &                                                         p_{R_1,M_-  ,M_+-1}(x_i+\xi\Delta x;x_i,\Delta x;p)\stackrel{\eqref{Eq_Prp_AELRP_s_EPR_ss_PR_001_001d}}{=}&
                                                           \sum_{\ell=-M_-  }^{M_+-1} \alpha_{R_1,M_-  ,M_+-1,\ell}(\xi)\;p(x_i+\ell\;\Delta x)\stackrel{\eqref{Eq_Prp_AELRP_s_EPR_ss_AELPR_001_001c}}{=}
                                                                                                       \notag\\
\stackrel{\eqref{Eq_Prp_AELRP_s_EPR_ss_AELPR_001_001c}}{=}&p_{R_1,M_--1,M_+  }(x_i+\xi\Delta x;x_i,\Delta x;p)\stackrel{\eqref{Eq_Prp_AELRP_s_EPR_ss_PR_001_001d}}{=}&
                                                           \sum_{\ell=-M_-+1}^{M_+  } \alpha_{R_1,M_--1,M_+  ,\ell}(\xi)\;p(x_i+\ell\;\Delta x)\stackrel{\eqref{Eq_Prp_AELRP_s_EPR_ss_AELPR_001_001c}}{=}
                                                                                                       \notag\\
\stackrel{\eqref{Eq_Prp_AELRP_s_EPR_ss_AELPR_001_001c}}{=}&q(x_i+\xi\Delta x):=[R_{(1;\Delta x)}(p)](x_i+\xi\Delta x)&\qquad\qquad\left\{\begin{array}{l}\forall p(x)\in \mathbb{R}_{M-1}[x]\\
                                                                                                                                                         \forall \xi\in{\mathbb R}\\
                                                                                                                                                         \forall x_i\in{\mathbb R}\\
                                                                                                                                                         \forall \Delta x\in{\mathbb R}_{>0}\\\end{array}\right.
                                                                                                        \label{Eq_Prp_RLRPCS_s_FPLIR_ss_SICFPLR_001_002a}
\end{alignat}
whence
\begin{alignat}{6}
\begin{array}{rr}
                                                                           \alpha_{R_1,M_-  ,M_+-1,-M_-}(\xi)\;p(x_i-M_-\;\Delta x)+& \\
\displaystyle
\sum_{\ell=-M_-+1}^{M_+-1}\left(\alpha_{R_1,M_-  ,M_+-1,\ell}(\xi)-\alpha_{R_1,M_--1,M_+  ,\ell}(\xi)\right)\;p(x_i+\ell\;\Delta x)-& \\
                                                                            \alpha_{R_1,M_--1,M_+  ,M_+}(\xi)\;p(x_i+M_+\;\Delta x)=&0\\\end{array}
                                                          \qquad\qquad\left\{\begin{array}{l}\forall p(x)\in \mathbb{R}_{M-1}[x]\\
                                                                                             \forall \xi\in{\mathbb R}\\
                                                                                             \forall x_i\in{\mathbb R}\\
                                                                                             \forall \Delta x\in{\mathbb R}_{>0}\\\end{array}\right.
                                                                                                       \label{Eq_Prp_RLRPCS_s_FPLIR_ss_SICFPLR_001_002b}
\end{alignat}
Applying \eqref{Eq_Prp_RLRPCS_s_FPLIR_ss_SICFPLR_001_002b} to the polynomial
\begin{alignat}{6}
\mathbb{R}_{M-1}[x]\ni\prod_{m=-M_-+1}^{M_+-1}(x-x_i-m\Delta x)=0\quad\forall x\in\left\{x_i-(M_--1)\Delta x,\cdots,x_i+(M_+-1)\Delta x\right\}
                                                                                                        \label{Eq_Prp_RLRPCS_s_FPLIR_ss_SICFPLR_001_002c}
\end{alignat}
yields
\begin{alignat}{6}
                         \alpha_{R_1,M_-  ,M_+-1,-M_-}(\xi)\;\left(\prod_{m=-M_-+1}^{M_+-1}(-M_--m)\right)+&
                                                                                                       \notag\\
                         \sum_{\ell=-M_-+1}^{M_+-1}\left(\alpha_{R_1,M_-  ,M_+-1,\ell}(\xi)-\alpha_{R_1,M_--1,M_+  ,\ell}(\xi)\right)\;
                                                   \underbrace{\left(\prod_{m=-M_-+1}^{M_+-1}(\ell-m)\right)}_{=0\;\forall\;\ell\in\{-M_-+1,\cdots,M_+-1\}}-&
                                                                                                       \notag\\
                         \alpha_{R_1,M_--1,M_+  ,+M_+}(\xi)\;\left(\prod_{m=-M_-+1}^{M_+-1}(+M_+-m)\right)=&0\qquad\qquad\left\{\begin{array}{l}\forall \xi\in{\mathbb R}\\
                                                                                                                                                \forall \Delta x\in{\mathbb R}_{>0}\\\end{array}\right.
                                                                                                       \label{Eq_Prp_RLRPCS_s_FPLIR_ss_SICFPLR_001_002d}
\end{alignat}
{\em ie}
\begin{alignat}{6}
\alpha_{R_1,M_-  ,M_+-1,-M_-}(\xi)\;\left(\prod_{m=-M_-+1}^{M_+-1}(-M_--m)\right)=
\alpha_{R_1,M_--1,M_+  ,+M_+}(\xi)\;\left(\prod_{m=-M_-+1}^{M_+-1}(+M_+-m)\right)  \qquad\qquad\forall \xi\in{\mathbb R}
                                                                                                       \label{Eq_Prp_RLRPCS_s_FPLIR_ss_SICFPLR_001_002e}
\end{alignat}
Since
\begin{alignat}{6}
\prod_{m=-M_-+1}^{M_+-1}(+M_+-m) \stackrel{k:=M_+-m}{=}&\prod_{k=M-1}^{1}k   =&(M-1)!
                                                                                                       \label{Eq_Prp_RLRPCS_s_FPLIR_ss_SICFPLR_001_002f}\\
\prod_{m=-M_-+1}^{M_+-1}(-M_--m)\stackrel{k:=-M_--m}{=}&\prod_{k=-1}^{-(M-1)}k=&(-1)^{M-1}\;(M-1)!
                                                                                                       \label{Eq_Prp_RLRPCS_s_FPLIR_ss_SICFPLR_001_002g}
\end{alignat}
using \eqrefsab{Eq_Prp_RLRPCS_s_FPLIR_ss_SICFPLR_001_002f}{Eq_Prp_RLRPCS_s_FPLIR_ss_SICFPLR_001_002g} in
\eqref{Eq_Prp_RLRPCS_s_FPLIR_ss_SICFPLR_001_002d} proves \eqref{Eq_Prp_RLRPCS_s_FPLIR_ss_SICFPLR_001_001a}.
\end{subequations}

\begin{subequations}
                                                                                                       \label{Eq_Prp_RLRPCS_s_FPLIR_ss_SICFPLR_001_003}
\underline{Proof of \eqref{Eq_Prp_RLRPCS_s_FPLIR_ss_SICFPLR_001_001b}:} Applying \eqref{Eq_Prp_RLRPCS_s_FPLIR_ss_SICFPLR_001_002b}, successively for $k\in\{-M_-+1,\cdots,M_+-1\}$, to the polynomials
\begin{alignat}{6}
\mathbb{R}_{M-1}[x]\ni
\prod_{\substack{m=-M_-+1\\
                 m\neq k}  }^{M_+}(x-x_i-m\Delta x)=0\quad\forall x\in\left\{x_i-(M_--1)\Delta x,\cdots,x_i+M_+\Delta x\right\}\setminus\{x_i+k\Delta x\}
                                                                                                        \label{Eq_Prp_RLRPCS_s_FPLIR_ss_SICFPLR_001_003a}
\end{alignat}
yields
\begin{alignat}{6}
\left(\alpha_{R_1,M_-  ,M_+-1,k}(\xi)-\alpha_{R_1,M_--1,M_+  ,k}(\xi)\right)\prod_{\substack{m=-M_-+1\\
                                                                                             m\neq k}  }^{M_+}(k-m)=-\alpha_{R_1,M_-  ,M_+-1,-M_-}(\xi)\prod_{\substack{m=-M_-+1\\
                                                                                                                                                                         m\neq -M_-}  }^{M_+}(-M_--m)
                                                                                                       \notag\\
\forall\;k\in\{-M_-+1,\cdots,M_+-1\}
                                                                                                        \label{Eq_Prp_RLRPCS_s_FPLIR_ss_SICFPLR_001_003b}
\end{alignat}
and using \eqref{Eq_Prp_RLRPCS_s_FPLIR_ss_SICFPLR_001_002g}
\begin{alignat}{6}
\alpha_{R_1,M_-  ,M_+-1,k}(\xi)-\alpha_{R_1,M_--1,M_+  ,k}(\xi)=-\dfrac{(-1)^M\;M!}
                                                                       {\displaystyle\prod_{\substack{n=k+M_--1\\
                                                                                                      n\neq 0}  }^{k-M_+}n}\;\alpha_{R_1,M_-  ,M_+-1,-M_-}(\xi)\stackrel{\text{\eqref{Eq_Lem_RLRPCS_s_FPLIR_ss_RPsFPs_001_001b}}}{\neq}0_{\mathbb{R}_{M-1}[\xi]}(\xi)
                                                                                                       \notag\\
\forall\;k\in\{-M_-+1,\cdots,M_+-1\}
                                                                                                        \label{Eq_Prp_RLRPCS_s_FPLIR_ss_SICFPLR_001_003c}
\end{alignat}
proving \eqref{Eq_Prp_RLRPCS_s_FPLIR_ss_SICFPLR_001_001b}
by \eqref{Eq_Lem_RLRPCS_s_FPLIR_ss_RPsFPs_001_001b}.\footnote{\label{ff_Prp_RLRPCS_s_FPLIR_ss_SICFPLR_001_001}\begin{alignat}{6}\prod_{\substack{m=-M_-+1\\
                                                                                                                                                  m\neq k}  }^{M_+}(k-m)\stackrel{n:=k-m}{=}
                                                                                                                                \prod_{\substack{n=k+M_--1\\
                                                                                                                                                   n\neq 0}  }^{k-M_+}n\neq0\notag\end{alignat}}
\end{subequations}

\begin{subequations}
                                                                                                       \label{Eq_Prp_RLRPCS_s_FPLIR_ss_SICFPLR_001_004}
\underline{Proof of \eqrefsab{Eq_Prp_RLRPCS_s_FPLIR_ss_SICFPLR_001_001c}{Eq_Prp_RLRPCS_s_FPLIR_ss_SICFPLR_001_001d}:} Notice first that by \eqref{Eq_Prp_AELRP_s_EPR_ss_AELPR_002_001c},
for $n=M+1\stackrel{\eqref{Eq_Def_AELRP_s_EPR_ss_PR_001_001b}}{=}M_-+M_++1$,
\begin{alignat}{6}
\lambda_{R_1,M_-,M_+,M+1}(\xi)\stackrel{\eqref{Eq_Prp_AELRP_s_EPR_ss_AELPR_002_001c}}{=}&\sum_{\ell=0}^{M+1-M-1}\mu_{R_1,M_-,M_+,M+1-\ell}(\xi)
                                                                                                               \dfrac{(-1)^{\ell+1}}
                                                                                                                     {(\ell+1)!    }\left((\xi-\tfrac{1}{2})^{\ell+1}
                                                                                                                                         -(\xi+\tfrac{1}{2})^{\ell+1}\right)
                                                                                                       \notag\\
                                                                                      = &\mu_{R_1,M_-,M_+,M+1}(\xi)\dfrac{-1}
                                                                                                                         {1!}\left( \xi-\tfrac{1}{2}
                                                                                                                                  - \xi-\tfrac{1}{2}\right)
                                                                                                       \notag\\
                                                                                      =& \mu_{R_1,M_-,M_+,M+1}(\xi)\qquad\left\{\begin{array}{l}\forall\xi\in{\mathbb R}\\
                                                                                                                                                \forall M_\pm\in{\mathbb Z}: M:=M_-+M_+\geq1\\\end{array}\right.
                                                                                                       \label{Eq_Prp_RLRPCS_s_FPLIR_ss_SICFPLR_001_004a}
\end{alignat}
where $\mu_{R_1,M_-,M_+,n}(\xi)$ is defined by \eqref{Eq_Prp_AELRP_s_EPR_ss_AELPR_001_001f}. By \eqref{Eq_Prp_RLRPCS_s_FPLIR_ss_SICFPLR_001_004a}
\begin{alignat}{6}
\lambda_{R_1,M_-,M_+-1,M}(\xi)\stackrel{\eqref{Eq_Prp_RLRPCS_s_FPLIR_ss_SICFPLR_001_004a}}{=}&\mu_{R_1,M_-,M_+-1,M}(\xi)
                                                                                                       \label{Eq_Prp_RLRPCS_s_FPLIR_ss_SICFPLR_001_004b}\\
\lambda_{R_1,M_--1,M_+,M}(\xi)\stackrel{\eqref{Eq_Prp_RLRPCS_s_FPLIR_ss_SICFPLR_001_004a}}{=}&\mu_{R_1,M_--1,M_+,M}(\xi)
                                                                                              \qquad\left\{\begin{array}{l}\forall\xi\in{\mathbb R}\\
                                                                                        \forall M_\pm\in{\mathbb Z}: M:=M_-+M_+\geq2\\\end{array}\right.
                                                                                                       \label{Eq_Prp_RLRPCS_s_FPLIR_ss_SICFPLR_001_004c}
\end{alignat}
Since $\forall p(x)\in \mathbb{R}_M[x]\Longrightarrow p^{(n)}(x)=0_{\mathbb{R}_M[x]}(x)\;\forall n\geq M+1$,we have by \eqrefsab{Eq_Prp_AELRP_s_EPR_ss_PR_001_001d}{Eq_Prp_AELRP_s_EPR_ss_AELPR_001_001c},
\begin{alignat}{6}
\sum_{\ell=-M_-  }^{M_+-1} \alpha_{R_1,M_-  ,M_+-1,\ell}(\xi)\;p(x_i+\ell\;\Delta x)\stackrel{\eqrefsab{Eq_Prp_AELRP_s_EPR_ss_PR_001_001d}{Eq_Prp_AELRP_s_EPR_ss_AELPR_001_001c}}{=}&
                                                                                     q(x_i+\xi\Delta x)+\mu_{R_1,M_-  ,M_+-1,M}(\xi)\Delta x^M\;p^{(M)}(x_i)
                                                                                                       \label{Eq_Prp_RLRPCS_s_FPLIR_ss_SICFPLR_001_004d}\\
\sum_{\ell=-M_-  }^{M_+  } \alpha_{R_1,M_-  ,M_+  ,\ell}(\xi)\;p(x_i+\ell\;\Delta x)\stackrel{\eqrefsab{Eq_Prp_AELRP_s_EPR_ss_PR_001_001d}{Eq_Prp_AELRP_s_EPR_ss_AELPR_001_001c}}{=}&
                                                                                     q(x_i+\xi\Delta x)
                                                        \qquad\qquad\left\{\begin{array}{l}\forall p(x)\in \mathbb{R}_M[x]\\
                                                                                           q(x)=[R_{(1;\Delta x)}(p)](x)~~~\text{\eqref{Eq_Rmk_RLRPCS_s_RB_002_001}}\\
                                                                                             \forall \xi\in{\mathbb R}\\
                                                                                             \forall x_i\in{\mathbb R}\\
                                                                                             \forall \Delta x\in{\mathbb R}_{>0}\\\end{array}\right.
                                                                                                       \label{Eq_Prp_RLRPCS_s_FPLIR_ss_SICFPLR_001_004e}\\
\sum_{\ell=-M_--1}^{M_+  } \alpha_{R_1,M_--1,M_+  ,\ell}(\xi)\;p(x_i+\ell\;\Delta x)\stackrel{\eqrefsab{Eq_Prp_AELRP_s_EPR_ss_PR_001_001d}{Eq_Prp_AELRP_s_EPR_ss_AELPR_001_001c}}{=}&
                                                                                     q(x_i+\xi\Delta x)+\mu_{R_1,M_--1,M_+  ,M}(\xi)\Delta x^M\;p^{(M)}(x_i)
                                                                                                        \label{Eq_Prp_RLRPCS_s_FPLIR_ss_SICFPLR_001_004f}
\end{alignat}
for the reconstructing polynomials
$p_{R_1,M_-,M_+-1}(x_i+\xi\Delta x;x_i,\Delta x;p)$ \eqref{Eq_Prp_RLRPCS_s_FPLIR_ss_SICFPLR_001_004d},
$p_{R_1,M_-,M_+}(x_i+\xi\Delta x;x_i,\Delta x;p)$ \eqref{Eq_Prp_RLRPCS_s_FPLIR_ss_SICFPLR_001_004e}, and
$p_{R_1,M_--1,M_+}(x_i+\xi\Delta x;x_i,\Delta x;p)$ \eqref{Eq_Prp_RLRPCS_s_FPLIR_ss_SICFPLR_001_004f}.
Consider the polynomials
\begin{alignat}{6}
\mathbb{R}_M[x]\ni\prod_{m=-M_-}^{M_+-1}(x-x_i-m\Delta x)=0\quad\forall x\in\left\{x_i-M_-\Delta x,\cdots,x_i+(M_+-1)\Delta x\right\}
                                                                                                        \label{Eq_Prp_RLRPCS_s_FPLIR_ss_SICFPLR_001_004g}\\
\mathbb{R}_M[x]\ni\prod_{m=-M_-+1}^{M_+}(x-x_i-m\Delta x)=0\quad\forall x\in\left\{x_i-(M_--1)\Delta x,\cdots,x_i+M_+\Delta x\right\}
                                                                                                        \label{Eq_Prp_RLRPCS_s_FPLIR_ss_SICFPLR_001_004h}
\end{alignat}
Obviously,
\begin{alignat}{6}
\dfrac{d^M}{dx^M}\left(\prod_{m=-M_-}^{M_+-1}(x-x_i-m\Delta x)\right)=
\dfrac{d^M}{dx^M}\left(\prod_{m=-M_-+1}^{M_+}(x-x_i-m\Delta x)\right)=M!\quad\left\{\begin{array}{l}\forall x\in{\mathbb R}\\
                                                                                                    \forall x_i\in{\mathbb R}\\
                                                                                                    \forall \Delta x\in{\mathbb R}_{>0}\\\end{array}\right.
                                                                                                        \label{Eq_Prp_RLRPCS_s_FPLIR_ss_SICFPLR_001_004i}
\end{alignat}
Applying \eqref{Eq_Prp_RLRPCS_s_FPLIR_ss_SICFPLR_001_004e} to the polynomials \eqrefsab{Eq_Prp_RLRPCS_s_FPLIR_ss_SICFPLR_001_004g}{Eq_Prp_RLRPCS_s_FPLIR_ss_SICFPLR_001_004h},
\eqref{Eq_Prp_RLRPCS_s_FPLIR_ss_SICFPLR_001_004d} to the polynomial \eqref{Eq_Prp_RLRPCS_s_FPLIR_ss_SICFPLR_001_004g}, and
\eqref{Eq_Prp_RLRPCS_s_FPLIR_ss_SICFPLR_001_004f} to the polynomial \eqref{Eq_Prp_RLRPCS_s_FPLIR_ss_SICFPLR_001_004h}, we have, using \eqref{Eq_Prp_RLRPCS_s_FPLIR_ss_SICFPLR_001_004i},
\begin{alignat}{6}
\eqrefsab{Eq_Prp_RLRPCS_s_FPLIR_ss_SICFPLR_001_004e}{Eq_Prp_RLRPCS_s_FPLIR_ss_SICFPLR_001_004g}\Longrightarrow&\;
\alpha_{R_1,M_-,M_+,+M_+}(\xi)\;\Delta x^M\;\prod_{m=-M_-}^{M_+-1}( M_+-m)=R_{(1;\Delta x)}\left(\prod_{m=-M_-}^{M_+-1}(x-x_i-m\Delta x)\right)&
                                                                                                       \label{Eq_Prp_RLRPCS_s_FPLIR_ss_SICFPLR_001_004j}\\
\eqrefsab{Eq_Prp_RLRPCS_s_FPLIR_ss_SICFPLR_001_004e}{Eq_Prp_RLRPCS_s_FPLIR_ss_SICFPLR_001_004h}\Longrightarrow&\;
\alpha_{R_1,M_-,M_+,-M_-}(\xi)\;\Delta x^M\;\prod_{m=-M_-+1}^{M_+}(-M_--m)=R_{(1;\Delta x)}\left(\prod_{m=-M_-+1}^{M_+}(x-x_i-m\Delta x)\right)&
                                                                                                       \label{Eq_Prp_RLRPCS_s_FPLIR_ss_SICFPLR_001_004k}\\
\eqrefsab{Eq_Prp_RLRPCS_s_FPLIR_ss_SICFPLR_001_004d}{Eq_Prp_RLRPCS_s_FPLIR_ss_SICFPLR_001_004g}\Longrightarrow&\;
0=R_{(1;\Delta x)}\left(\prod_{m=-M_-}^{M_+-1}(x-x_i-m\Delta x)\right)+\mu_{R_1,M_-  ,M_+-1,M}(\xi)\Delta x^M\;M!&
                                                                                                       \label{Eq_Prp_RLRPCS_s_FPLIR_ss_SICFPLR_001_004l}\\
\eqrefsab{Eq_Prp_RLRPCS_s_FPLIR_ss_SICFPLR_001_004f}{Eq_Prp_RLRPCS_s_FPLIR_ss_SICFPLR_001_004h}\Longrightarrow&\;
0=R_{(1;\Delta x)}\left(\prod_{m=-M_-+1}^{M_+}(x-x_i-m\Delta x)\right)+\mu_{R_1,M_--1,M_+  ,M}(\xi)\Delta x^M\;M!&
                                                                                                       \label{Eq_Prp_RLRPCS_s_FPLIR_ss_SICFPLR_001_004m}\\
                                                                                          &\forall x\in{\mathbb R}\quad
                                                                                           \xi\;\Delta x:=x-x_i\quad
                                                                                           \forall x_i\in{\mathbb R}\quad
                                                                                           \forall \Delta x\in{\mathbb R}_{>0}&
                                                                                                       \notag
\end{alignat}
and combining \eqref{Eq_Prp_RLRPCS_s_FPLIR_ss_SICFPLR_001_004j} with \eqref{Eq_Prp_RLRPCS_s_FPLIR_ss_SICFPLR_001_004l},
and \eqref{Eq_Prp_RLRPCS_s_FPLIR_ss_SICFPLR_001_004k} with \eqref{Eq_Prp_RLRPCS_s_FPLIR_ss_SICFPLR_001_004m}, we have
\begin{alignat}{6}
\alpha_{R_1,M_-,M_+,+M_+}(\xi)\;\prod_{m=-M_-}^{M_+-1}( M_+-m)\stackrel{\eqrefsab{Eq_Prp_RLRPCS_s_FPLIR_ss_SICFPLR_001_004j}
                                                                                 {Eq_Prp_RLRPCS_s_FPLIR_ss_SICFPLR_001_004l}}{=}&-\mu_{R_1,M_-  ,M_+-1,M}(\xi)\;M!&\quad\forall\xi\in{\mathbb R}
                                                                                                       \label{Eq_Prp_RLRPCS_s_FPLIR_ss_SICFPLR_001_004n}\\
\alpha_{R_1,M_-,M_+,-M_-}(\xi)\;\prod_{m=-M_-+1}^{M_+}(-M_--m)\stackrel{\eqrefsab{Eq_Prp_RLRPCS_s_FPLIR_ss_SICFPLR_001_004k}
                                                                                 {Eq_Prp_RLRPCS_s_FPLIR_ss_SICFPLR_001_004m}}{=}&-\mu_{R_1,M_--1,M_+  ,M}(\xi)\;M!&\quad\forall\xi\in{\mathbb R}
                                                                                                       \label{Eq_Prp_RLRPCS_s_FPLIR_ss_SICFPLR_001_004o}
\end{alignat}
which\footnote{\label{ff_Prp_RLRPCS_s_FPLIR_ss_SICFPLR_001_002}By analogy with \eqrefsab{Eq_Prp_RLRPCS_s_FPLIR_ss_SICFPLR_001_002f}{Eq_Prp_RLRPCS_s_FPLIR_ss_SICFPLR_001_002g} we have
               \begin{alignat}{6}\prod_{m=-M_-}^{M_+-1}(+M_+-m)\stackrel{k:=M_+-m}{=}&\prod_{k=M}^{1}k=M!
                                 \qquad;\qquad
                                 \prod_{m=-M_-+1}^{M_+}(-M_--m)\stackrel{k:=-M_--m}{=}&\prod_{k=-1}^{-M}k=(-1)^M\;M!\notag\end{alignat}}
by \eqrefsab{Eq_Prp_RLRPCS_s_FPLIR_ss_SICFPLR_001_004b}{Eq_Prp_RLRPCS_s_FPLIR_ss_SICFPLR_001_004c}
prove \eqrefsab{Eq_Prp_RLRPCS_s_FPLIR_ss_SICFPLR_001_001c}{Eq_Prp_RLRPCS_s_FPLIR_ss_SICFPLR_001_001d}.
\end{subequations}

\begin{subequations}
                                                                                                       \label{Eq_Prp_RLRPCS_s_FPLIR_ss_SICFPLR_001_005}
\underline{Proof of \eqref{Eq_Prp_RLRPCS_s_FPLIR_ss_SICFPLR_001_001e}:}  Applying \eqref{Eq_Prp_RLRPCS_s_FPLIR_ss_SICFPLR_001_004f} to the polynomial \eqref{Eq_Prp_RLRPCS_s_FPLIR_ss_SICFPLR_001_004g} yields
\begin{alignat}{6}
\alpha_{R_1,M_-1,M_+,+M_+}(\xi)\;\Delta x^M\;\prod_{m=-M_-}^{M_+-1}( M_+-m)\stackrel{\eqrefsab{Eq_Prp_RLRPCS_s_FPLIR_ss_SICFPLR_001_004f}
                                                                                              {Eq_Prp_RLRPCS_s_FPLIR_ss_SICFPLR_001_004g}}{=}R_{(1;\Delta x)}\left(\prod_{m=-M_-}^{M_+-1}(x-x_i-m\Delta x)\right)
                                                                                                                                             +\mu_{R_1,M_--1,M_+  ,M}(\xi)\Delta x^M\;M!
                                                                                                       \label{Eq_Prp_RLRPCS_s_FPLIR_ss_SICFPLR_001_005a}\\
                                                                                           \forall x\in{\mathbb R}\quad
                                                                                           \xi\;\Delta x:=x-x_i\quad
                                                                                           \forall x_i\in{\mathbb R}\quad
                                                                                           \forall \Delta x\in{\mathbb R}_{>0}
                                                                                                       \notag
\end{alignat}
Combining \eqref{Eq_Prp_RLRPCS_s_FPLIR_ss_SICFPLR_001_005a} and \eqref{Eq_Prp_RLRPCS_s_FPLIR_ss_SICFPLR_001_004l} yields$^{\text{\ref{ff_Prp_RLRPCS_s_FPLIR_ss_SICFPLR_001_002}}}$
\begin{alignat}{6}
\alpha_{R_1,M_-1,M_+,+M_+}(\xi)\stackrel{\eqrefsab{Eq_Prp_RLRPCS_s_FPLIR_ss_SICFPLR_001_005a}{Eq_Prp_RLRPCS_s_FPLIR_ss_SICFPLR_001_004l}}{=}-\mu_{R_1,M_-  ,M_+-1,M}(\xi)+\mu_{R_1,M_--1,M_+  ,M}(\xi)\qquad\forall \xi\in{\mathbb R}
                                                                                                       \label{Eq_Prp_RLRPCS_s_FPLIR_ss_SICFPLR_001_005b}
\end{alignat}
which by \eqrefsab{Eq_Prp_RLRPCS_s_FPLIR_ss_SICFPLR_001_004b}{Eq_Prp_RLRPCS_s_FPLIR_ss_SICFPLR_001_004c}
proves \eqref{Eq_Prp_RLRPCS_s_FPLIR_ss_SICFPLR_001_001e}.
\end{subequations}
\qed
\end{proof}
%

%
%
%
%
%
%
%
%
%
\section{Reconstruction by combination of substencils}\label{RLRPCS_s_RCsSs}
%
%
%
%
%
%
%
%
%

%
%
%
%
%
\subsection{Substencils of $\tsc{s}_{i,M_-,M_+}$}\label{RLRPCS_s_RCsSs_ss_sSs}
%
%
%
%
%

\tsc{weno} reconstruction \cite{Shu_2009a} on $\tsc{s}_{i,M_-,M_+}$ achieves high-order in smooth regions and monotonicity near discontinuities by a nonlinear
(depending on the values $f_{i+\ell}$ of $f(x)$ on the points of the stencil $\tsc{s}_{i,M_-,M_+}$) combination of reconstructions
on substencils whose union equals the stencil.\footnote{\label{ff_RLRPCS_s_RCsSs_ss_sSs_001} Shu \cite{Shu_1998a,
                                                                                                        Shu_2009a}
                                                        uses the terms big stencil and small stencils to denote the stencil and its substencils.
                                                       }
Central to this development is the determination of the underlying optimal (linear in $f$ in the sense that the weight-functions depend only on $x$ and not on $f$)
combination of the reconstructing polynomials on the substencils to exactly obtain the reconstructing polynomial of the entire stencil.
\begin{figure}[ht!]
\begin{picture}(500,220)
\put(110,-10){\includegraphics[angle=0,width=180pt]{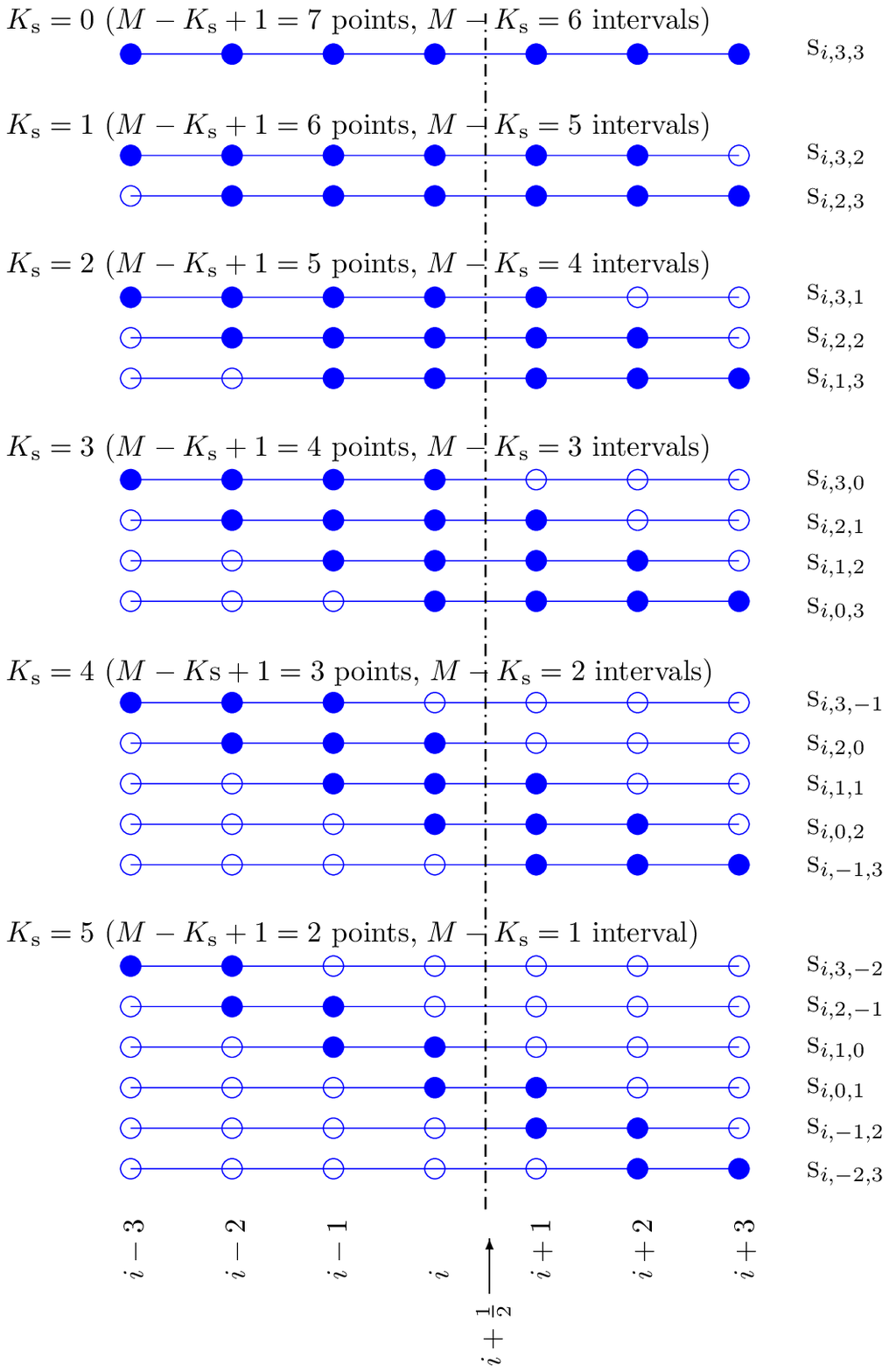}}
\end{picture}
\caption{Successive subdivisions of the stencil $\tsc{s}_{i,3,3}$, for different values of $K_{\rm s}\in\{0,\cdots,M-1=5\}$ \defref{Def_RLRPCS_s_I_002}.}
\label{Fig_Xmp_RLRPCS_s_RCsSs_ss_sSs_001_001}
%
%
\begin{picture}(500,360)
\put(110,-10){\includegraphics[angle=0,width=200pt]{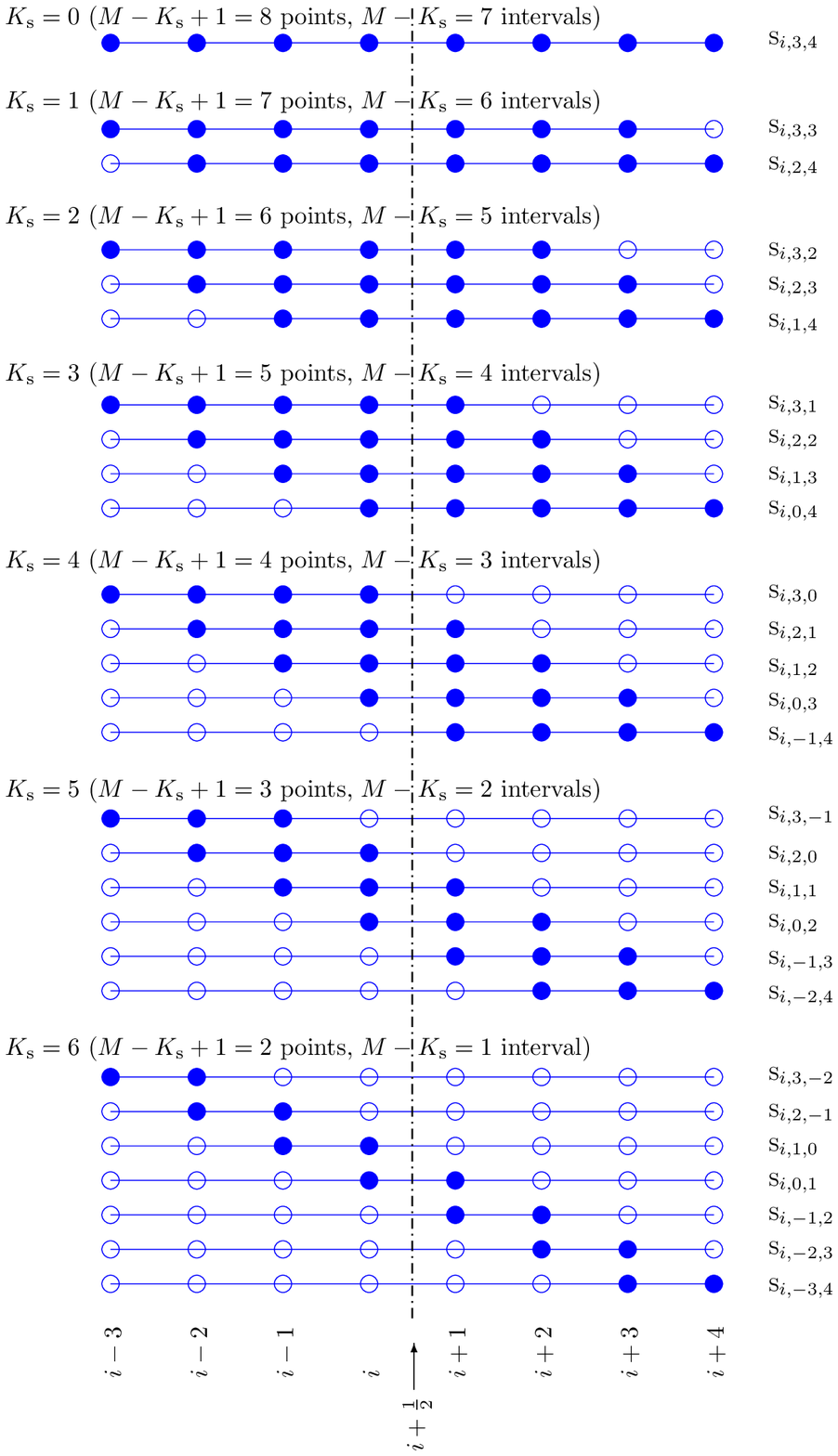}}
\end{picture}
\caption{Successive subdivisions of the stencil $\tsc{s}_{i,3,4}$, for different values of $K_{\rm s}\in\{0,\cdots,M-1=6\}$ \defref{Def_RLRPCS_s_I_002}.}
\label{Fig_Xmp_RLRPCS_s_RCsSs_ss_sSs_001_002}
\end{figure}
%
%
\begin{example}[{\rm Substencils \defref{Def_RLRPCS_s_I_002}}]
\label{Xmp_RLRPCS_s_RCsSs_ss_sSs_001}
Notice that a given stencil $\tsc{s}_{i,M_-,M_+}$ can be divided into different families of substencils \defref{Def_RLRPCS_s_I_002},
depending on the chosen value of $K_{\rm s}\leq M-1$ \eqref{Eq_Def_RLRPCS_s_I_002_001b}.
The $0$-level of subdivision ($K_{\rm s}=0$) corresponds to the original stencil, without subdivision.
The $(M-1)$-level of subdivision ($K_{\rm s}=M-1$) corresponds to the subdivision of the original stencil to $K_{\rm s}+1=M$ substencils of length equal to 1 cell,
{\em ie} to the substencils $\{S_{i,M_-,M_-+1},\cdots,S_{i,M_+-1,M_+}\}$, on each of which polynomial interpolation,
and as a consequence polynomial reconstruction \rmkref{Rmk_RLRPCS_s_RB_002}, are of degree $1$ (linear).
As an example, we consider the successive subdivisions of the stencil $\tsc{s}_{i,3,3}$ \figref{Fig_Xmp_RLRPCS_s_RCsSs_ss_sSs_001_001} which corresponds to a stencil symmetric around point $i$,
and of the stencil $\tsc{s}_{i,3,4}$ \figref{Fig_Xmp_RLRPCS_s_RCsSs_ss_sSs_001_002} which corresponds to a stencil symmetric around point $i+\tfrac{1}{2}$.
We called the substencils of \defrefnp{Def_RLRPCS_s_I_002} Neville substencils \cite{Gerolymos_2011a_news} because they are those used in the Neville algorithm \cite[pp. 207--208]{Henrici_1964a}
for the recursive construction of the interpolating polynomial.
\qed
\end{example}
%
\clearpage
%
%
%
%
%
\subsection{$(K_{\rm s}=1)$-level subdivision}\label{RLRPCS_s_RCsSs_ss_Ks1}
%
%
%
%
%

The starting point for developing a recursive formulation for the weight-functions is to consider the $(K_{\rm s}=1)$-level subdivision of $\tsc{s}_{i,M_-,M_+}$ The resulting substencils
$\tsc{s}_{i,M_-,M_+-1}$ and $\tsc{s}_{i,M_--1,M_+}$ have equal lengths of $M-1$ cells ($M$ points), but are shifted by 1 cell \figrefsab{Fig_Xmp_RLRPCS_s_RCsSs_ss_sSs_001_001}
                                                                                                                                        {Fig_Xmp_RLRPCS_s_RCsSs_ss_sSs_001_002}.
If a ($K_\mathrm{s}=1$)-level subdivision rule can be established, then it can be readily extended to ($K_\mathrm{s}>1$)-levels using the general
recurrence relation proven in \cite[Lemma 2.1]{Gerolymos_2011a_news}.
%
\begin{lemma}[{\rm Rational weight-functions for $(K_{\rm s}=1)$-level subdivision}]
\label{Lem_RLRPCS_s_RCsSs_ss_Ks1_001}
Assume the conditions and definitions of \prprefnp{Prp_RLRPCS_s_RB_001},
and consider the stencil $\tsc{s}_{i,M_-,M_+}$ and its substencils \defref{Def_RLRPCS_s_I_002}
$\tsc{s}_{i,M_-,M_+-1}$ and $\tsc{s}_{i,M_--1,M_+}$.
Define the functions $\sigma_{R_1,M_-,M_+,1,0}(\xi)$ and $\sigma_{R_1,M_-,M_+,1,1}(\xi)$ by
\begin{subequations}
                                                                                                       \label{Eq_Lem_RLRPCS_s_RCsSs_ss_Ks1_001_001}
\begin{alignat}{6}
\sigma_{R_1,M_-,M_+,1,0}(\xi):=&\dfrac{\alpha_{R_1,M_-  ,M_+  ,-M_-}(\xi)}
                                      {\alpha_{R_1,M_-  ,M_+-1,-M_-}(\xi)}\stackrel{\eqref{Eq_Prp_RLRPCS_s_FPLIR_ss_SICFPLR_001_001}}{=}\dfrac{\lambda_{R_1,M_--1,M_+  ,M}(\xi)}
                                                                                                                                              {\lambda_{R_1,M_--1,M_+  ,M}(\xi)-\lambda_{R_1,M_-  ,M_+-1,M}(\xi)}
                                                                                                       \label{Eq_Lem_RLRPCS_s_RCsSs_ss_Ks1_001_001a}\\
\sigma_{R_1,M_-,M_+,1,1}(\xi):=&\dfrac{\alpha_{R_1,M_-  ,M_+  , M_+}(\xi)}
                                      {\alpha_{R_1,M_--1,M_+  , M_+}(\xi)}\stackrel{\eqref{Eq_Prp_RLRPCS_s_FPLIR_ss_SICFPLR_001_001}}{=}\dfrac{\lambda_{R_1,M_-  ,M_+-1,M}(\xi)}
                                                                                                                                              {\lambda_{R_1,M_-  ,M_+-1,M}(\xi)-\lambda_{R_1,M_--1,M_+  ,M}(\xi)}
                                                                                                       \label{Eq_Lem_RLRPCS_s_RCsSs_ss_Ks1_001_001b}
\end{alignat}
satisfying the consistency condition
\begin{alignat}{6}
\sigma_{R_1,M_-,M_+,1,0}(\xi)+\sigma_{R_1,M_-,M_+,1,1}(\xi)=1
                                                                                                       \label{Eq_Lem_RLRPCS_s_RCsSs_ss_Ks1_001_001c}
\end{alignat}
Then the reconstructing polynomial on $\tsc{s}_{i,M_-,M_+}$ \prpref{Prp_RLRPCS_s_RB_001} can be constructed by combination of the reconstructing polynomials
on the 2 $(K_{\rm s}=1)$-level-subdivision substencils as
\begin{alignat}{6}
 &p_{R_1,M_-,M_+}(x_i+\xi\Delta x;x_i,\Delta x;f)=
                                                                                                       \notag\\
=&\dfrac{\lambda_{R_1,M_--1,M_+  ,M}(\xi)\;p_{R_1,M_-  ,M_+-1}(x_i+\xi\Delta x;x_i,\Delta x;f)
        -\lambda_{R_1,M_-  ,M_+-1,M}(\xi)\;p_{R_1,M_--1,M_+  }(x_i+\xi\Delta x;x_i,\Delta x;f)}
        {\lambda_{R_1,M_--1,M_+  ,M}(\xi)-\lambda_{R_1,M_-  ,M_+-1,M}(\xi)}
                                                                                                       \label{Eq_Lem_RLRPCS_s_RCsSs_ss_Ks1_001_001d}\\
                                                                                          &\forall \xi\in{\mathbb R}\quad
                                                                                           \forall x_i\in{\mathbb R}\quad
                                                                                           \forall \Delta x\in{\mathbb R}_{>0}\quad
                                                                                           \forall f:{\mathbb R}\longrightarrow{\mathbb R}
                                                                                                       \notag
\end{alignat}
and can be represented, almost everywhere, as
\begin{alignat}{6}
 p_{R_1,M_-,M_+}(x_i+\xi\Delta x;x_i,\Delta x;f)=&\sigma_{R_1,M_-,M_+,1,0}(\xi)\;p_{R_1,M_-  ,M_+-1}(x_i+\xi\Delta x;x_i,\Delta x;f)
                                                                                                       \notag\\
                                                +&\sigma_{R_1,M_-,M_+,1,1}(\xi)\;p_{R_1,M_--1,M_+  }(x_i+\xi\Delta x;x_i,\Delta x;f)
                                                                                                       \label{Eq_Lem_RLRPCS_s_RCsSs_ss_Ks1_001_001e}
\end{alignat}
\begin{alignat}{6}
                                                                                           \forall \xi\in{\mathbb R}\setminus\{\xi\in{\mathbb R}:\alpha_{R_1,M_--1,M_+  , M_+}(\xi)=0\}\quad
                                                                                           \forall x_i\in{\mathbb R}\quad
                                                                                           \forall \Delta x\in{\mathbb R}_{>0}\quad
                                                                                           \forall f:{\mathbb R}\longrightarrow{\mathbb R}
                                                                                                       \notag
\end{alignat}
The functions $\sigma_{R_1,M_-,M_+,1,0}(\xi)$ and $\sigma_{R_1,M_-,M_+,1,1}(\xi)$ satisfying \eqrefsab{Eq_Lem_RLRPCS_s_RCsSs_ss_Ks1_001_001c}{Eq_Lem_RLRPCS_s_RCsSs_ss_Ks1_001_001e}
are unique.
\end{subequations}
\end{lemma}
%
%
\begin{proof}
\begin{subequations}
                                                                                                       \label{Eq_Lem_RLRPCS_s_RCsSs_ss_Ks1_001_002}
By \prprefnp{Prp_RLRPCS_s_FPLIR_ss_SICFPLR_001} we have
\begin{alignat}{6}
\dfrac{\alpha_{R_1,M_-  ,M_+  ,-M_-}(\xi)}
      {\alpha_{R_1,M_-  ,M_+-1,-M_-}(\xi)}\stackrel{\eqrefsab{Eq_Prp_RLRPCS_s_FPLIR_ss_SICFPLR_001_001a}
                                                             {Eq_Prp_RLRPCS_s_FPLIR_ss_SICFPLR_001_001c}}{=} \dfrac{(-1)^{M-1}\;\lambda_{R_1,M_--1,M_+,M}(\xi)}
                                                                                                                   {(-1)^{M-1}\;\alpha_{R_1,M_--1,M_+  ,+M_+}(\xi)}
                                          \stackrel{\eqref{Eq_Prp_RLRPCS_s_FPLIR_ss_SICFPLR_001_001e}}{=}    \dfrac{\lambda_{R_1,M_--1,M_+,M}(\xi)}
                                                                                                                   {\lambda_{R_1,M_--1,M_+,M}(\xi)-\lambda_{R_1,M_-,M_+-1,M}(\xi)}
                                                                                                       \label{Eq_Lem_RLRPCS_s_RCsSs_ss_Ks1_001_002a}
\end{alignat}
proving \eqref{Eq_Lem_RLRPCS_s_RCsSs_ss_Ks1_001_001a}, and
\begin{alignat}{6}
\dfrac{\alpha_{R_1,M_-  ,M_+  , M_+}(\xi)}
      {\alpha_{R_1,M_--1,M_+  , M_+}(\xi)}\stackrel{\eqrefsab{Eq_Prp_RLRPCS_s_FPLIR_ss_SICFPLR_001_001d}
                                                             {Eq_Prp_RLRPCS_s_FPLIR_ss_SICFPLR_001_001e}}{=}\dfrac{-\lambda_{R_1,M_-,M_+-1,M}(\xi)}
                                                                                                                  {\lambda_{R_1,M_--1,M_+,M}(\xi)-\lambda_{R_1,M_-,M_+-1,M}(\xi)}
                                                                                                       \label{Eq_Lem_RLRPCS_s_RCsSs_ss_Ks1_001_002b}
\end{alignat}
proving \eqref{Eq_Lem_RLRPCS_s_RCsSs_ss_Ks1_001_001b}.
Obviously \eqref{Eq_Lem_RLRPCS_s_RCsSs_ss_Ks1_001_001c} holds because
\begin{alignat}{6}
\sigma_{R_1,M_-,M_+,1,0}(\xi)+\sigma_{R_1,M_-,M_+,1,1}(\xi)
\stackrel{\eqrefsab{Eq_Lem_RLRPCS_s_RCsSs_ss_Ks1_001_001a}{Eq_Lem_RLRPCS_s_RCsSs_ss_Ks1_001_001b}}{=}&
\dfrac{\lambda_{R_1,M_--1,M_+,M}(\xi)}
      {\lambda_{R_1,M_--1,M_+,M}(\xi)-\lambda_{R_1,M_-,M_+-1,M}(\xi)}
                                                                                                       \notag\\
+&
\dfrac{-\lambda_{R_1,M_-,M_+-1,M}(\xi)}
      {\lambda_{R_1,M_--1,M_+,M}(\xi)-\lambda_{R_1,M_-,M_+-1,M}(\xi)}
                                                                                                       \notag\\
=&
\dfrac{\lambda_{R_1,M_--1,M_+,M}(\xi)-\lambda_{R_1,M_-,M_+-1,M}(\xi)}
      {\lambda_{R_1,M_--1,M_+,M}(\xi)-\lambda_{R_1,M_-,M_+-1,M}(\xi)}=1
                                                                                                       \label{Eq_Lem_RLRPCS_s_RCsSs_ss_Ks1_001_002c}
\end{alignat}
\end{subequations}

\begin{subequations}
                                                                                                       \label{Eq_Lem_RLRPCS_s_RCsSs_ss_Ks1_001_003}
\underline{Proof of \eqref{Eq_Lem_RLRPCS_s_RCsSs_ss_Ks1_001_001e}:}
Let $p(x)\in \mathbb{R}_M[x]$ \eqref{Eq_Rmk_RLRPCS_s_RB_002_001}. Then, by \eqref{Eq_Rmk_RLRPCS_s_RB_002_001}, $q(x):=[R_{(1;\Delta x)}(p)](x)\in \mathbb{R}_M[x]$.
Since $\forall q(x)\in \mathbb{R}_M[x]\Longrightarrow q^{(n)}(x)=0_{\mathbb{R}_M[x]}\;\forall n\geq M+1$,
we have, by application of \eqref{Eq_Prp_AELRP_s_EPR_ss_AELPR_002_001a}, and taking into account \rmkrefnp{Rmk_RLRPCS_s_RB_002},
\begin{alignat}{6}
p_{R_1,M_-  ,M_+-1}(x_i+\xi\Delta x;x_i,\Delta x;p)\stackrel{\eqrefsab{Eq_Prp_AELRP_s_EPR_ss_AELPR_002_001a}
                                                                      {Eq_Rmk_RLRPCS_s_RB_002_001}}{=}&q(x_i+\xi\Delta x)+\lambda_{R_1,M_-  ,M_+-1,M}(\xi)\;\Delta x^M\;q^{(M)}(x_i+\xi\Delta x)
                                                                                                       \label{Eq_Lem_RLRPCS_s_RCsSs_ss_Ks1_001_003a}\\
p_{R_1,M_-  ,M_+  }(x_i+\xi\Delta x;x_i,\Delta x;p)\stackrel{\eqrefsab{Eq_Prp_AELRP_s_EPR_ss_AELPR_002_001a}
                                                                      {Eq_Rmk_RLRPCS_s_RB_002_001}}{=}&q(x_i+\xi\Delta x)
                                                        \qquad\qquad\left\{\begin{array}{l}\forall p(x)\in \mathbb{R}_M[x]\\
                                                                                           q(x):=[R_{(1;\Delta x)}(p)](x)~~~\text{\eqref{Eq_Rmk_RLRPCS_s_RB_002_001}}\\
                                                                                             \forall \xi\in{\mathbb R}\\
                                                                                             \forall x_i\in{\mathbb R}\\
                                                                                             \forall \Delta x\in{\mathbb R}_{>0}\\\end{array}\right.
                                                                                                       \label{Eq_Lem_RLRPCS_s_RCsSs_ss_Ks1_001_003b}\\
p_{R_1,M_--1,M_+  }(x_i+\xi\Delta x;x_i,\Delta x;p)\stackrel{\eqrefsab{Eq_Prp_AELRP_s_EPR_ss_AELPR_002_001a}
                                                                      {Eq_Rmk_RLRPCS_s_RB_002_001}}{=}&q(x_i+\xi\Delta x)+\lambda_{R_1,M_--1,M_+  ,M}(\xi)\;\Delta x^M\;q^{(M)}(x_i+\xi\Delta x)
                                                                                                       \label{Eq_Lem_RLRPCS_s_RCsSs_ss_Ks1_001_003c}
\end{alignat}
Combining
\eqref{Eq_Lem_RLRPCS_s_RCsSs_ss_Ks1_001_003a} weighted by \eqref{Eq_Lem_RLRPCS_s_RCsSs_ss_Ks1_001_001a},and
\eqref{Eq_Lem_RLRPCS_s_RCsSs_ss_Ks1_001_003c} weighted by \eqref{Eq_Lem_RLRPCS_s_RCsSs_ss_Ks1_001_001b} yields
\begin{alignat}{6}
&\sigma_{R_1,M_-,M_+,1,0}(\xi)\;p_{R_1,M_-  ,M_+-1}(x_i+\xi\Delta x;x_i,\Delta x;p)+
 \sigma_{R_1,M_-,M_+,1,1}(\xi)\;p_{R_1,M_--1,M_+  }(x_i+\xi\Delta x;x_i,\Delta x;p)=
                                                                                                       \notag\\
\stackrel{\eqrefsab{Eq_Lem_RLRPCS_s_RCsSs_ss_Ks1_001_003a}{Eq_Lem_RLRPCS_s_RCsSs_ss_Ks1_001_003c}}{=}&
                                                                                     \left(\sigma_{R_1,M_-,M_+,1,0}(\xi)
                                                                                          +\sigma_{R_1,M_-,M_+,1,1}(\xi)\right)\;q(x_i+\xi\Delta x)
                                                                                                       \notag\\
                                                                                   +&\left(\sigma_{R_1,M_-,M_+,1,0}(\xi)\;\lambda_{R_1,M_-  ,M_+-1,M}(\xi)
                                                                                          +\sigma_{R_1,M_-,M_+,1,1}(\xi)\;\lambda_{R_1,M_--1,M_+  ,M}(\xi)\right)\;\Delta x^M\;q^{(M)}(x_i+\xi\Delta x)
                                                                                                       \notag\\
\stackrel{\eqrefsatob{Eq_Lem_RLRPCS_s_RCsSs_ss_Ks1_001_001a}{Eq_Lem_RLRPCS_s_RCsSs_ss_Ks1_001_001c}}{=}&
                                                                                     q(x_i+\xi\Delta x)
                                                                                                       \notag\\
                                                                                   +&\left(\dfrac{ \lambda_{R_1,M_--1,M_+,M}(\xi)\;\lambda_{R_1,M_-  ,M_+-1,M}(\xi)}
                                                                                                 { \lambda_{R_1,M_--1,M_+,M}(\xi) -\lambda_{R_1,M_-  ,M_+-1,M}(\xi)}
                                                                                          +\dfrac{-\lambda_{R_1,M_-,M_+-1,M}(\xi)\;\lambda_{R_1,M_--1,M_+  ,M}(\xi)}
                                                                                                 { \lambda_{R_1,M_--1,M_+,M}(\xi) -\lambda_{R_1,M_-  ,M_+-1,M}(\xi)}\right)\;\Delta x^M\;q^{(M)}(x_i+\xi\Delta x)
                                                                                                       \notag\\
=&q(x_i+\xi\Delta x)
\stackrel{\eqref{Eq_Lem_RLRPCS_s_RCsSs_ss_Ks1_001_003b}}{:=}
p_{R_1,M_-  ,M_+  }(x_i+\xi\Delta x;x_i,\Delta x;p)
                                                              \qquad\left\{\begin{array}{l} \forall p(x)\in \mathbb{R}_M[x]\\
                                                                                            q(x):=[R_{(1;\Delta x)}(p)](x)~\text{\eqref{Eq_Rmk_RLRPCS_s_RB_002_001}}\\
                                                                                             \forall \xi\in{\mathbb R}\\
                                                                                             \forall x_i\in{\mathbb R}\\
                                                                                             \forall \Delta x\in{\mathbb R}_{>0}\\\end{array}\right.
                                                                                                       \label{Eq_Lem_RLRPCS_s_RCsSs_ss_Ks1_001_003d}
\end{alignat}
%
\begin{alignat}{6}
\Longrightarrow
&\sigma_{R_1,M_-,M_+,1,0}(\xi)\;p_{R_1,M_-  ,M_+-1}(x_i+\xi\Delta x;x_i,\Delta x;p)+
                                                                                                       \notag\\
&\sigma_{R_1,M_-,M_+,1,1}(\xi)\;p_{R_1,M_--1,M_+  }(x_i+\xi\Delta x;x_i,\Delta x;p)=p_{R_1,M_-  ,M_+  }(x_i+\xi\Delta x;x_i,\Delta x;p)
                                                                                                       \label{Eq_Lem_RLRPCS_s_RCsSs_ss_Ks1_001_003e}\\
                                                                                          &\forall p(x)\in \mathbb{R}_M[x]\qquad
                                                                                             \forall \xi\in{\mathbb R}\setminus\{\xi\in{\mathbb R}:\alpha_{R_1,M_--1,M_+  , M_+}(\xi)=0\}\qquad
                                                                                             \forall x_i\in{\mathbb R}\qquad
                                                                                             \forall \Delta x\in{\mathbb R}_{>0}
                                                                                                       \notag
\end{alignat}
which shows that \eqref{Eq_Lem_RLRPCS_s_RCsSs_ss_Ks1_001_001e} is valid $\forall f(x)\in\mathbb{R}_M[x]$.
Using the representation \eqref{Eq_Prp_AELRP_s_EPR_ss_PR_001_001d} of the reconstructing polynomial in \eqref{Eq_Lem_RLRPCS_s_RCsSs_ss_Ks1_001_003e}
\begin{alignat}{6}
0\stackrel{\eqrefsab{Eq_Lem_RLRPCS_s_RCsSs_ss_Ks1_001_003e}{Eq_Prp_AELRP_s_EPR_ss_PR_001_001d}}{=}&
   \left(\alpha_{R_1,M_-,M_+,-M_-}(\xi)-\sigma_{R_1,M_-,M_+,1,0}(\xi)\;\alpha_{R_1,M_-  ,M_+-1,-M_-}(\xi)\right)\;p(x_i-M_-\;\Delta x)
                                                                                                       \notag\\
 +&\sum_{\ell=-M_-+1}^{M_+-1}\left(\alpha_{R_1,M_-,M_+,\ell}(\xi)-\sigma_{R_1,M_-,M_+,1,0}(\xi)\;\alpha_{R_1,M_-  ,M_+-1,\ell}(\xi)
                                                                 -\sigma_{R_1,M_-,M_+,1,1}(\xi)\;\alpha_{R_1,M_--1,M_+  ,\ell}(\xi)\right)\;p(x_i+\ell\;\Delta x)
                                                                                                       \notag\\
 +&\left(\alpha_{R_1,M_-,M_+,+M_+}(\xi)-\sigma_{R_1,M_-,M_+,1,1}(\xi)\;\alpha_{R_1,M_--1,M_+  ,+M_+}(\xi)\right)\;p(x_i+M_+\;\Delta x)
                                                                                                       \notag\\
\stackrel{\eqrefsab{Eq_Lem_RLRPCS_s_RCsSs_ss_Ks1_001_001a}{Eq_Lem_RLRPCS_s_RCsSs_ss_Ks1_001_001b}}{=}&
   \sum_{\ell=-M_-+1}^{M_+-1}\left(\alpha_{R_1,M_-,M_+,\ell}(\xi)-\sigma_{R_1,M_-,M_+,1,0}(\xi)\;\alpha_{R_1,M_-  ,M_+-1,\ell}(\xi)
                                                                 -\sigma_{R_1,M_-,M_+,1,1}(\xi)\;\alpha_{R_1,M_--1,M_+  ,\ell}(\xi)\right)\;p(x_i+\ell\;\Delta x)
                                                                                                       \notag\\
                                                                                          &\forall p(x)\in \mathbb{R}_M[x]\qquad
                                                                                             \forall \xi\in{\mathbb R}\setminus\{\xi\in{\mathbb R}:\alpha_{R_1,M_--1,M_+  , M_+}(\xi)=0\}\qquad
                                                                                             \forall x_i\in{\mathbb R}\qquad
                                                                                             \forall \Delta x\in{\mathbb R}_{>0}
                                                                                                       \label{Eq_Lem_RLRPCS_s_RCsSs_ss_Ks1_001_003f}
\end{alignat}
where we used
\begin{alignat}{6}
\alpha_{R_1,M_-  ,M_+  ,-M_-}(\xi)\stackrel{\text{\eqref{Eq_Lem_RLRPCS_s_RCsSs_ss_Ks1_001_001a}}}{=}&\sigma_{R_1,M_-,M_+,1,0}(\xi)\;\alpha_{R_1,M_-  ,M_+-1,-M_-}(\xi)\qquad\forall\xi\in{\mathbb R}
                                                                                                       \label{Eq_Lem_RLRPCS_s_RCsSs_ss_Ks1_001_003g}\\
\alpha_{R_1,M_-  ,M_+  , M_+}(\xi)\stackrel{\text{\eqref{Eq_Lem_RLRPCS_s_RCsSs_ss_Ks1_001_001b}}}{=}&\sigma_{R_1,M_-,M_+,1,1}(\xi)\;\alpha_{R_1,M_--1,M_+  , M_+}(\xi)\qquad\forall\xi\in{\mathbb R}
                                                                                                       \label{Eq_Lem_RLRPCS_s_RCsSs_ss_Ks1_001_003h}
\end{alignat}
Applying \eqref{Eq_Lem_RLRPCS_s_RCsSs_ss_Ks1_001_003f} successively to the polynomials
\begin{alignat}{6}
\mathbb{R}_M[x]\ni
\prod_{\substack{m=-M_-\\
                 m\neq k}  }^{M_+}(x-x_i-m\Delta x)=0\quad\forall x\in\left\{x_i-M_-\Delta x,\cdots,x_i+M_+\Delta x\right\}\setminus\{x_i+k\Delta x\}
                                                                                                        \label{Eq_Lem_RLRPCS_s_RCsSs_ss_Ks1_001_003i}
\end{alignat}
yields
\begin{alignat}{6}
\alpha_{R_1,M_-,M_+,\ell}(\xi)=&\sigma_{R_1,M_-,M_+,1,0}(\xi)\;\alpha_{R_1,M_-  ,M_+-1,\ell}(\xi)
                                                                                                       \notag\\
                              +&\sigma_{R_1,M_-,M_+,1,1}(\xi)\;\alpha_{R_1,M_--1,M_+  ,\ell}(\xi)\qquad\left\{\begin{array}{l}\forall\xi\in{\mathbb R}\setminus\{\xi\in{\mathbb R}:\alpha_{R_1,M_--1,M_+  , M_+}(\xi)=0\}\\
                                                                                                                              \forall\ell\in\{-M_-+1,\cdots,M_+-1\}\\\end{array}\right.
                                                                                                        \label{Eq_Lem_RLRPCS_s_RCsSs_ss_Ks1_001_003j}
\end{alignat}
Combining the representation \eqref{Eq_Prp_AELRP_s_EPR_ss_PR_001_001d} of the reconstructing polynomial with
\eqrefsabc{Eq_Lem_RLRPCS_s_RCsSs_ss_Ks1_001_003g}{Eq_Lem_RLRPCS_s_RCsSs_ss_Ks1_001_003h}{Eq_Lem_RLRPCS_s_RCsSs_ss_Ks1_001_003j}
proves \eqref{Eq_Lem_RLRPCS_s_RCsSs_ss_Ks1_001_001e}, $\forall f:\mathbb{R}\to\mathbb{R}$.
\end{subequations}

\begin{subequations}
                                                                                                       \label{Eq_Lem_RLRPCS_s_RCsSs_ss_Ks1_001_004}
\underline{Proof of \eqref{Eq_Lem_RLRPCS_s_RCsSs_ss_Ks1_001_001d}:} Obviously, the functions $\sigma_{R_1,M_-,M_+,1,0}(\xi)$ and $\sigma_{R_1,M_-,M_+,1,1}(\xi)$ are defined everywhere ($\forall\xi\in{\mathbb R}$) except at
\begin{alignat}{6}
                                                               &\{\xi_{R_1,M_--1,M_+,M_+,n};\;n\in\{-M_-+1,\cdots,M_+-1\}
                                                                                                       \label{Eq_Lem_RLRPCS_s_RCsSs_ss_Ks1_001_004a}\\
\stackrel{\eqref{Eq_Prp_RLRPCS_s_FPLIR_ss_RFPs_001_001}}{=}    &\left\{\xi\in{\mathbb R}:\alpha_{R_1,M_--1,M_+  ,+M_+}(\xi)=0\right\}
                                                                                                       \label{Eq_Lem_RLRPCS_s_RCsSs_ss_Ks1_001_004b}\\
\stackrel{\eqref{Eq_Prp_RLRPCS_s_FPLIR_ss_SICFPLR_001_001a}}{=}&\left\{\xi\in{\mathbb R}:\alpha_{R_1,M_-  ,M_+-1,-M_-}(\xi)=0\right\}
                                                                                                       \label{Eq_Lem_RLRPCS_s_RCsSs_ss_Ks1_001_004c}\\
\stackrel{\eqref{Eq_Prp_RLRPCS_s_FPLIR_ss_SICFPLR_001_001e}}{=}&\left\{\xi\in{\mathbb R}:\lambda_{R_1,M_--1,M_+,M}(\xi)-\lambda_{R_1,M_-,M_+-1,M}(\xi)=0\right\}
                                                                                                       \label{Eq_Lem_RLRPCS_s_RCsSs_ss_Ks1_001_004d}
\end{alignat}
Recall that \prpref{Prp_RLRPCS_s_FPLIR_ss_RFPs_001} all of the $M-1$ roots of the polynomial $\alpha_{R_1,M_--1,M_+  ,+M_+}(\xi)$ are real \eqref{Eq_Prp_RLRPCS_s_FPLIR_ss_RFPs_001_001}.
However, using \eqrefsab{Eq_Prp_RLRPCS_s_FPLIR_ss_SICFPLR_001_001a}{Eq_Prp_RLRPCS_s_FPLIR_ss_SICFPLR_001_001b} in \eqref{Eq_Lem_RLRPCS_s_RCsSs_ss_Ks1_001_001e} yields
\begin{alignat}{6}
 &p_{R_1,M_-,M_+}(x_i+\xi\Delta x;x_i,\Delta x;f)\stackrel{\eqrefsabc{Eq_Lem_RLRPCS_s_RCsSs_ss_Ks1_001_001e}{Eq_Prp_RLRPCS_s_FPLIR_ss_SICFPLR_001_001a}{Eq_Prp_RLRPCS_s_FPLIR_ss_SICFPLR_001_001b}}{=}
                                                                                                       \notag\\
=&\dfrac{\lambda_{R_1,M_--1,M_+  ,M}(\xi)\;p_{R_1,M_-  ,M_+-1}(x_i+\xi\Delta x;x_i,\Delta x;f)
        -\lambda_{R_1,M_-  ,M_+-1,M}(\xi)\;p_{R_1,M_--1,M_+  }(x_i+\xi\Delta x;x_i,\Delta x;f)}
        {\lambda_{R_1,M_--1,M_+  ,M}(\xi)-\lambda_{R_1,M_-  ,M_+-1,M}(\xi)}
                                                                                                       \label{Eq_Lem_RLRPCS_s_RCsSs_ss_Ks1_001_004d}\\
                                                                                          &\forall \xi\in{\mathbb R}\setminus\{\xi\in{\mathbb R}:\alpha_{R_1,M_--1,M_+  , M_+}(\xi)=0\}\quad
                                                                                           \forall x_i\in{\mathbb R}\quad
                                                                                           \forall \Delta x\in{\mathbb R}_{>0}\quad
                                                                                           \forall f:{\mathbb R}\longrightarrow{\mathbb R}
                                                                                                       \notag
\end{alignat}
or equivalently, using the representation \eqref{Eq_Prp_AELRP_s_EPR_ss_PR_001_001d} of the reconstructing polynomial in \eqref{Eq_Lem_RLRPCS_s_RCsSs_ss_Ks1_001_004d}
\begin{alignat}{6}
 &p_{R_1,M_-,M_+}(x_i+\xi\Delta x;x_i,\Delta x;f)\stackrel{\eqrefsab{Eq_Lem_RLRPCS_s_RCsSs_ss_Ks1_001_004d}{Eq_Prp_AELRP_s_EPR_ss_PR_001_001d}}{=}
                                                                                                       \notag\\
                                                  &\dfrac{\lambda_{R_1,M_--1,M_+,M}(\xi)\;\alpha_{R_1,M_-  ,M_+-1,-M_-}(\xi)}
                                                         {\lambda_{R_1,M_--1,M_+,M}(\xi)-\lambda_{R_1,M_-,M_+-1,M}(\xi)     }\;f(x_i-M_-\Delta x)
                                                                                                       \notag\\
                                                 +&\sum_{\ell=-M_-+1}^{M_+-1}\dfrac{\lambda_{R_1,M_--1,M_+,M}(\xi)\;\alpha_{R_1,M_-  ,M_+-1,\ell}(\xi)
                                                                                   -\lambda_{R_1,M_-,M_+-1,M}(\xi)\;\alpha_{R_1,M_--1,M_+  ,\ell}(\xi)}
                                                                                   {\lambda_{R_1,M_--1,M_+,M}(\xi)-\lambda_{R_1,M_-,M_+-1,M}(\xi)     }\;f(x_i+\ell\Delta x)
                                                                                                       \notag\\
                                                 +&\dfrac{-\lambda_{R_1,M_-,M_+-1,M}(\xi)\;\alpha_{R_1,M_--1,M_+  , M_+}(\xi)}
                                                         {\lambda_{R_1,M_--1,M_+,M}(\xi)-\lambda_{R_1,M_-,M_+-1,M}(\xi)      }\;f(x_i+M_+\Delta x)
                                                                                                       \label{Eq_Lem_RLRPCS_s_RCsSs_ss_Ks1_001_004e}\\
                                                                                          &\forall \xi\in{\mathbb R}\setminus\{\xi\in{\mathbb R}:\alpha_{R_1,M_--1,M_+  , M_+}(\xi)=0\}\quad
                                                                                           \forall x_i\in{\mathbb R}\quad
                                                                                           \forall \Delta x\in{\mathbb R}_{>0}\quad
                                                                                           \forall f:{\mathbb R}\longrightarrow{\mathbb R}
                                                                                                       \notag
\end{alignat}
To prove \eqref{Eq_Lem_RLRPCS_s_RCsSs_ss_Ks1_001_001d} by \eqref{Eq_Lem_RLRPCS_s_RCsSs_ss_Ks1_001_004e} we need to show that it is valid $\forall\xi\in{\mathbb R}$.
Rewriting \eqrefsab{Eq_Lem_RLRPCS_s_RCsSs_ss_Ks1_001_002a}{Eq_Lem_RLRPCS_s_RCsSs_ss_Ks1_001_002b} we have
\begin{alignat}{6}
\alpha_{R_1,M_-  ,M_+  ,-M_-}(\xi)\stackrel{\eqref{Eq_Lem_RLRPCS_s_RCsSs_ss_Ks1_001_002a}}{=}&\dfrac{\lambda_{R_1,M_--1,M_+,M}(\xi)\;\alpha_{R_1,M_-  ,M_+-1,-M_-}(\xi)}
                                                                                                    {\lambda_{R_1,M_--1,M_+,M}(\xi)-\lambda_{R_1,M_-,M_+-1,M}(\xi)}\qquad\forall \xi\in{\mathbb R}\Longrightarrow
                                                                                                       \notag\\
                                   &      \Big(\lambda_{R_1,M_--1,M_+,M}(\xi)-\lambda_{R_1,M_-,M_+-1,M}(\xi)\Big)\;\mid\;\Big(\lambda_{R_1,M_--1,M_+,M}(\xi)\;\alpha_{R_1,M_-  ,M_+-1,-M_-}(\xi)\Big)
                                                                                                       \label{Eq_Lem_RLRPCS_s_RCsSs_ss_Ks1_001_004f}\\
\alpha_{R_1,M_-  ,M_+  ,+M_+}(\xi)\stackrel{\eqref{Eq_Lem_RLRPCS_s_RCsSs_ss_Ks1_001_002a}}{=}&\dfrac{-\lambda_{R_1,M_-,M_+-1,M}(\xi)\;\alpha_{R_1,M_--1,M_+  , M_+}(\xi)}
                                                                                                    {\lambda_{R_1,M_--1,M_+,M}(\xi)-\lambda_{R_1,M_-,M_+-1,M}(\xi)}\qquad\forall \xi\in{\mathbb R}\Longrightarrow
                                                                                                       \notag\\
                                   &      \Big(\lambda_{R_1,M_--1,M_+,M}(\xi)-\lambda_{R_1,M_-,M_+-1,M}(\xi)\Big)\;\mid\;\Big(\lambda_{R_1,M_-,M_+-1,M}(\xi)\;\alpha_{R_1,M_--1,M_+  , M_+}(\xi)\Big)
                                                                                                       \label{Eq_Lem_RLRPCS_s_RCsSs_ss_Ks1_001_004g}
\end{alignat}
Recall that
by \eqref{Eq_Prp_AELRP_s_EPR_ss_AELPR_002_001c} ${\rm deg}(\lambda_{R_1,M_-,M_+-1,M})={\rm deg}(\lambda_{R_1,M_--1,M_+,M})=M$ \cite[Proposition 4.7, p. 294]{Gerolymos_2011a},
by \eqref{Eq_Prp_AELRP_s_EPR_ss_PR_001_001g} ${\rm deg}(\alpha_{R_1,M_-  ,M_+-1,-M_-})={\rm deg}(\alpha_{R_1,M_--1,M_+  , M_+})=M-1$ and
${\rm deg}(\alpha_{R_1,M_-  ,M_+  ,-M_-})={\rm deg}(\alpha_{R_1,M_-  ,M_+  ,+M_+})=M$, and
by \eqref{Eq_Prp_RLRPCS_s_FPLIR_ss_SICFPLR_001_001e} $\mathrm{deg}(\lambda_{R_1,M_--1,M_+,M}-\lambda_{R_1,M_-,M_+-1,M})=\mathrm{deg}(\alpha_{R_1,M_--1,M_+,M_+})=M-1$.

Using again \eqrefsab{Eq_Prp_RLRPCS_s_FPLIR_ss_SICFPLR_001_001a}{Eq_Prp_RLRPCS_s_FPLIR_ss_SICFPLR_001_001b} in \eqref{Eq_Lem_RLRPCS_s_RCsSs_ss_Ks1_001_003j} yields
\begin{alignat}{6}
                               &\forall \xi\in{\mathbb R}\setminus\{\xi\in{\mathbb R}:\alpha_{R_1,M_--1,M_+  , M_+}(\xi)=0\}\qquad\qquad\forall\ell\in\{-M_-+1,\cdots,M_+-1\}&
                                                                                                       \notag\\
&\alpha_{R_1,M_-,M_+,\ell}(\xi)\stackrel{\eqrefsabc{Eq_Prp_RLRPCS_s_FPLIR_ss_SICFPLR_001_001a}
                                                   {Eq_Prp_RLRPCS_s_FPLIR_ss_SICFPLR_001_001b}
                                                   {Eq_Lem_RLRPCS_s_RCsSs_ss_Ks1_001_003j}    }{=}\dfrac{\lambda_{R_1,M_--1,M_+,M}(\xi)\;\alpha_{R_1,M_-  ,M_+-1,\ell}(\xi)
                                                                                                        -\lambda_{R_1,M_-,M_+-1,M}(\xi)\;\alpha_{R_1,M_--1,M_+  ,\ell}(\xi)}
                                                                                                        {\lambda_{R_1,M_--1,M_+,M}(\xi)-\lambda_{R_1,M_-,M_+-1,M}(\xi)     }&
                                                                                                        \label{Eq_Lem_RLRPCS_s_RCsSs_ss_Ks1_001_004h}
\end{alignat}
Since the set $\{\xi\in{\mathbb R}:\alpha_{R_1,M_--1,M_+  , M_+}(\xi)=0\}$ contains only $M-1$ isolated points \prpref{Prp_RLRPCS_s_FPLIR_ss_RFPs_001},
the result of the polynomial division \eqref{Eq_Lem_RLRPCS_s_RCsSs_ss_Ks1_001_004h} must be valid $\forall\xi\in\mathbb{R}$, implying
\begin{alignat}{6}
                               &      \Big(\lambda_{R_1,M_--1,M_+,M}(\xi)-\lambda_{R_1,M_-,M_+-1,M}(\xi)\Big)\;\mid\;
                                                                                                       \notag\\
                               &      \Big(\lambda_{R_1,M_--1,M_+,M}(\xi)\;\alpha_{R_1,M_-  ,M_+-1,\ell}(\xi)
                                          -\lambda_{R_1,M_-,M_+-1,M}(\xi)\;\alpha_{R_1,M_--1,M_+  ,\ell}(\xi))\Big)
                                                                                                       \notag\\
                               &\forall\ell\in\{-M_-+1,\cdots,M_+-1\}
                                                                                                        \label{Eq_Lem_RLRPCS_s_RCsSs_ss_Ks1_001_004i}
\end{alignat}
By \eqrefsabc{Eq_Lem_RLRPCS_s_RCsSs_ss_Ks1_001_004f}{Eq_Lem_RLRPCS_s_RCsSs_ss_Ks1_001_004g}{Eq_Lem_RLRPCS_s_RCsSs_ss_Ks1_001_004i}, we have that \eqref{Eq_Lem_RLRPCS_s_RCsSs_ss_Ks1_001_004e}
is valid $\forall\xi\in{\mathbb R}$, proving \eqref{Eq_Lem_RLRPCS_s_RCsSs_ss_Ks1_001_001d}.
\end{subequations}

\begin{subequations}
                                                                                                       \label{Eq_Lem_RLRPCS_s_RCsSs_ss_Ks1_001_005}
\underline{Proof of uniqueness:} We have proved existence by construction, $\forall\xi\in{\mathbb R}\setminus\{\xi\in{\mathbb R}:\alpha_{R_1,M_--1,M_+  , M_+}(\xi)=0\}$,
of rational weighting functions, $\sigma_{R_1,M_-,M_+,1,0}(\xi)$ \eqref{Eq_Lem_RLRPCS_s_RCsSs_ss_Ks1_001_001a} and $\sigma_{R_1,M_-,M_+,1,1}(\xi)$ \eqref{Eq_Lem_RLRPCS_s_RCsSs_ss_Ks1_001_001b},
satisfying the consistency relation \eqref{Eq_Lem_RLRPCS_s_RCsSs_ss_Ks1_001_001c},
which combine the reconstructing polynomials on the substencils \defref{Def_RLRPCS_s_I_002} $\tsc{s}_{i,M_-,M_+-1}$ and $\tsc{s}_{i,M_--1,M_+}$
into the reconstructing polynomial on $\tsc{s}_{i,M_-,M_+}$ \eqref{Eq_Lem_RLRPCS_s_RCsSs_ss_Ks1_001_001e}.
To prove uniqueness, recall that by \eqref{Eq_Lem_RLRPCS_s_RCsSs_ss_Ks1_001_001c} $\sigma_{R_1,M_-,M_+,1,1}(\xi)=1-\sigma_{R_1,M_-,M_+,1,0}(\xi)$, and rewrite
\eqref{Eq_Lem_RLRPCS_s_RCsSs_ss_Ks1_001_001e} as 
\begin{alignat}{6}
p_{R_1,M_-,M_+}(x_i+\xi\Delta x;x_i,\Delta x;f)\stackrel{\eqrefsab{Eq_Lem_RLRPCS_s_RCsSs_ss_Ks1_001_001c}{Eq_Lem_RLRPCS_s_RCsSs_ss_Ks1_001_001e}}{=}
  \sigma_{R_1,M_-,M_+,1,0}(\xi)\Big(&p_{R_1,M_-  ,M_+-1}(x_i+\xi\Delta x;x_i,\Delta x;f)
                                                                                                       \notag\\
                                   -&p_{R_1,M_--1,M_+  }(x_i+\xi\Delta x;x_i,\Delta x;f)\Big)
                                                                                                       \notag\\
                                   +&p_{R_1,M_--1,M_+  }(x_i+\xi\Delta x;x_i,\Delta x;f)
                                                                                                       \label{Eq_Lem_RLRPCS_s_RCsSs_ss_Ks1_001_005a}
\end{alignat}
\begin{alignat}{6}
                                                                                           \forall \xi\in{\mathbb R}\setminus\{\xi\in{\mathbb R}:\alpha_{R_1,M_--1,M_+  , M_+}(\xi)=0\}\quad
                                                                                           \forall x_i\in{\mathbb R}\quad
                                                                                           \forall \Delta x\in{\mathbb R}_{>0}\quad
                                                                                           \forall f:{\mathbb R}\longrightarrow{\mathbb R}
                                                                                                       \notag
\end{alignat}
Hence, assuming the existence of 2 different weight-functions $\Big[\sigma_{R_1,M_-,M_+,1,0}\Big]_\tsc{a}(\xi)\neq\Big[\sigma_{R_1,M_-,M_+,1,0}\Big]_\tsc{b}(\xi)$
satisfying \eqref{Eq_Lem_RLRPCS_s_RCsSs_ss_Ks1_001_005a} $\forall \xi\in{\mathbb R}\;\forall f:{\mathbb R}\longrightarrow{\mathbb R}$ would imply
$p_{R_1,M_-  ,M_+-1}(x_i+\xi\Delta x;x_i,\Delta x;f)=p_{R_1,M_--1,M_+  }(x_i+\xi\Delta x;x_i,\Delta x;f)$ $\forall\xi\in{\mathbb R}\;\forall f:{\mathbb R}\longrightarrow{\mathbb R}$.
This is obviously a contradiction, since, by \prprefnp{Prp_RLRPCS_s_RB_001}, the 2 polynomials $p_{R_1,M_-  ,M_+-1}(x_i+\xi\Delta x;x_i,\Delta x;f)$
and $p_{R_1,M_--1,M_+  }(x_i+\xi\Delta x;x_i,\Delta x;f)$ are defined by 2 different sets of values,
$\{f_{i-M_-},\cdots,f_{i+M_+-1}\}$ and $\{f_{i-M_-+1},\cdots,f_{i+M_+}\}$,
respectively.\footnote{\label{ff_Lem_RLRPCS_s_RCsSs_ss_Ks1_001_001}A more detailed proof is given in \prprefnp{Prp_RLRPCS_s_RCsSs_ss_Ks_002}.
                      }
\end{subequations}
\qed
\end{proof}
%
%
\begin{remark}[{\rm \eqref{Eq_Lem_RLRPCS_s_RCsSs_ss_Ks1_001_001d} {\em vs} \eqref{Eq_Lem_RLRPCS_s_RCsSs_ss_Ks1_001_001e}}]
\label{Rmk_RLRPCS_s_RCsSs_ss_Ks1_001}
The expression \eqref{Eq_Lem_RLRPCS_s_RCsSs_ss_Ks1_001_001d} of the Lagrange reconstructing polynomial $p_{R_1,M_-,M_+}(x_i+\xi\Delta x;x_i,\Delta x;f)$ on $S_{i,M_-,M_+}$
is valid $\forall\xi\in\mathbb{R}$, because the rational expression \eqref{Eq_Lem_RLRPCS_s_RCsSs_ss_Ks1_001_001d} yields exactly $p_{R_1,M_-,M_+}(x_i+\xi\Delta x;x_i,\Delta x;f)$
by polynomial division. On the other hand, the weight functions $\sigma_{R_1,M_-,M_+,1,0}(\xi)$ \eqref{Eq_Lem_RLRPCS_s_RCsSs_ss_Ks1_001_001a} and $\sigma_{R_1,M_-,M_+,1,1}(\xi)$ \eqref{Eq_Lem_RLRPCS_s_RCsSs_ss_Ks1_001_001b},
are not defined at the poles of the rational expressions \eqrefsab{Eq_Lem_RLRPCS_s_RCsSs_ss_Ks1_001_001a}{Eq_Lem_RLRPCS_s_RCsSs_ss_Ks1_001_001b},
where the representation \eqref{Eq_Lem_RLRPCS_s_RCsSs_ss_Ks1_001_001e} is not possible.
\qed
\end{remark}
%
%
\begin{corollary}[{\rm Identities for $(K_{\rm s}=1)$-level subdivision}]
\label{Crl_RLRPCS_s_RCsSs_ss_Ks1_001}
Assume the conditions and definitions of \lemrefnp{Lem_RLRPCS_s_RCsSs_ss_Ks1_001}. Then the following identities hold
\begin{subequations}
                                                                                                       \label{Eq_Crl_RLRPCS_s_RCsSs_ss_Ks1_001_001}
\begin{alignat}{6}
\alpha_{R_1,M_-  ,M_+  ,-M_-}(\xi)=&\sigma_{R_1,M_-,M_+,1,0}(\xi)\;\alpha_{R_1,M_-  ,M_+-1,-M_-}(\xi)
                                                                                                       \label{Eq_Crl_RLRPCS_s_RCsSs_ss_Ks1_001_001a}\\
\alpha_{R_1,M_-  ,M_+  ,\ell}(\xi)=&\sigma_{R_1,M_-,M_+,1,0}(\xi)\;\alpha_{R_1,M_-  ,M_+-1,\ell}(\xi)
                                                                                                       \notag\\
                                  +&\sigma_{R_1,M_-,M_+,1,1}(\xi)\;\alpha_{R_1,M_--1,M_+  ,\ell}(\xi)
                                  \qquad\forall\ell\in\{-M_-+1,\cdots,M_+-1\}
                                                                                                       \label{Eq_Crl_RLRPCS_s_RCsSs_ss_Ks1_001_001b}\\
\alpha_{R_1,M_-  ,M_+  ,+M_+}(\xi)=&\sigma_{R_1,M_-,M_+,1,1}(\xi)\;\alpha_{R_1,M_--1,M_+  ,+M_+}(\xi)
                                                                                                       \label{Eq_Crl_RLRPCS_s_RCsSs_ss_Ks1_001_001c}\\
                               0=&\sigma_{R_1,M_-,M_+,1,0}(\xi)\;\lambda_{R_1,M_-  ,M_+-1,M}(\xi)
                                + \sigma_{R_1,M_-,M_+,1,1}(\xi)\;\lambda_{R_1,M_--1,M_+  ,M}(\xi)
                                                                                                       \label{Eq_Crl_RLRPCS_s_RCsSs_ss_Ks1_001_001d}\\
\lambda_{R_1,M_-  ,M_+  ,n}(\xi)=&\sigma_{R_1,M_-,M_+,1,0}(\xi)\;\lambda_{R_1,M_-  ,M_+-1,n}(\xi)
                                + \sigma_{R_1,M_-,M_+,1,1}(\xi)\;\lambda_{R_1,M_--1,M_+  ,n}(\xi)
                                  \qquad\forall n\geq M+1
                                                                                                       \label{Eq_Crl_RLRPCS_s_RCsSs_ss_Ks1_001_001e}\\
&\forall \xi\in{\mathbb R}\setminus\{\xi\in{\mathbb R}:\alpha_{R_1,M_--1,M_+  , M_+}(\xi)=0\}
                                                                                                       \notag
\end{alignat}
\end{subequations}
\end{corollary}
%
%
\begin{proof}
We have already proved
\eqref{Eq_Crl_RLRPCS_s_RCsSs_ss_Ks1_001_001a} as \eqref{Eq_Lem_RLRPCS_s_RCsSs_ss_Ks1_001_003g},
\eqref{Eq_Crl_RLRPCS_s_RCsSs_ss_Ks1_001_001c} as \eqref{Eq_Lem_RLRPCS_s_RCsSs_ss_Ks1_001_003h}, and
\eqref{Eq_Crl_RLRPCS_s_RCsSs_ss_Ks1_001_001b} as \eqref{Eq_Lem_RLRPCS_s_RCsSs_ss_Ks1_001_003j}. They are summarized separately here for future use.
Identity \eqref{Eq_Crl_RLRPCS_s_RCsSs_ss_Ks1_001_001d} follows directly from the definitions of
$\sigma_{R_1,M_-,M_+,1,0}(\xi)$ \eqref{Eq_Lem_RLRPCS_s_RCsSs_ss_Ks1_001_001a} and
$\sigma_{R_1,M_-,M_+,1,1}(\xi)$ \eqref{Eq_Lem_RLRPCS_s_RCsSs_ss_Ks1_001_001b}, and was used in the calculations leading to \eqref{Eq_Lem_RLRPCS_s_RCsSs_ss_Ks1_001_003d}.
To prove the relation \eqref{Eq_Crl_RLRPCS_s_RCsSs_ss_Ks1_001_001e}, we replace 
$p_{R_1,M_-,M_+}(x_i+\xi\Delta x;x_i,\Delta x;f)$, $p_{R_1,M_-  ,M_+-1}(x_i+\xi\Delta x;x_i,\Delta x;f)$ and $p_{R_1,M_--1,M_+  }(x_i+\xi\Delta x;x_i,\Delta x;f)$ in \eqref{Eq_Lem_RLRPCS_s_RCsSs_ss_Ks1_001_001e}
by their expansions in terms of the derivatives $h^{(n)}(x_i+\xi\Delta x)$ \eqref{Eq_Prp_AELRP_s_EPR_ss_AELPR_002_001a}, and obtain,
using \eqrefsab{Eq_Lem_RLRPCS_s_RCsSs_ss_Ks1_001_001c}{Eq_Crl_RLRPCS_s_RCsSs_ss_Ks1_001_001d}
\begin{alignat}{6}
\sum_{n=M+1}^{N_\tsc{tj}}\Big(\sigma_{R_1,M_-,M_+,1,0}(\xi)\;\lambda_{R_1,M_-  ,M_+-1,n}(\xi)
                             +\sigma_{R_1,M_-,M_+,1,1}(\xi)\;\lambda_{R_1,M_--1,M_+  ,n}(\xi)\Big)\;\Delta x^n\;h^{(n)}(x+\xi\Delta x)
\stackrel{\eqrefsab{Eq_Lem_RLRPCS_s_RCsSs_ss_Ks1_001_001c}{Eq_Crl_RLRPCS_s_RCsSs_ss_Ks1_001_001d}}{=}O(\Delta x^{N_\tsc{tj}+1})
                                                                                                       \label{Eq_Crl_RLRPCS_s_RCsSs_ss_Ks1_001_002}\\
\forall \xi\in{\mathbb R}\setminus\{\xi\in{\mathbb R}:\alpha_{R_1,M_--1,M_+  , M_+}(\xi)=0\}\qquad\qquad\forall h\in C^{N_\tsc{tj}+1}({\mathbb R}).
                                                                                                       \notag
\end{alignat}
Using polynomials $q(x)\in \mathbb{R}_n[x]$ \rmkref{Rmk_RLRPCS_s_RB_002}, recursively for $n\geq M+1$, in \eqref{Eq_Crl_RLRPCS_s_RCsSs_ss_Ks1_001_002} proves \eqref{Eq_Crl_RLRPCS_s_RCsSs_ss_Ks1_001_001e}, by induction.
\qed
\end{proof}
%

%
%
%
%
%
\subsection{$(K_\mathrm{s}\geq1)$-level subdivision}\label{RLRPCS_s_RCsSs_ss_Ks}
%
%
%
%
%

We have shown in \cite[Lemma 2.1]{Gerolymos_2011a_news} that if a general family of functions $p_{M_-,M_+}(x)$, depending on 2 integer indices $M_\pm\in\mathbb{Z}:M_-+M_+\geq1$,
is equipped with a $(K_\mathrm{s}=1)$-level subdivision rule, defined by a relation of the form \eqref{Eq_RLRPCS_s_I_001a} with $K_\mathrm{s}=1$,
then, by recurrence, we can construct weight-functions satisfying \eqref{Eq_RLRPCS_s_I_001a} $\forall K_\mathrm{s}\leq M-1$.
By \lemrefnp{Lem_RLRPCS_s_RCsSs_ss_Ks1_001}, we can always define uniquely the optimal weight-functions \eqref{Eq_Lem_RLRPCS_s_RCsSs_ss_Ks1_001_001} of the
$(K_\mathrm{s}=1)$-level subdivision of $\tsc{s}_{i,M_-,M_+}$ \defref{Def_RLRPCS_s_I_002}.
Therefore, the Lagrange reconstructing polynomials $p_{R_1,M_-,M_+}(x_i+\xi\Delta x;x_i,\Delta x;f)$ \eqref{Eq_Prp_AELRP_s_EPR_ss_PR_001_001d}
satisfy the conditions of \cite[Lemma 2.1]{Gerolymos_2011a_news}.

%
\begin{proposition}[{\rm Recursive generation of weight-functions for the Lagrange reconstructing polynomial}]
\label{Prp_RLRPCS_s_RCsSs_ss_Ks_001}
\begin{subequations}
                                                                                                       \label{Eq_Prp_RLRPCS_s_RCsSs_ss_Ks_001_001}
Assume the conditions of \lemrefnp{Lem_RLRPCS_s_RCsSs_ss_Ks1_001}. Then, $\forall M_\pm\in{\mathbb Z}:\;M:=M_-+M_+\geq2$, $\forall K_\mathrm{s}\leq M-1$,
the reconstructing polynomial on $\tsc{s}_{i,M_-,M_+}$ \prpref{Prp_RLRPCS_s_RB_001}
can be represented, almost everywhere, by combination of the reconstructing polynomials on the $K_\mathrm{s}$-level substencils \defref{Def_RLRPCS_s_I_002} of $\tsc{s}_{i,M_-,M_+}$, as
\begin{alignat}{6}
 &p_{R_1,M_-,M_+}(x_i+\xi\Delta x;x_i,\Delta x;f)=\sum_{k_\mathrm{s}=0}^{K_\mathrm{s}}\sigma_{R_1,M_-,M_+,K_\mathrm{s},k_\mathrm{s}}(\xi)\;p_{R_1,M_--k_\mathrm{s},M_+-K_\mathrm{s}+k_\mathrm{s}}(x_i+\xi\Delta x;x_i,\Delta x;f)
                                                                                                       \label{Eq_Prp_RLRPCS_s_RCsSs_ss_Ks_001_001a}\\
                                                                                          &\forall \xi\in{\mathbb R}\setminus\mathcal{S}_{R_1,M_-,M_+,K_\mathrm{s}}\quad
                                                                                           \forall x_i\in{\mathbb R}\quad
                                                                                           \forall \Delta x\in{\mathbb R}_{>0}\quad
                                                                                           \forall f:{\mathbb R}\longrightarrow{\mathbb R}
                                                                                                       \notag
\end{alignat}
where the rational weight-functions $\sigma_{R_1,M_-,M_+,K_\mathrm{s},k_\mathrm{s}}(\xi)$ are defined recursively by
\begin{alignat}{6}
&\sigma_{R_1,M_-,M_+,K_\mathrm{s},k_\mathrm{s}}(\xi)=\left\{\begin{array}{ll}\dfrac{\alpha_{R_1,M_-             ,M_+               ,-M_-+k_\mathrm{s}M}(\xi)}
                                                                                   {\alpha_{R_1,M_--k_\mathrm{s},M_+-1+k_\mathrm{s},-M_-+k_\mathrm{s}M}(\xi)}      &K_\mathrm{s}=1       \\
                                                                                                                                                                   &                     \\
                                                                             {\displaystyle\sum_{\ell_\mathrm{s}=\max(0,k_\mathrm{s}-1)}^{\min(K_\mathrm{s}-1,k_\mathrm{s})}
                                                                                           \sigma_{R_1,M_-             ,M_+,K_\mathrm{s}-1                 ,\ell_\mathrm{s}}(\xi)\;
                                                                                           \sigma_{R_1,M_--\ell_{\rm s},M_+-(K_\mathrm{s}-1)+\ell_\mathrm{s},1,k_\mathrm{s}-\ell_\mathrm{s}}(\xi)
                                                                             }                                                                                     &K_\mathrm{s}\geq2    \\\end{array}\right.
                                                                                                       \notag\\
&\forall k_\mathrm{s}\in\{0,\cdots,K_\mathrm{s}\}\qquad\forall K_\mathrm{s}\in\{1,\cdots,M-1\}
                                                                                                       \label{Eq_Prp_RLRPCS_s_RCsSs_ss_Ks_001_001b}
\end{alignat}
and satisfy the consistency condition
\begin{alignat}{6}
 \sum_{k_\mathrm{s}=0}^{K_\mathrm{s}}\sigma_{R_1,M_-,M_+,K_\mathrm{s},k_\mathrm{s}}(\xi)=1\qquad\forall\xi\in{\mathbb R}
                                                                                                       \label{Eq_Prp_RLRPCS_s_RCsSs_ss_Ks_001_001c}
\end{alignat}
The set of poles of the rational weight-functions $\mathcal{S}_{R_1,M_-,M_+,K_\mathrm{s}}$ \eqref{Eq_Prp_RLRPCS_s_RCsSs_ss_Ks_001_001a} satisfies
\begin{alignat}{6}
\mathcal{S}_{R_1,M_-,M_+,1           }                  &:=     &\{\xi\in{\mathbb R}:\alpha_{R_1,M_--1,M_+  , M_+}(\xi)=0\}
\stackrel{\eqref{Eq_Prp_RLRPCS_s_FPLIR_ss_RFPs_001_001}}{=}      \{\xi_{R_1,M_--1,M_+,M_+,n};\;n\in\{-M_-+1,\cdots,M_+-1\}
                                                                                                       \label{Eq_Prp_RLRPCS_s_RCsSs_ss_Ks_001_001d}\\
\mathcal{S}_{R_1,M_-,M_+,K_\mathrm{s}}&\subseteq&\bigcup_{   L_\mathrm{s}=0}^{K_\mathrm{s}-1}\displaylimits
                                                 \bigcup_{\ell_\mathrm{s}=0}^{L_\mathrm{s}  }{\mathcal S}_{\sigma_{M_--\ell_\mathrm{s},M_+-L_\mathrm{s}+\ell_\mathrm{s},1}}
                                       =         \left\{\xi\in{\mathbb R}:\prod_{   L_\mathrm{s}=0}^{K_\mathrm{s}-1}
                                                                          \prod_{\ell_\mathrm{s}=0}^{L_\mathrm{s}  }
                                                                          \alpha_{R_1,M_--1-\ell_\mathrm{s},M_+-L_\mathrm{s}+\ell_\mathrm{s},M_+-L_\mathrm{s}+\ell_\mathrm{s}}(\xi)=0\right\}
                                                                                                       \notag\\
&&\forall K_\mathrm{s}\in\{1,\cdots,M-1\}
                                                                                                       \label{Eq_Prp_RLRPCS_s_RCsSs_ss_Ks_001_001e}
\end{alignat}
\end{subequations}
\end{proposition}
%
%
\begin{proof}
The case $K_{\rm s}=1$ follows from \lemrefnp{Lem_RLRPCS_s_RCsSs_ss_Ks1_001}, with the set of isolated singular points $\mathcal{S}_{R_1,M_-,M_+,1}$ defined by \eqref{Eq_Prp_RLRPCS_s_RCsSs_ss_Ks_001_001d},
because of \eqref{Eq_Lem_RLRPCS_s_RCsSs_ss_Ks1_001_001e}. Since the $(K_{\rm s}=1)$-level subdivision rule is established, the conditions
of \cite[Lemma 2.1]{Gerolymos_2011a_news} are satisfied, proving \eqrefsatob{Eq_Prp_RLRPCS_s_RCsSs_ss_Ks_001_001a}{Eq_Prp_RLRPCS_s_RCsSs_ss_Ks_001_001c},
and the recursive definition \eqref{Eq_Prp_RLRPCS_s_RCsSs_ss_Ks_001_001e} of the set of isolated singular points $\mathcal{S}_{R_1,M_-,M_+,K_\mathrm{s}}$.
The $\subseteq$ relation is used in \eqref{Eq_Prp_RLRPCS_s_RCsSs_ss_Ks_001_001e} for $K_\mathrm{s}>1$,
because there may be pole cancellation by the multiplications in \eqref{Eq_Prp_RLRPCS_s_RCsSs_ss_Ks_001_001b}.
\qed
\end{proof}
%
%
\begin{corollary}[{\rm Representation of the fundamental polynomials $\alpha_{R_1,M_-,M_+,\ell}(\xi)$~\eqref{Eq_Prp_AELRP_s_EPR_ss_PR_001_001g}}] \label{Crl_RLRPCS_s_RCsSs_ss_Ks_001}
Assume the conditions of \prprefnp{Prp_RLRPCS_s_RCsSs_ss_Ks_001}. Then, the fundamental polynomials of Lagrange reconstruction
$\alpha_{R_1,M_-,M_+,\ell}(\xi)$ \eqref{Eq_Prp_AELRP_s_EPR_ss_PR_001_001g} on $\tsc{s}_{i,M_-,M_+}$ \defref{Def_AELRP_s_EPR_ss_PR_001},
can be represented by a weighted combination of the basis \prpref{Prp_RLRPCS_s_FPLIR_ss_RPsFPs_002} Lagrange reconstructing polynomials
on the $K_\mathrm{s}$-level substencils \defref{Def_RLRPCS_s_I_002} of $\tsc{s}_{i,M_-,M_+}$ as
\begin{alignat}{6}
\alpha_{R_1,M_-,M_+,\ell}(\xi)=&\sum_{k_\mathrm{s}=\max(0,\ell+K_\mathrm{s}-M_+)}^{\min(K_\mathrm{s},\ell+M_-)}
                                \sigma_{R_1,M_-,M_+,K_\mathrm{s},k_\mathrm{s}}(\xi)\;
                                \alpha_{R_1,M_--k_\mathrm{s},M_+-K_\mathrm{s}+k_\mathrm{s},\ell}(\xi)
                                \qquad\forall \xi\in{\mathbb R}\setminus\mathcal{S}_{R_1,M_-,M_+,K_\mathrm{s}}
                                                                                                       \label{Eq_Crl_RLRPCS_s_RCsSs_ss_Ks_001_001}
\end{alignat}
where the weight-functions $\sigma_{R_1,M_-,M_+,K_\mathrm{s},k_\mathrm{s}}(\xi)$ are defined by \eqref{Eq_Prp_RLRPCS_s_RCsSs_ss_Ks_001_001b} in \prprefnp{Prp_RLRPCS_s_RCsSs_ss_Ks_001},
and the set of isolated singular points $\mathcal{S}_{R_1,M_-,M_+,K_\mathrm{s}}$ by \eqrefsab{Eq_Prp_RLRPCS_s_RCsSs_ss_Ks_001_001d}
                                                                                             {Eq_Prp_RLRPCS_s_RCsSs_ss_Ks_001_001e}.
\end{corollary}
%
%
\begin{proof} Rewrite \eqref{Eq_Prp_RLRPCS_s_RCsSs_ss_Ks_001_001a}
as\footnote{\label{ff_Crl_RLRPCS_s_RCsSs_ss_Ks_001_001}
\begin{alignat}{6}
\begin{array}{c}0\leq k_\mathrm{s}\leq K_\mathrm{s}                        \\
                -M_-+k_\mathrm{s}\leq\ell\leq M_+-K_\mathrm{s}+k_\mathrm{s}\\\end{array}\stackrel{\eqref{Eq_Def_RLRPCS_s_I_002_001b}}{\iff}
\begin{array}{c}0\leq k_\mathrm{s}\leq K_\mathrm{s}                        \\
                -M_-\leq\ell\leq M_+                                       \\
                -M_-+k_\mathrm{s}\leq\ell\leq M_+-K_\mathrm{s}+k_\mathrm{s}\\\end{array}\iff
\begin{array}{c}0\leq k_\mathrm{s}\leq K_\mathrm{s}\\
                -M_-\leq\ell\leq M_+               \\
                k_\mathrm{s}\leq\ell+M_-           \\
                \ell-M_++K_\mathrm{s}\leq k_\mathrm{s}\\\end{array}\iff
\begin{array}{c}-M_-\leq\ell\leq M_+                              \\
                0\leq k_\mathrm{s}\leq K_\mathrm{s}               \\
                \ell-M_++K_\mathrm{s}\leq k_\mathrm{s}\leq\ell+M_-\\\end{array}
                                                                                                       \notag
\end{alignat}
           }
\begin{alignat}{6}
p_{R_1,M_-,M_+}(x_i+\xi\Delta x;x_i,\Delta x;f)
\stackrel{\eqref{Eq_Prp_RLRPCS_s_RCsSs_ss_Ks_001_001a}}{=}&\sum_{k_\mathrm{s}=0}^{K_\mathrm{s}}\sigma_{R_1,M_-,M_+,K_\mathrm{s},k_\mathrm{s}}(\xi)\;p_{R_1,M_--k_\mathrm{s},M_+-K_\mathrm{s}+k_\mathrm{s}}(x_i+\xi\Delta x;x_i,\Delta x;f)
                                                                                                       \notag\\
\stackrel{\eqref{Eq_Prp_AELRP_s_EPR_ss_PR_001_001d}}{=}&\sum_{k_\mathrm{s}=0}^{K_\mathrm{s}}\left(\sigma_{R_1,M_-,M_+,K_\mathrm{s},k_\mathrm{s}}(\xi)\;
                                                                                           \sum_{\ell=-M_-+k_\mathrm{s}}^{M_+-K_\mathrm{s}+k_\mathrm{s}}
                                                                                           \Bigl(\alpha_{R_1,M_--k_\mathrm{s},M_+-K_\mathrm{s}+k_\mathrm{s},\ell}(\xi)\;f(x_i+\ell\Delta x)\Bigr)\right)
                                                                                                       \notag\\
\stackrel{\ffref{ff_Crl_RLRPCS_s_RCsSs_ss_Ks_001_001}}{=}&\sum_{\ell=-M_-}^{M_+}\underbrace{\left(\sum_{k_\mathrm{s}=\max(0,\ell+K_\mathrm{s}-M_+)}^{\min(K_\mathrm{s},\ell+M_-)}
                                                                                                  \sigma_{R_1,M_-,M_+,K_\mathrm{s},k_\mathrm{s}}(\xi)\;
                                                                                                   \alpha_{R_1,M_--k_\mathrm{s},M_+-K_\mathrm{s}+k_\mathrm{s},\ell}(\xi)\right)
                                                                                           }_{\displaystyle \alpha_{R_1,M_-,M_+,\ell}(\xi)}\;f(x_i+\ell\Delta x)
                                                                                                       \label{Eq_Crl_RLRPCS_s_RCsSs_ss_Ks_001_002}
\end{alignat}
proving \eqref{Eq_Crl_RLRPCS_s_RCsSs_ss_Ks_001_001} by \eqref{Eq_Prp_AELRP_s_EPR_ss_PR_001_001d}.
\qed
\end{proof}
%
%
\begin{proposition}[{\rm Uniqueness of weight-functions $\sigma_{R_1,M_-,M_+,K_\mathrm{s},k_\mathrm{s}}(\xi)$ \eqref{Eq_Prp_RLRPCS_s_RCsSs_ss_Ks_001_001b}}]
\label{Prp_RLRPCS_s_RCsSs_ss_Ks_002}
Assume the conditions of \prprefnp{Prp_RLRPCS_s_RCsSs_ss_Ks_001}.
The functions $\sigma_{R_1,M_-,M_+,K_\mathrm{s},k_\mathrm{s}}(\xi)$ satisfying \eqref{Eq_Prp_RLRPCS_s_RCsSs_ss_Ks_001_001a} are unique.
\end{proposition}
%
%
\begin{proof}
\begin{subequations}
                                                                                                       \label{Eq_Prp_RLRPCS_s_RCsSs_ss_Ks_002_001}
We have proved by construction \eqref{Eq_Prp_RLRPCS_s_RCsSs_ss_Ks_001_001b} the existence of weight-functions $\sigma_{R_1,M_-,M_+,K_\mathrm{s},k_\mathrm{s}}(\xi)$
satisfying \eqrefsab{Eq_Prp_RLRPCS_s_RCsSs_ss_Ks_001_001a}{Eq_Prp_RLRPCS_s_RCsSs_ss_Ks_001_001c}. Uniqueness for the case $K_\mathrm{s}=1$ was proved
in \lemrefnp{Lem_RLRPCS_s_RCsSs_ss_Ks1_001}. Notice first that \crlrefnp{Crl_RLRPCS_s_RCsSs_ss_Ks_001} does not require the validity of the particular expression \eqref{Eq_Prp_RLRPCS_s_RCsSs_ss_Ks_001_001b}
of the weight-functions, and is therefore valid for any set of weight-functions $\sigma_{R_1,M_-,M_+,K_\mathrm{s},k_\mathrm{s}}(\xi)$ satisfying \eqref{Eq_Prp_RLRPCS_s_RCsSs_ss_Ks_001_001a}.
To prove therefore uniqueness we can use \eqref{Eq_Crl_RLRPCS_s_RCsSs_ss_Ks_001_001}, which can be explicitly written as
\begin{alignat}{6}
\forall K_\mathrm{s}\geq 1        \;&\;\alpha_{R_1,M_-,+M_+,-M_-  }(\xi)&=&\sigma_{R_1,M_-,M_+,K_\mathrm{s},0}(\xi)\;\alpha_{R_1,M_-  ,M_+-K_\mathrm{s}  ,-M_-}(\xi)
                                                                                                       \label{Eq_Prp_RLRPCS_s_RCsSs_ss_Ks_002_001a}\\
\forall K_\mathrm{s}\geq 1        \;&\;\alpha_{R_1,M_-,+M_+,-M_-+1}(\xi)&=&\sigma_{R_1,M_-,M_+,K_\mathrm{s},0}(\xi)\;\alpha_{R_1,M_-  ,M_+-K_\mathrm{s}  ,-M_-+1}(\xi)
                                                                                                       \notag\\
                                  \;&\;                                 &+&\sigma_{R_1,M_-,M_+,K_\mathrm{s},1}(\xi)\;\alpha_{R_1,M_-+1,M_+-K_\mathrm{s}+1,-M_-+1}(\xi)
                                                                                                       \label{Eq_Prp_RLRPCS_s_RCsSs_ss_Ks_002_001b}\\
\forall K_\mathrm{s}\geq 2        \;&\;\alpha_{R_1,M_-,+M_+,-M_-+2}(\xi)&=&\sigma_{R_1,M_-,M_+,K_\mathrm{s},0}(\xi)\;\alpha_{R_1,M_-  ,M_+-K_\mathrm{s}  ,-M_-+2}(\xi)
                                                                                                       \notag\\
                                  \;&\;                                 &+&\sigma_{R_1,M_-,M_+,K_\mathrm{s},1}(\xi)\;\alpha_{R_1,M_-+1,M_+-K_\mathrm{s}+1,-M_-+2}(\xi)
                                                                                                       \notag\\
                                  \;&\;                                 &+&\sigma_{R_1,M_-,M_+,K_\mathrm{s},2}(\xi)\;\alpha_{R_1,M_-+2,M_+-K_\mathrm{s}+2,-M_-+2}(\xi)
                                                                                                       \label{Eq_Prp_RLRPCS_s_RCsSs_ss_Ks_002_001c}\\
                                  \;&\;                                 & &\vdots
                                                                                                       \notag\\
\forall\ell\in\{K_\mathrm{s}-M_-,
                M_+-K_\mathrm{s}\}\;&\;\alpha_{R_1,M_-,M_+,\ell}(\xi)   &=&\sum_{k_\mathrm{s}=\max(0,\ell+K_\mathrm{s}-M_+)}^{\min(K_\mathrm{s},\ell+M_-)}
                                                                           \sigma_{R_1,M_-,M_+,K_\mathrm{s},k_\mathrm{s}}(\xi)\;
                                                                           \alpha_{R_1,M_--k_\mathrm{s},M_+M_+-K_\mathrm{s}+k_\mathrm{s},\ell}(\xi)
                                                                                                       \label{Eq_Prp_RLRPCS_s_RCsSs_ss_Ks_002_001d}\\
                                  \;&\;                                 & &\vdots
                                                                                                       \notag\\
\forall K_\mathrm{s}\geq 2        \;&\;\alpha_{R_1,M_-,+M_+,+M_+-1}(\xi)&=&\sigma_{R_1,M_-,M_+,K_\mathrm{s},K_\mathrm{s}-2}(\xi)\;\alpha_{R_1,M_--K_\mathrm{s}+2,M_+-2,+M_+-2}(\xi)
                                                                                                       \notag\\
                                  \;&\;                                 &+&\sigma_{R_1,M_-,M_+,K_\mathrm{s},K_\mathrm{s}-1}(\xi)\;\alpha_{R_1,M_--K_\mathrm{s}+1,M_+-1,+M_+-2}(\xi)
                                                                                                       \notag\\
                                  \;&\;                                 &+&\sigma_{R_1,M_-,M_+,K_\mathrm{s},K_\mathrm{s}  }(\xi)\;\alpha_{R_1,M_--K_\mathrm{s}  ,M_+  ,+M_+-2}(\xi)
                                                                                                       \label{Eq_Prp_RLRPCS_s_RCsSs_ss_Ks_002_001e}\\
\forall K_\mathrm{s}\geq 1        \;&\;\alpha_{R_1,M_-,+M_+,+M_+-1}(\xi)&=&\sigma_{R_1,M_-,M_+,K_\mathrm{s},K_\mathrm{s}-1}(\xi)\;\alpha_{R_1,M_--K_\mathrm{s}+1,M_+-1,+M_+-1}(\xi)
                                                                                                       \notag\\
                                  \;&\;                                 &+&\sigma_{R_1,M_-,M_+,K_\mathrm{s},K_\mathrm{s}  }(\xi)\;\alpha_{R_1,M_--K_\mathrm{s}  ,M_+  ,+M_+-1}(\xi)
                                                                                                       \label{Eq_Prp_RLRPCS_s_RCsSs_ss_Ks_002_001f}\\
\forall K_\mathrm{s}\geq 1        \;&\;\alpha_{R_1,M_-,+M_+,+M_+  }(\xi)&=&\sigma_{R_1,M_-,M_+,K_\mathrm{s},K_\mathrm{s}  }(\xi)\;\alpha_{R_1,M_--K_\mathrm{s}  ,M_+  ,+M_+  }(\xi)
                                                                                                       \label{Eq_Prp_RLRPCS_s_RCsSs_ss_Ks_002_001g}
\end{alignat}
Starting with \eqrefsab{Eq_Prp_RLRPCS_s_RCsSs_ss_Ks_002_001a}{Eq_Prp_RLRPCS_s_RCsSs_ss_Ks_002_001g} we immediately prove uniqueness of $\sigma_{R_1,M_-,M_+,K_\mathrm{s},0}(\xi)$
and $\sigma_{R_1,M_-,M_+,K_\mathrm{s},K_\mathrm{s}}(\xi)$, by contradiction because of \eqref{Eq_Lem_RLRPCS_s_FPLIR_ss_RPsFPs_001_001b}.
Having proved uniqueness of $\sigma_{R_1,M_-,M_+,K_\mathrm{s},0}(\xi)$, \eqref{Eq_Prp_RLRPCS_s_RCsSs_ss_Ks_002_001b}
proves uniqueness of $\sigma_{R_1,M_-,M_+,K_\mathrm{s},1}(\xi)$, by contradiction because of \eqref{Eq_Lem_RLRPCS_s_FPLIR_ss_RPsFPs_001_001b}.
In exactly the same way, having proved uniqueness of $\sigma_{R_1,M_-,M_+,K_\mathrm{s},K_\mathrm{s}  }(\xi)$,
\eqref{Eq_Prp_RLRPCS_s_RCsSs_ss_Ks_002_001f} proves uniqueness of $\sigma_{R_1,M_-,M_+,K_\mathrm{s},K_\mathrm{s}-1}(\xi)$. Continuing the procedure until reaching 
$\sigma_{R_1,M_-,M_+,K_\mathrm{s},\left\lceil\frac{K_\mathrm{s}}{2}\right\rceil}(\xi)$ (for increasing $k_\mathrm{s}$, starting from \eqref{Eq_Prp_RLRPCS_s_RCsSs_ss_Ks_002_001a})
and $\sigma_{R_1,M_-,M_+,K_\mathrm{s},\left\lfloor\frac{K_\mathrm{s}}{2}\right\rfloor}(\xi)$ (for decreasing $k_\mathrm{s}$, starting from \eqref{Eq_Prp_RLRPCS_s_RCsSs_ss_Ks_002_001g}),
completes the proof of uniqueness.
\end{subequations}
\qed
\end{proof}
%
%
\begin{corollary}[{\rm Weight-functions and approximation-errors}]
\label{Crl_RLRPCS_s_RCsSs_ss_Ks_002}
\begin{subequations}
                                                                                                       \label{Eq_Crl_RLRPCS_s_RCsSs_ss_Ks_002_001}
Assume the conditions of \prprefnp{Prp_RLRPCS_s_RCsSs_ss_Ks_001}. Then, provided that the reconstruction pair \defref{Def_AELRP_s_RPERR_ss_RP_001} of $f(x)$, $h(x):=[R_{(1;\Delta x)}(f)](x)$,
is sufficiently smooth $\forall x\in[x_{i-M_-}-\tfrac{1}{2}\Delta x,x_{i+M_+}+\tfrac{1}{2}\Delta x]$,
for the expansions \eqrefsab{Eq_Prp_AELRP_s_EPR_ss_AELPR_001_001c}{Eq_Prp_AELRP_s_EPR_ss_AELPR_002_001a} of the approximation error
to hold, the weight-functions $\sigma_{R_1,M_-,M_+,K_\mathrm{s},k_\mathrm{s}}(\xi)$ \eqref{Eq_Prp_RLRPCS_s_RCsSs_ss_Ks_001_001b} and the approximation-error polynomials
$\mu_{R_1,M_-,M_+,n}(\xi)$ \eqref{Eq_Prp_AELRP_s_EPR_ss_AELPR_001_001f} and $\lambda_{R_1,M_-,M_+,n}(\xi)$ \eqref{Eq_Prp_AELRP_s_EPR_ss_AELPR_002_001c}, satisfy
\begin{alignat}{6}
\sum_{k_\mathrm{s}=0}^{K_\mathrm{s}}\sigma_{R_1,M_-,M_+,K_\mathrm{s},k_\mathrm{s}}(\xi)\;\lambda_{R_1,M_--k_\mathrm{s},M_+-K_\mathrm{s}+k_\mathrm{s},n}(\xi)=&0\qquad&\qquad\forall n\in\{M-K_\mathrm{s}+1,M\}
                                                                                                       \label{Eq_Crl_RLRPCS_s_RCsSs_ss_Ks_002_001a}\\
\sum_{k_\mathrm{s}=0}^{K_\mathrm{s}}\sigma_{R_1,M_-,M_+,K_\mathrm{s},k_\mathrm{s}}(\xi)\;\lambda_{R_1,M_--k_\mathrm{s},M_+-K_\mathrm{s}+k_\mathrm{s},n}(\xi)=&\lambda_{R_1,M_-,M_+,n}(\xi)\qquad&\qquad\forall n\geq M+1
                                                                                                       \label{Eq_Crl_RLRPCS_s_RCsSs_ss_Ks_002_001b}\\
\sum_{k_\mathrm{s}=0}^{K_\mathrm{s}}\sigma_{R_1,M_-,M_+,K_\mathrm{s},k_\mathrm{s}}(\xi)\;    \mu_{R_1,M_--k_\mathrm{s},M_+-K_\mathrm{s}+k_\mathrm{s},n}(\xi)=&0\qquad&\qquad\forall n\in\{M-K_\mathrm{s}+1,M\}
                                                                                                       \label{Eq_Crl_RLRPCS_s_RCsSs_ss_Ks_002_001c}\\
\sum_{k_\mathrm{s}=0}^{K_\mathrm{s}}\sigma_{R_1,M_-,M_+,K_\mathrm{s},k_\mathrm{s}}(\xi)\;    \mu_{R_1,M_--k_\mathrm{s},M_+-K_\mathrm{s}+k_\mathrm{s},n}(\xi)=&    \mu_{R_1,M_-,M_+,n}(\xi)\qquad&\qquad\forall n\geq M+1
                                                                                                       \label{Eq_Crl_RLRPCS_s_RCsSs_ss_Ks_002_001d}\\
                                                                                           \forall \xi\in{\mathbb R}\setminus\mathcal{S}_{R_1,M_-,M_+,K_\mathrm{s}}
                                                                                                       \notag
\end{alignat}
where the set of isolated singular points $\mathcal{S}_{R_1,M_-,M_+,K_\mathrm{s}}$ is defined by \eqrefsab{Eq_Prp_RLRPCS_s_RCsSs_ss_Ks_001_001d}
                                                                                                          {Eq_Prp_RLRPCS_s_RCsSs_ss_Ks_001_001e}.
\end{subequations}
\end{corollary}
%
%
\begin{proof}
\begin{subequations}
                                                                                                       \label{Eq_Crl_RLRPCS_s_RCsSs_ss_Ks_002_002}
The proof is quite obvious by replacing $p_{R_1,M_-,M_+}(x_i+\xi\Delta x;x_i,\Delta x;f)$ and $p_{R_1,M_--k_\mathrm{s},M_+-K_\mathrm{s}+k_\mathrm{s}}(x_i+\xi\Delta x;x_i,\Delta x;f)$
in \eqref{Eq_Prp_RLRPCS_s_RCsSs_ss_Ks_001_001a} by either \eqref{Eq_Prp_AELRP_s_EPR_ss_AELPR_001_001c} or \eqref{Eq_Prp_AELRP_s_EPR_ss_AELPR_002_001a}, yielding
\begin{alignat}{6}
p_{R_1,M_-,M_+}(x_i+\xi\Delta x;x_i,\Delta x;f)\stackrel{\eqref{Eq_Prp_AELRP_s_EPR_ss_AELPR_002_001a}}{=}&
                                               h(x_i+\xi\Delta x)+\sum_{n=M+1}^{N_\tsc{tj}}\lambda_{R_1,M_-,M_+,n}(\xi)\;\Delta x^n\;h^{(n)}(x_i+\xi\Delta x)+O(\Delta x^{N_\tsc{tj}+1})
                                                                                                       \notag\\
                                               \stackrel{\eqref{Eq_Prp_RLRPCS_s_RCsSs_ss_Ks_001_001a}}{=}&
                                               \sum_{k_\mathrm{s}=0}^{K_\mathrm{s}}\sigma_{R_1,M_-,M_+,K_\mathrm{s},k_\mathrm{s}}(\xi)\;p_{R_1,M_--k_\mathrm{s},M_+-K_\mathrm{s}+k_\mathrm{s}}(x_i+\xi\Delta x;x_i,\Delta x;f)
                                                                                                       \notag\\
                                               \stackrel{\eqref{Eq_Prp_AELRP_s_EPR_ss_AELPR_002_001a}}{=}&
                                               \underbrace{\left(\sum_{k_\mathrm{s}=0}^{K_\mathrm{s}}\sigma_{R_1,M_-,M_+,K_\mathrm{s},k_\mathrm{s}}(\xi)\right)}_{\displaystyle =1\;\eqref{Eq_Prp_RLRPCS_s_RCsSs_ss_Ks_001_001c}}
                                               h(x_i+\xi\Delta x)
                                                                                                       \notag\\
                                             +&\sum_{n=M-K_\mathrm{s}+1}^{N_\tsc{tj}}\left(\sum_{k_\mathrm{s}=0}^{K_\mathrm{s}}\sigma_{R_1,M_-,M_+,K_\mathrm{s},k_\mathrm{s}}(\xi)\;
                                                                                                            \lambda_{R_1,M_--k_\mathrm{s},M_+-K_\mathrm{s}+k_\mathrm{s},n}(\xi)\right)\;\Delta x^n\;h^{(n)}(x_i+\xi\Delta x)
                                                                                                       \notag\\
                                             +&O(\Delta x^{N_\tsc{tj}+1})
                                                                                                       \label{Eq_Crl_RLRPCS_s_RCsSs_ss_Ks_002_002a}
\end{alignat}
\begin{alignat}{6}
p_{R_1,M_-,M_+}(x_i+\xi\Delta x;x_i,\Delta x;f)\stackrel{\eqref{Eq_Prp_AELRP_s_EPR_ss_AELPR_001_001c}}{=}&
                                               h(x_i+\xi\Delta x)+\sum_{n=M+1}^{N_\tsc{tj}}    \mu_{R_1,M_-,M_+,n}(\xi)\;\Delta x^n\;f^{(n)}(x_i)+O(\Delta x^{N_\tsc{tj}+1})
                                                                                                       \notag\\
                                               \stackrel{\eqref{Eq_Prp_RLRPCS_s_RCsSs_ss_Ks_001_001a}}{=}&
                                               \sum_{k_\mathrm{s}=0}^{K_\mathrm{s}}\sigma_{R_1,M_-,M_+,K_\mathrm{s},k_\mathrm{s}}(\xi)\;p_{R_1,M_--k_\mathrm{s},M_+-K_\mathrm{s}+k_\mathrm{s}}(x_i+\xi\Delta x;x_i,\Delta x;f)
                                                                                                       \notag\\
                                               \stackrel{\eqref{Eq_Prp_AELRP_s_EPR_ss_AELPR_001_001c}}{=}&
                                               \underbrace{\left(\sum_{k_\mathrm{s}=0}^{K_\mathrm{s}}\sigma_{R_1,M_-,M_+,K_\mathrm{s},k_\mathrm{s}}(\xi)\right)}_{\displaystyle =1\;\eqref{Eq_Prp_RLRPCS_s_RCsSs_ss_Ks_001_001c}}
                                               h(x_i+\xi\Delta x)
                                                                                                       \notag\\
                                             +&\sum_{n=M-K_\mathrm{s}+1}^{N_\tsc{tj}}\left(\sum_{k_\mathrm{s}=0}^{K_\mathrm{s}}\sigma_{R_1,M_-,M_+,K_\mathrm{s},k_\mathrm{s}}(\xi)\;
                                                                                                                \mu_{R_1,M_--k_\mathrm{s},M_+-K_\mathrm{s}+k_\mathrm{s},n}(\xi)\right)\;\Delta x^n\;f^{(n)}(x_i)
                                                                                                       \notag\\
                                             +&O(\Delta x^{N_\tsc{tj}+1})
                                                                                                       \label{Eq_Crl_RLRPCS_s_RCsSs_ss_Ks_002_002b}
\end{alignat}
\begin{alignat}{6}
                                                                                           \forall \xi\in{\mathbb R}\setminus\mathcal{S}_{R_1,M_-,M_+,K_\mathrm{s}}
                                                                                           \forall x_i\in{\mathbb R}\quad
                                                                                           \forall \Delta x\in{\mathbb R}_{>0}\quad
                                                                                           \forall h\in C^{N_\tsc{tj}+1}(\mathbb{R})\quad
                                                                                                   f:=R^{-1}_{(1,\Delta x)}(h)
                                                                                                       \notag
\end{alignat}
which prove \eqref{Eq_Crl_RLRPCS_s_RCsSs_ss_Ks_002_001} by identification of coefficients of $\Delta x^n$.
\end{subequations}
\qed
\end{proof}
%
%
\begin{example}[{\rm Rational weight-functions $\sigma_{R_1,M_-,M_+,K_\mathrm{s},k_\mathrm{s}}(\xi)$ \eqref{Eq_Prp_RLRPCS_s_RCsSs_ss_Ks_001_001b}}]
\label{Xmp_RLRPCS_s_RCsSs_ss_Ks_001}
The stencil $\tsc{s}_{i,3,3}$ ($M_-=3$, $M_+=3$, $M:=M_-+M_+=6$) is symmetric around $\xi=0$ \figref{Fig_Xmp_RLRPCS_s_RCsSs_ss_sSs_001_001}.
The ($K_\mathrm{s}=\left\lceil\frac{M}{2}\right\rceil=3$)-level subdivision \defref{Def_RLRPCS_s_I_002} is the highest level of
subdivision for which all of the substencils contain either point $i$ or point $i+1$ \figref{Fig_Xmp_RLRPCS_s_RCsSs_ss_sSs_001_001}.
The rational weight-functions $\sigma_{R_1,3,3,3,k_\mathrm{s}}(\xi)$ ($k_\mathrm{s}\in\{0,\cdots,3\}$)
are all $>0$ in the interval $I_{\tsc{c}_{R_1}(\tfrac{1}{2}),3,3,3}$ around point $\xi=+\tfrac{1}{2}$ \figref{Fig_Xmp_RLRPCS_s_RCsSs_ss_Ks_001_001}.
Because of the symmetry of the stencil $\tsc{s}_{i,3,3}$ around $\xi=0$ \figref{Fig_Xmp_RLRPCS_s_RCsSs_ss_Ks_001_001},
we also have $\sigma_{R_1,3,3,3,k_\mathrm{s}}(-\tfrac{1}{2})>0\;\forall k_\mathrm{s}\in\{0,\cdots,3\}$.
The stencil $\tsc{s}_{i,3,4}$ ($M_-=3$, $M_+=4$, $M:=M_-+M_+=7$) is symmetric around $\xi=\tfrac{1}{2}$ \figref{Fig_Xmp_RLRPCS_s_RCsSs_ss_sSs_001_002}.
The ($K_\mathrm{s}=\left\lceil\frac{M}{2}\right\rceil=4$)-level subdivision \defref{Def_RLRPCS_s_I_002} is the highest level of
subdivision for which all of the substencils contain either point $i$ or point $i+1$\figref{Fig_Xmp_RLRPCS_s_RCsSs_ss_sSs_001_002}.
The rational weight-functions $\sigma_{R_1,3,4,4,k_\mathrm{s}}(\xi)$ ($k_\mathrm{s}\in\{0,\cdots,4\}$)
are all $>0$ in the interval $I_{\tsc{c}_{R_1}(\tfrac{1}{2}),3,4,4}$ around point $\xi=+\tfrac{1}{2}$ \figref{Fig_Xmp_RLRPCS_s_RCsSs_ss_Ks_001_002}.
The stencil $\tsc{s}_{i,3,4}$ not being symmetric around $\xi=0$ \figref{Fig_Xmp_RLRPCS_s_RCsSs_ss_Ks_001_002}, positivity of the weight-functions does not
hold around $\xi=-\tfrac{1}{2}$, where $\sigma_{R_1,3,4,4,4}(-\tfrac{1}{2})=-\tfrac{3}{770}$, by direct computation using \eqref{Eq_Prp_RLRPCS_s_RCsSs_ss_Ks_001_001b}.
The conditions of positivity of weight-functions at $\xi=+\tfrac{1}{2}$, which is important in the development of \tsc{weno} schemes \cite{Shu_2009a},
are studied below (\S\ref{RLRPCS_s_RCsSs_ss_C}).
\qed
\end{example}
%
%
\begin{figure}[ht!]
\begin{picture}(500,180)
\put(0,-10){\includegraphics[angle=0,width=400pt]{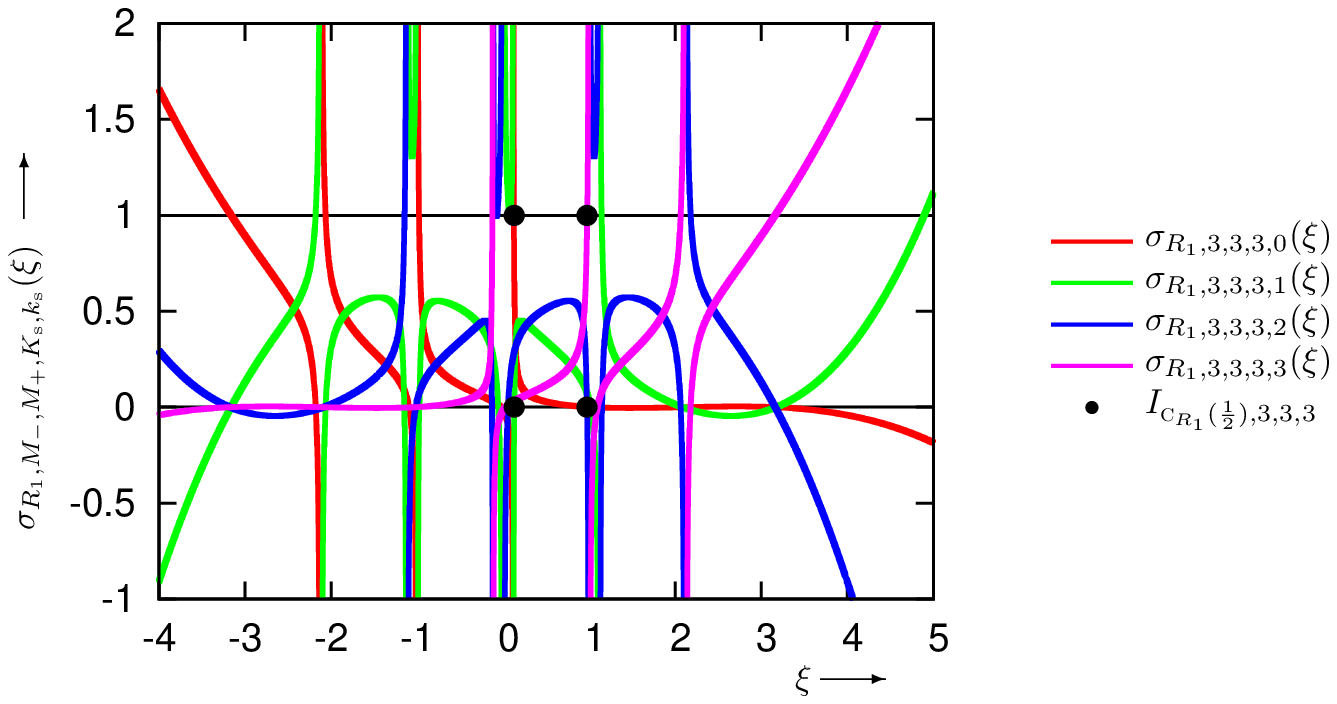}}
\end{picture}
\caption{Rational weight-functions $\sigma_{R_1,M_-,M_+,K_\mathrm{s},k_\mathrm{s}}(\xi)$ \eqref{Eq_Prp_RLRPCS_s_RCsSs_ss_Ks_001_001b} for
the ($K_\mathrm{s}=3$)-level subdivition \defref{Def_RLRPCS_s_I_002} of the stencil $\tsc{s}_{i,3,3}$ \figref{Fig_Xmp_RLRPCS_s_FPLIR_ss_RFPs_001_001},
and interval of convexity of the weight-functions around $i+\tfrac{1}{2}$, $I_{\tsc{c}_{R_1}(\tfrac{1}{2}),M_-,M_+,K_\mathrm{s}}$ \thmref{Thm_RLRPCS_s_RCsSs_ss_C_001}.}
\label{Fig_Xmp_RLRPCS_s_RCsSs_ss_Ks_001_001}
%
%
\begin{picture}(500,210)
\put(0,-10){\includegraphics[angle=0,width=400pt]{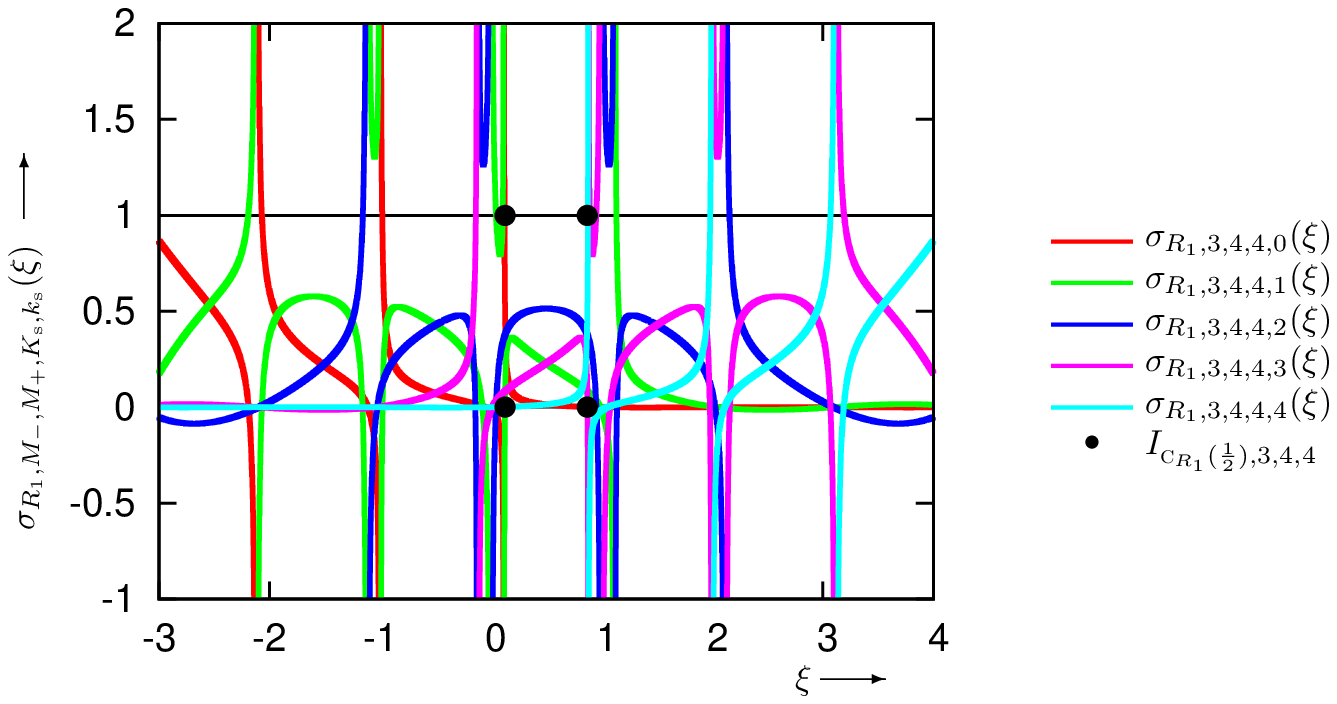}}
\end{picture}
\caption{Rational weight-functions $\sigma_{R_1,M_-,M_+,K_\mathrm{s},k_\mathrm{s}}(\xi)$ \eqref{Eq_Prp_RLRPCS_s_RCsSs_ss_Ks_001_001b} for
the ($K_\mathrm{s}=4$)-level subdivition \defref{Def_RLRPCS_s_I_002} of the stencil $\tsc{s}_{i,3,4}$ \figref{Fig_Xmp_RLRPCS_s_FPLIR_ss_RFPs_001_002},
and interval of convexity of the weight-functions around $i+\tfrac{1}{2}$, $I_{\tsc{c}_{R_1}(\tfrac{1}{2}),M_-,M_+,K_\mathrm{s}}$ \thmref{Thm_RLRPCS_s_RCsSs_ss_C_001}.}
\label{Fig_Xmp_RLRPCS_s_RCsSs_ss_Ks_001_002}
\end{figure}
%

%
%
%
%
%
\subsection{Convexity}\label{RLRPCS_s_RCsSs_ss_C}
%
%
%
%
%

The nonlinear modification of the optimal (linear) weights in \tsc{weno} schemes \cite{Jiang_Shu_1996a,
                                                                                       Balsara_Shu_2000a}
is more straightforward when the combination \eqref{Eq_RLRPCS_s_I_001a} is convex \cite{Shi_Hu_Shu_2002a}.
%
\begin{remark}[{\rm Consistency, positivity and convexity}]
\label{Rmk_RLRPCS_s_RCsSs_ss_C_001}
As can be seen by \eqrefsab{Eq_Crl_RLRPCS_s_RCsSs_ss_Ks_002_002a}{Eq_Crl_RLRPCS_s_RCsSs_ss_Ks_002_002b}, condition \eqref{Eq_Prp_RLRPCS_s_RCsSs_ss_Ks_001_001c}
ensures the consistency of the representation \eqref{Eq_Prp_RLRPCS_s_RCsSs_ss_Ks_001_001a} as an approximation of $h(x)=:[R_{(1;\Delta x)}(f)](x)$ \defref{Def_AELRP_s_RPERR_ss_RP_001},
and is therefore called the consistency condition of the representation \eqref{Eq_Prp_RLRPCS_s_RCsSs_ss_Ks_001_001a}.
Obviously, when at a fixed $\xi\in\mathbb{R}$ all of the $K_\mathrm{s}$-level-subdivision weight-functions are $\geq0$
then, because of \eqref{Eq_Prp_RLRPCS_s_RCsSs_ss_Ks_001_001c}, they must take values $\in[0,1]$ (proof by contradiction)
\begin{alignat}{6}
\Big[\sigma_{R_1,M_-,M_+,K_\mathrm{s},k_\mathrm{s}}(\xi)\geq0\;\forall k_\mathrm{s}\in\{0,\cdots,K_\mathrm{s}\}\Big]
\stackrel{\eqref{Eq_Prp_RLRPCS_s_RCsSs_ss_Ks_001_001c}}{\iff}
\Big[0\leq\sigma_{R_1,M_-,M_+,K_\mathrm{s},k_\mathrm{s}}(\xi)\leq1\;\forall k_\mathrm{s}\in\{0,\cdots,K_\mathrm{s}\}\Big]
                                                                                                       \label{Eq_Rmk_RLRPCS_s_RCsSs_ss_C_001_001}
\end{alignat}
Hence, positivity of the weight-functions at a fixed $\xi\in\mathbb{R}$ ensures, by the consistency condition \eqref{Eq_Prp_RLRPCS_s_RCsSs_ss_Ks_001_001c}, that,
locally, the representation \eqref{Eq_Prp_RLRPCS_s_RCsSs_ss_Ks_001_001a} is convex.
\qed
\end{remark}
%

In the early \tsc{weno} papers \cite{Liu_Osher_Chan_1994a,
                                     Jiang_Shu_1996a}
convexity of the combination \eqref{Eq_RLRPCS_s_I_001a} had been postulated, and verified by direct determination of the coefficients
at $\xi=\tfrac{1}{2}$ \cite{Jiang_Shu_1996a,
                            Balsara_Shu_2000a}.
Shu \cite{Shu_1998a} showed examples of combinations of choices of the stencil $\tsc{s}_{i,M_-,M_+}$ \defref{Def_AELRP_s_EPR_ss_PR_001},
of the level of subdivision $K_\mathrm{s}$ \defref{Def_RLRPCS_s_I_002}, and of the location $\xi\in\mathbb{R}$,
for which convexity of \eqref{Eq_RLRPCS_s_I_001a} is lost, and this appeared as a practical problem, not only in 2-D and 3-D unstructured grids \cite{Shi_Hu_Shu_2002a},
but also in the development of centered (central) \tsc{weno} schemes \cite{Qiu_Shu_2002a}. For this reason the intervals of convexity were investigated numerically \cite{Shu_1998a,
                                                                                                                                                                          Shu_2009a,
                                                                                                                                                                          Liu_Shu_Zhang_2009a}.

The analytical results obtained in the present work, in particular the recursive analytical expression of the weight-functions $\sigma_{R_1,M_-,M_+,K_\mathrm{s},k_\mathrm{s}}(\xi)$
\prpref{Prp_RLRPCS_s_RCsSs_ss_Ks_001} and the factorization of the fundamental functions of Lagrange reconstruction
$\alpha_{R_1,M_-,M_+,\ell}(\xi)$ \prpref{Prp_RLRPCS_s_FPLIR_ss_RFPs_002}, can be used to study convexity intervals for arbitrary values of $[M_\pm,K_\mathrm{s}]$,
as was recently done for the Lagrange interpolating polynomial \cite[Proposition 3.2]{Gerolymos_2011a_news}.
In \cite[Result 6.1, p. 300]{Gerolymos_2011a} we had conjectured that for any choice of $[M_\pm,K_\mathrm{s}]$ for which all of the
substencils $\tsc{s}_{i,M_--k_\mathrm{s},M_+-K_\mathrm{s}+k_\mathrm{s}}$ ($k_\mathrm{s}\in\{0,\cdots,K_\mathrm{s}\}$) contain
either point $i$ or point $i+1$ (or both), convexity was observed at $\xi=\tfrac{1}{2}$. We provide here a formal proof of this conjecture, and give
an estimate of the interval of convexity around $\xi=\tfrac{1}{2}$.
%
\begin{lemma}[{\rm Positive subdivision}]
\label{Lem_RLRPCS_s_RCsSs_ss_C_001}
Consider the subdivision level $K_\mathrm{s}\geq1$ of $\tsc{s}_{i,M_-,M_+}$ \defref{Def_RLRPCS_s_I_002}. Iff
\begin{subequations}
                                                                                                       \label{Eq_Lem_RLRPCS_s_RCsSs_ss_C_001_001}
\begin{alignat}{6}
&-M_-\leq0<1\leq M_+&
                                                                                                       \label{Eq_Lem_RLRPCS_s_RCsSs_ss_C_001_001a}\\
&1\leq K_\mathrm{s}\leq\min(M_-+1,M_+)&
                                                                                                       \label{Eq_Lem_RLRPCS_s_RCsSs_ss_C_001_001b}
\end{alignat}
then all substencils contain either point $i$ or point $i+1$
\begin{equation}
\tsc{s}_{i,M_--k_\mathrm{s},M_+-K_\mathrm{s}+k_\mathrm{s}}\cap\{i,i+1\}\neq\varnothing\quad \forall k_\mathrm{s}\in\{0,\cdots,K_\mathrm{s}\}
                                                                                                       \label{Eq_Lem_RLRPCS_s_RCsSs_ss_C_001_001c}
\end{equation}
More precisely
\begin{alignat}{6}
\eqrefsab{Eq_Lem_RLRPCS_s_RCsSs_ss_C_001_001a}
         {Eq_Lem_RLRPCS_s_RCsSs_ss_C_001_001b}\iff
\eqref{Eq_Lem_RLRPCS_s_RCsSs_ss_C_001_001c}\iff\left[\begin{array}{lcl} \{i,i+1\}&\subseteq&\tsc{s}_{i,M_-,M_+}\\
                                                                        i        &\in      &\tsc{s}_{i,M_-             ,M_+-K_\mathrm{s}             }\\
                                                                        \{i,i+1\}&\subseteq&\tsc{s}_{i,M_--k_\mathrm{s},M_+-K_\mathrm{s}+k_\mathrm{s}}\;\forall k_\mathrm{s}\in\{1,\cdots,K_\mathrm{s}-1\}\\
                                                                        i+1      &\in      &\tsc{s}_{i,M_--K_\mathrm{s},M_+                          }\\
                                                     \end{array}\right]
                                                                                                       \label{Eq_Lem_RLRPCS_s_RCsSs_ss_C_001_001d}
\end{alignat}
\end{subequations}
A subdivision \defref{Def_RLRPCS_s_I_002} satisfying \eqref{Eq_Lem_RLRPCS_s_RCsSs_ss_C_001_001d} will be called a positive subdivision \cite[Result 6.1, p. 300]{Gerolymos_2011a}.
\end{lemma}
%
%
\begin{proof}
First notice that if all substencils contain either point $i$ or point $i+1$ \eqref{Eq_Lem_RLRPCS_s_RCsSs_ss_C_001_001c} then
so does the entire stencil $\tsc{s}_{i,M_-,M_+}\stackrel{\eqref{Eq_Def_RLRPCS_s_I_002_001d}}{=}\bigcup_{k_\mathrm{s}=0}^{K_\mathrm{s}}\tsc{s}_{i,M_--k_\mathrm{s},M_+-K_\mathrm{s}+k_\mathrm{s}}$.
Taking into account that by hypothesis $K_\mathrm{s}\geq1$, in the condition $\{i,i+1\}\cap\tsc{s}_{i,M_-,M_+}\neq\varnothing$, implies that $\tsc{s}_{i,M_-,M_+}$
must contain both points $i$ and $i+1$ (proof\footnote{\label{ff_Lem_RLRPCS_s_RCsSs_ss_C_001_001}Since by \eqref{Eq_Lem_RLRPCS_s_RCsSs_ss_C_001_001c} each of the substencils
                                                                                                 $\tsc{s}_{i,M_--k_\mathrm{s},M_+-K_\mathrm{s}+k_\mathrm{s}}$ ($k_\mathrm{s}\in\{0,\cdots,K_\mathrm{s}\geq1\}$)
                                                                                                 has a non-empty intersection with $\{i,i+1\}$, so does their union $\tsc{s}_{i,M_-,M_+}$
                                                                                                 \eqref{Eq_Def_RLRPCS_s_I_002_001d}, {\em ie} $\{i,i+1\}\cap\tsc{s}_{i,M_-,M_+}\neq\varnothing$.
                                                                                                 Obviously the conditions $\Big((i+M_+<i<i+1)\vee(i<i+1<i-M_-)\Big)\Longrightarrow\{i,i+1\}\cap\tsc{s}_{i,M_-,M_+}=\varnothing$
                                                                                                 are a contradiction, implying that their negation is true,
                                                                                                 {\em ie} we must have $\Big((i+M_+\geq i)\wedge(i+1\geq i-M_-)\Big)$.
                                                                                                 It turns out that the inequalities in $\Big((i+M_+\geq i)\wedge(i+1\geq i-M_-)\Big)$ must be strict.
                                                                                                 Assuming $i+1 =   i-M_-\Longrightarrow
                                                                                                           i<i+1<i-M_-+1\Longrightarrow
                                                                                                           \{i,i+1\}\cap\tsc{s}_{i,M_--1,M_+-K_\mathrm{s}+1}=\varnothing$
                                                                                                           contradicts \eqref{Eq_Lem_RLRPCS_s_RCsSs_ss_C_001_001c},
                                                                                                           implying $i+1>i-M_-\Longrightarrow M_->-1\stackrel{M_-\in\mathbb{Z}}{\Longrightarrow}M_-\geq0$.
                                                                                                 Assuming $i  = i+M_+\Longrightarrow
                                                                                                           i+1>i>i+M_+-K_\mathrm{s}+(K_\mathrm{s}-1)\Longrightarrow
                                                                                                           \{i,i+1\}\cap\tsc{s}_{i,M_-+(K_\mathrm{s}-1),M_+-K_{\rm s}+(K_\mathrm{s}-1)}=\varnothing$
                                                                                                           contradicts \eqref{Eq_Lem_RLRPCS_s_RCsSs_ss_C_001_001c}, because by hypothesis $K_\mathrm{s}\geq1$,
                                                                                                           implying $i<i+M_+\Longrightarrow M_+>0\stackrel{M_+\in\mathbb{Z}}{\Longrightarrow}M_+\geq1$.
                                                      }
by contradiction taking into account $K_{\rm s}\geq1$).
The condition that both points $\{i,i+1\}$ must be contained in the big stencil $\tsc{s}_{i,M_-,M_+}$ yields
\begin{subequations}
                                                                                                       \label{Eq_Lem_RLRPCS_s_RCsSs_ss_C_001_002}
\begin{alignat}{6}
\eqref{Eq_Lem_RLRPCS_s_RCsSs_ss_C_001_001c}\stackrel{\eqref{Eq_Def_RLRPCS_s_I_002_001d}}{\Longrightarrow}&
\{i,i+1\}\cap\tsc{s}_{i,M_-,M_+}\neq\varnothing&\stackrel{\ffref{ff_Lem_RLRPCS_s_RCsSs_ss_C_001_001}}{\Longrightarrow}&
                                                                                                       \notag\\
&\{i,i+1\}\subset\{i-M_,\cdots,i+M_+\}&\iff& i-M_-\leq i<i+1\leq i+M_+
                                      &\iff&  -M_-\leq 0<1\leq M_+
                                                                                                       \label{Eq_Lem_RLRPCS_s_RCsSs_ss_C_001_002a}
\end{alignat}
proving that \eqref{Eq_Lem_RLRPCS_s_RCsSs_ss_C_001_001a} is a necessary condition for the validity of \eqref{Eq_Lem_RLRPCS_s_RCsSs_ss_C_001_001c}.
Combining \eqrefsab{Eq_Lem_RLRPCS_s_RCsSs_ss_C_001_001c}
                   {Eq_Lem_RLRPCS_s_RCsSs_ss_C_001_002a} implies (proof\footnote{\label{ff_Lem_RLRPCS_s_RCsSs_ss_C_001_002}Assuming $i\notin\tsc{s}_{i,M_-,M_+-K_\mathrm{s}}
                                                                                                                                     \stackrel{\eqref{Eq_Lem_RLRPCS_s_RCsSs_ss_C_001_002a}}{\Longrightarrow}
                                                                                                                                     i>i+M_+-K_\mathrm{s}\Longrightarrow\{i,i+1\}\cap\tsc{s}_{i,M_-,M_+-K_{\rm s}}=\varnothing$
                                                                                                                                     contradicts \eqref{Eq_Lem_RLRPCS_s_RCsSs_ss_C_001_001c} for $k_\mathrm{s}=0$.
                                                                                                                           Assuming $i+1\notin\tsc{s}_{i,M_-+K_\mathrm{s},M_+}
                                                                                                                                     \stackrel{\eqref{Eq_Lem_RLRPCS_s_RCsSs_ss_C_001_002a}}{\Longrightarrow}
                                                                                                                                     i+1<i-M_-+K_\mathrm{s}\Longrightarrow\{i,i+1\}\cap\tsc{s}_{i,M_-+K_\mathrm{s},M_+}=\varnothing$
                                                                                                                                     contradicts \eqref{Eq_Lem_RLRPCS_s_RCsSs_ss_C_001_001c}
                                                                                                                                     for $k_\mathrm{s}=K_\mathrm{s}\geq1$.
                                                                                     }
by contradiction) that $i$ must belong to the leftmost substencil ($k_\mathrm{s}=0$) and $i+1$ must belong to the rightmost substencil ($k_\mathrm{s}=K_\mathrm{s}\geq1$)
\begin{alignat}{6}
\eqrefsab{Eq_Lem_RLRPCS_s_RCsSs_ss_C_001_001c}
         {Eq_Lem_RLRPCS_s_RCsSs_ss_C_001_002a}\stackrel{\ffref{ff_Lem_RLRPCS_s_RCsSs_ss_C_001_002}}{\Longrightarrow}&
                                                             \left[\begin{array}{lclclclcl}i  &\in&\tsc{s}_{i,M_-             ,M_+-K_\mathrm{s}}&\Longrightarrow&i-M_-             &\leq&i  &\leq&i+M_+-K_\mathrm{s}\\
                                                                                           i+1&\in&\tsc{s}_{i,M_-+K_\mathrm{s},M_+             }&\Longrightarrow&i-M_-+K_\mathrm{s}&\leq&i+1&\leq&i+M_+             \\
                                                                   \end{array}\right]
                                                                                                       \notag\\
                                                             \Longrightarrow&K_\mathrm{s}\leq\min(M_-+1,M_+)
                                                                                                       \label{Eq_Lem_RLRPCS_s_RCsSs_ss_C_001_002b}
\end{alignat}
proving that \eqref{Eq_Lem_RLRPCS_s_RCsSs_ss_C_001_001b} is also a necessary condition for the validity of \eqref{Eq_Lem_RLRPCS_s_RCsSs_ss_C_001_001c}.
To complete the proof it suffices to show that \eqrefsab{Eq_Lem_RLRPCS_s_RCsSs_ss_C_001_001a}
                                                        {Eq_Lem_RLRPCS_s_RCsSs_ss_C_001_001b}
are not only necessary but also sufficient conditions for \eqref{Eq_Lem_RLRPCS_s_RCsSs_ss_C_001_001c}.
We have
\begin{alignat}{6}
\eqrefsab{Eq_Lem_RLRPCS_s_RCsSs_ss_C_001_001a}
         {Eq_Lem_RLRPCS_s_RCsSs_ss_C_001_001b}\Longrightarrow&
                                               \left[\begin{array}{ccl} M_+             &\geq&1\\
                                                                       -M_-             &\leq&0\\
                                                                       -M_-+K_\mathrm{s}&\leq&1\\
                                                                        M_+-K_\mathrm{s}&\geq&0\\
                                                     \end{array}\right]
                                              \Longrightarrow
                                               \left[\begin{array}{lcccl}i-M_-             &\leq&i<i+1  &\leq&i+M_+             \\
                                                                         i-M_-+K_\mathrm{s}&\leq&i+1                            \\
                                                                                           &    &i      &\leq&i+M_+-K_\mathrm{s}\\
                                                     \end{array}\right]
                                                                                                       \notag\\
                                              \Longrightarrow&
                                               \left[\begin{array}{lcccll}i-M_-             &\leq&i<i+1  &\leq&i+M_+                                                                               \\
                                                                          i-M_-             &\leq&i      &\leq&i+M_+-K_\mathrm{s}                                                                  \\
                                                                          i-M_-+k_\mathrm{s}&<   &i+1    &    &                               &\;\forall k_\mathrm{s}\in\{0,\cdots,K_\mathrm{s}-1\}\\
                                                                                            &    &i      &<   &i+M_+-K_\mathrm{s}+k_\mathrm{s}&\;\forall k_\mathrm{s}\in\{1,\cdots,K_\mathrm{s}  \}\\
                                                                          i-M_-+K_\mathrm{s}&\leq&i+1    &\leq&i+M_+                                                                               \\
                                                     \end{array}\right]
                                                                                                       \notag\\
                                              \Longrightarrow&
                                               \left[\begin{array}{lcccll}i-M_-             &\leq&i<i+1  &\leq&i+M_+                                                                                \\
                                                                          i-M_-             &\leq&i      &\leq&i+M_+-K_\mathrm{s }                                                                  \\
                                                                          i-M_-+k_\mathrm{s}&\leq&i<i+1  &\leq&i+M_+-K_\mathrm{s}+k_\mathrm{s} &\;\forall k_\mathrm{s}\in\{1,\cdots,K_\mathrm{s}-1\}\\
                                                                          i-M_-+K_\mathrm{s}&\leq&i+1    &\leq&i+M_+                                                                                \\
                                                     \end{array}\right]
                                                                                                       \label{Eq_Lem_RLRPCS_s_RCsSs_ss_C_001_002c}
\end{alignat}
completing the proof, the last conditions in \eqref{Eq_Lem_RLRPCS_s_RCsSs_ss_C_001_002c} being exactly \eqref{Eq_Lem_RLRPCS_s_RCsSs_ss_C_001_001d}.
\end{subequations}
\qed
\end{proof}
%
%
\begin{corollary}[{\rm $(K_\mathrm{s}=1)$-level positively subdivisible stencils}]
\label{Crl_RLRPCS_s_RCsSs_ss_C_001}
Assume that $M_\pm\in\mathbb{Z}:M:=M_-+M_+\geq 2$ defining the stencil $\tsc{s}_{i,M_-,M_+}$ \defref{Def_AELRP_s_EPR_ss_PR_001} satisfy $-M_-\leq0<1\leq M_+$ \eqref{Eq_Lem_RLRPCS_s_RCsSs_ss_C_001_001a}.
Then the $(K_\mathrm{s}=1)$-level subdivision of the stencil $\tsc{s}_{i,M_-,M_+}$ \defref{Def_RLRPCS_s_I_002} is a positive subdivision \lemref{Lem_RLRPCS_s_RCsSs_ss_C_001}.
\end{corollary}
%
%
\begin{proof}
By \eqref{Eq_Lem_RLRPCS_s_RCsSs_ss_C_001_001a} we have that $\Big((M_-\geq0)\wedge(M_+\geq1)\Big)\Longrightarrow \min(M_-+1,M_+)\geq1$, so that $K_\mathrm{s}=1\leq \min(M_-+1,M_+)$.
Hence the conditions \eqrefsab{Eq_Lem_RLRPCS_s_RCsSs_ss_C_001_001a}{Eq_Lem_RLRPCS_s_RCsSs_ss_C_001_001b} are satisfied, so that, by \lemrefnp{Lem_RLRPCS_s_RCsSs_ss_C_001},
the $(K_\mathrm{s}=1)$-level subdivision of a stencil satisfying \eqref{Eq_Lem_RLRPCS_s_RCsSs_ss_C_001_001a} is a a positive subdivision.
\qed
\end{proof}
%
%
\begin{lemma}[{\rm Convexity in the neighborhood of $i+\tfrac{1}{2}$ for $(K_{\rm s}=1)$-level subdivision}]
\label{Lem_RLRPCS_s_RCsSs_ss_C_002}
\begin{subequations}
                                                                                                       \label{Eq_Lem_RLRPCS_s_RCsSs_ss_C_002_001}
Assume that $M_\pm\in\mathbb{Z}:M:=M_-+M_+\geq 2$ \eqref{Eq_Def_RLRPCS_s_I_002_001a}
defining the stencil $\tsc{s}_{i,M_-,M_+}$ \defref{Def_AELRP_s_EPR_ss_PR_001} satisfy $-M_-\leq0<1\leq M_+$ \eqref{Eq_Lem_RLRPCS_s_RCsSs_ss_C_001_001a}.
Then the rational weight-functions
$\sigma_{R_1,M_-,M+,1,0}(\xi)$ \eqref{Eq_Lem_RLRPCS_s_RCsSs_ss_Ks1_001_001a} and
$\sigma_{R_1,M_-,M+,1,1}(\xi)$ \eqref{Eq_Lem_RLRPCS_s_RCsSs_ss_Ks1_001_001b}
for the representation of the Lagrange reconstructing polynomial by the Lagrange reconstructing polynomials
of the $(K_\mathrm{s}=1)$-level substencils of $\tsc{s}_{i,M_-,M_+}$ \lemref{Lem_RLRPCS_s_RCsSs_ss_Ks1_001} satisfy
\begin{equation}
0<\sigma_{R_1,M_-,M+,1,k_{\rm s}}(\xi)<1\qquad\left\{\begin{array}{l}\forall \xi\in I_{\tsc{c}_{R_1}(\tfrac{1}{2}),M_-,M_+,1}:=\left(\xi^-_{\tsc{c}_{R_1}(\frac{1}{2}),M_-,M_+,1},
                                                                                                                                     \xi^+_{\tsc{c}_{R_1}(\frac{1}{2}),M_-,M_+,1}\right)\subset\mathbb{R}\\
                                                                     \forall k_{\rm s}\in\{0,1\}\\\end{array}\right.
                                                                                                       \label{Eq_Lem_RLRPCS_s_RCsSs_ss_C_002_001a}
\end{equation}
where the limits of the convexity interval around $\xi=\tfrac{1}{2}$, $I_{\tsc{c}_{R_1}(\frac{1}{2}),M_-,M_+,1}\ni\tfrac{1}{2}$ of length $>0$,
are defined by
\begin{alignat}{6}
\xi^-_{\tsc{c}_{R_1}(\frac{1}{2}),M_-,M_+,1}:=&\left\{\begin{array}{ll}         \xi_{R_1,M_-,M_+,+M_+,0}   &\quad M_-=0\\
                                                                       \max\left(\xi_{R_1,M_-,M_+,+M_+,0},
                                                                                \xi_{R_1,M_-,M_+,-M_-,0},
                                                                                \xi_{R_1,M_-,M_+-1,-M_-,0}\right)&\quad M_->0\\
                                                      \end{array}\right.
                                                                                                       \label{Eq_Lem_RLRPCS_s_RCsSs_ss_C_002_001b}\\
\xi^+_{\tsc{c}_{R_1}(\frac{1}{2}),M_-,M_+,1}:=&\left\{\begin{array}{ll}         \xi_{R_1,M_-,M_+,-M_-,1}       &\quad M_+=1\\
                                                                       \min\left(\xi_{R_1,M_-,M_+,-M_-,1},
                                                                                \xi_{R_1,M_-,M_+,+M_+,1},
                                                                                \xi_{R_1,M_--1,M_+,+M_+,1}\right)&\quad M_+>1\\
                                                      \end{array}\right.
                                                                                                       \label{Eq_Lem_RLRPCS_s_RCsSs_ss_C_002_001c}
\end{alignat}
where $\xi_{R_1,M_-,M_+,\ell,n}$ ($n\in\{-M_-,\cdots,M_+\}\setminus\{\ell\}$) are the $M$ real roots \prpref{Prp_RLRPCS_s_FPLIR_ss_RFPs_001} of the fundamental polynomial
of Lagrange reconstruction $\alpha_{R_1,M_-,M_+,\ell}(\xi)$ \eqref{Eq_Prp_AELRP_s_EPR_ss_PR_001_001g}.
\end{subequations}
\end{lemma}
%
%
\begin{proof}
\begin{subequations}
                                                                                                       \label{Eq_Lem_RLRPCS_s_RCsSs_ss_C_002_002}
By hypothesis, the stencil $\tsc{s}_{i,M_-,M_+}$ satisfies the conditions of \crlrefnp{Crl_RLRPCS_s_RCsSs_ss_C_001}, implying that the ($K_\mathrm{s}=1$)-level
subdivision of $\tsc{s}_{i,M_-,M_+}$ is a positive subdivision, satisfying the conditions of \lemrefnp{Lem_RLRPCS_s_RCsSs_ss_C_001}, and we have by \eqref{Eq_Lem_RLRPCS_s_RCsSs_ss_C_001_001d}
\begin{alignat}{6}
\{i,i+1\}&\subseteq&\tsc{s}_{i,M_-,M_+}
                                                                                                       \label{Eq_Lem_RLRPCS_s_RCsSs_ss_C_002_002a}\\
i        &\in      &\tsc{s}_{i,M_-,M_+-1}
                                                                                                       \label{Eq_Lem_RLRPCS_s_RCsSs_ss_C_002_002b}\\
i+1      &\in      &\tsc{s}_{i,M_--1,M_+}
                                                                                                       \label{Eq_Lem_RLRPCS_s_RCsSs_ss_C_002_002c}
\end{alignat}
By \lemrefnp{Lem_RLRPCS_s_RCsSs_ss_Ks1_001} the rational weight-functions
$\sigma_{R_1,M_-,M+,1,0}(\xi)$ \eqref{Eq_Lem_RLRPCS_s_RCsSs_ss_Ks1_001_001a} and
$\sigma_{R_1,M_-,M+,1,1}(\xi)$ \eqref{Eq_Lem_RLRPCS_s_RCsSs_ss_Ks1_001_001b}
can be expressed in terms of the fundamental polynomials of Lagrange reconstruction \prpref{Prp_RLRPCS_s_RB_001}
$\alpha_{R_1,M_-,M_+,-M_-}(\xi)$, $\alpha_{R_1,M_-,M_+,+M_+}(\xi)$, $\alpha_{R_1,M_-,M_+-1,-M_-}(\xi)$, and $\alpha_{R_1,M_--1,M_+,+M_+}(\xi)$.
Notice that, because of the identities of \prprefnp{Prp_RLRPCS_s_FPLIR_ss_SICFPLR_001}, we have
$\alpha_{R_1,M_--1,M_+,M_+}(\xi)\stackrel{\eqref{Eq_Prp_RLRPCS_s_FPLIR_ss_SICFPLR_001_001a}}{=}(-1)^{M-1}\;\alpha_{R_1,M_-,M_+-1,-M_-}(\xi)$.
By \prprefnp{Prp_RLRPCS_s_FPLIR_ss_RFPs_001}, all of the roots of the fundamental polynomials of Lagrange reconstruction are real,
and therefore the factorization of \prprefnp{Prp_RLRPCS_s_FPLIR_ss_RFPs_002} applies. Applying the factorization \eqref{Eq_Prp_RLRPCS_s_FPLIR_ss_RFPs_002_001a},
and taking into account \eqrefsab{Eq_Lem_RLRPCS_s_RCsSs_ss_C_002_002a}{Eq_Lem_RLRPCS_s_RCsSs_ss_C_002_002b},
which were shown in \lemrefnp{Lem_RLRPCS_s_RCsSs_ss_C_001} to be direct consequences of \eqref{Eq_Lem_RLRPCS_s_RCsSs_ss_C_001_001a}, we have
\begin{alignat}{6}
\alpha_{R_1,M_-,M_+,-M_-}(\xi)\stackrel{\eqref{Eq_Prp_RLRPCS_s_FPLIR_ss_RFPs_002_001a}}{=}&\frac{(-1)^M}
                                                                                                {    M!}\prod_{n=-M_-+1}^{M_+}(\xi-\xi_{R_1,M_-,M_+,-M_-,n})
                                                                                                       \notag\\
                              \stackrel{\eqref{Eq_Lem_RLRPCS_s_RCsSs_ss_C_001_001a}}{=}&
\frac{(-1)^M}
     {    M!}\left\{\begin{array}{lll}                                                                                                                      &
                                      \displaystyle\prod_{n=1}^{M_+}\overbrace{(\xi-\xi_{R_1,M_-,M_+,-M_-,n})}^{<0\;\forall\xi<\xi_{R_1,M_-,M_+,-M_-,1}}    &\quad M_-=0\\
                                      \displaystyle\prod_{n=-M_-+1}^{0}\underbrace{(\xi-\xi_{R_1,M_-,M_+,-M_-,n})}_{>0\;\forall\xi>\xi_{R_1,M_-,M_+,-M_-,0}}&
                                      \displaystyle\prod_{n=1}^{M_+}\underbrace{(\xi-\xi_{R_1,M_-,M_+,-M_-,n})}_{<0\;\forall\xi<\xi_{R_1,M_-,M_+,-M_-,1}}   &\quad M_->0\\\end{array}\right.
                                                                                                       \label{Eq_Lem_RLRPCS_s_RCsSs_ss_C_002_002d}\\
\alpha_{R_1,M_-,M_+,+M_+}(\xi)\stackrel{\eqref{Eq_Prp_RLRPCS_s_FPLIR_ss_RFPs_002_001a}}{=}&\frac{(-1)^{2M_+}}
                                                                                                {         M!}\prod_{n=-M_-}^{M_+-1}(\xi-\xi_{R_1,M_-,M_+,+M_+,n})
                                                                                                       \notag\\
                              \stackrel{\eqref{Eq_Lem_RLRPCS_s_RCsSs_ss_C_001_001a}}{=}&
\frac{(-1)^{2M_+}}
     {         M!}\left\{\begin{array}{lll}\displaystyle\prod_{n=-M_-}^{0}\overbrace{(\xi-\xi_{R_1,M_-,M_+,+M_+,n})}^{>0\;\forall\xi>\xi_{R_1,M_-,M_+,+M_+,0}}   &
                                                                                                                                                                 &\quad M_+=1\\
                                           \displaystyle\prod_{n=-M_-}^{0}\underbrace{(\xi-\xi_{R_1,M_-,M_+,+M_+,n})}_{>0\;\forall\xi>\xi_{R_1,M_-,M_+,+M_+,0}}  &
                                           \displaystyle\prod_{n=1}^{M_+-1}\underbrace{(\xi-\xi_{R_1,M_-,M_+,+M_+,n})}_{<0\;\forall\xi<\xi_{R_1,M_-,M_+,+M_+,1}} &\quad M_+>1\\\end{array}\right.
                                                                                                       \label{Eq_Lem_RLRPCS_s_RCsSs_ss_C_002_002e}\\
\alpha_{R_1,M_-,M_+-1,-M_-}(\xi)\stackrel{\eqref{Eq_Prp_RLRPCS_s_FPLIR_ss_RFPs_002_001a}}{=}&\frac{(-1)^{M-1}}
                                                                                                  {    (M-1)!}\prod_{n=-M_-+1}^{M_+-1}(\xi-\xi_{R_1,M_-,M_+-1,-M_-,n})
                                                                                                       \notag\\
                              \stackrel{\eqref{Eq_Lem_RLRPCS_s_RCsSs_ss_C_001_001a}}{=}&
\frac{(-1)^{M-1}}
     {    (M-1)!}\left\{\begin{array}{lll}                                                                                                                      &
                                      \displaystyle\prod_{n=1}^{M_+-1}\overbrace{(\xi-\xi_{R_1,M_-,M_+-1,-M_-,n})}^{<0\;\forall\xi<\xi_{R_1,M_-,M_+-1,-M_-,1}}  &\quad M_-=0\\
                                      \displaystyle\prod_{n=-M_-+1}^{0}\overbrace{(\xi-\xi_{R_1,M_-,M_+-1,-M_-,n})}^{>0\;\forall\xi>\xi_{R_1,M_-,M_+-1,-M_-,0}} &
                                                                                                                                                                &\quad M_+=1\\
                                      \displaystyle\prod_{n=-M_-+1}^{0}\underbrace{(\xi-\xi_{R_1,M_-,M_+-1,-M_-,n})}_{>0\;\forall\xi>\xi_{R_1,M_-,M_+-1,-M_-,0}}&
                                      \displaystyle\prod_{n=1}^{M_+-1}\underbrace{(\xi-\xi_{R_1,M_-,M_+-1,-M_-,n})}_{<0\;\forall\xi<\xi_{R_1,M_-,M_+-1,-M_-,1}} &\quad M_-\neq0\neq M_+-1\\\end{array}\right.
                                                                                                       \label{Eq_Lem_RLRPCS_s_RCsSs_ss_C_002_002f}
\end{alignat}
where in \eqref{Eq_Lem_RLRPCS_s_RCsSs_ss_C_002_002f} we only need to distinguish 3 cases because the constraint $M:=M_-+M_+\geq2$ \eqref{Eq_Def_RLRPCS_s_I_002_001a} implies
that we cannot have simultaneously $M_-=0$ and $M_+=1$. Since by \prprefnp{Prp_RLRPCS_s_FPLIR_ss_SICFPLR_001},
$\alpha_{R_1,M_--1,M_+,M_+}(\xi)\stackrel{\eqref{Eq_Prp_RLRPCS_s_FPLIR_ss_SICFPLR_001_001a}}{=}(-1)^{M-1}\;\alpha_{R_1,M_-,M_+-1,-M_-}(\xi)
 \Longrightarrow\xi_{R_1,M_-,M_+-1,-M_-,n}=\xi_{R_1,M_--1,M_+,M_+,n}\;\forall n\in\{-M_-+1,\cdots,M_+-1\}$,
defining the limits of the open convexity interval around $\xi=\tfrac{1}{2}$, $I_{\tsc{c}_{R_1}(\frac{1}{2}),M_-,M_+,1}$ \eqref{Eq_Lem_RLRPCS_s_RCsSs_ss_C_002_001a},
by \eqrefsab{Eq_Lem_RLRPCS_s_RCsSs_ss_C_002_001b}{Eq_Lem_RLRPCS_s_RCsSs_ss_C_002_001c}, we have
\begin{alignat}{6}
\mathrm{sign}\Big(\alpha_{R_1,M_-,M_+,-M_-}(\xi)\Big)\stackrel{\eqref{Eq_Lem_RLRPCS_s_RCsSs_ss_C_002_002d}}{=}&(-1)^{M+M_+}
                                                     \stackrel{\eqref{Eq_Def_RLRPCS_s_I_002_001a}}{=}(-1)^{M_-}
                                                     &\quad\forall\xi\in I_{\tsc{c}_{R_1}(\frac{1}{2}),M_-,M_+,1}
                                                                                                       \label{Eq_Lem_RLRPCS_s_RCsSs_ss_C_002_002g}\\
\mathrm{sign}\Big(\alpha_{R_1,M_-,M_+,+M_+}(\xi)\Big)\stackrel{\eqref{Eq_Lem_RLRPCS_s_RCsSs_ss_C_002_002e}}{=}&(-1)^{2M_++M_+-1}=(-1)^{M_+-1}
                                                     &\quad\forall\xi\in I_{\tsc{c}_{R_1}(\frac{1}{2}),M_-,M_+,1}
                                                                                                       \label{Eq_Lem_RLRPCS_s_RCsSs_ss_C_002_002h}\\
\mathrm{sign}\Big(\alpha_{R_1,M_-,M_+-1,-M_-}(\xi)\Big)\stackrel{\eqref{Eq_Lem_RLRPCS_s_RCsSs_ss_C_002_002f}}{=}&(-1)^{M-1+M_+-1}
                                                     \stackrel{\eqref{Eq_Def_RLRPCS_s_I_002_001a}}{=}(-1)^{M_-}
                                                     &\quad\forall\xi\in I_{\tsc{c}_{R_1}(\frac{1}{2}),M_-,M_+,1}
                                                                                                       \label{Eq_Lem_RLRPCS_s_RCsSs_ss_C_002_002i}\\
\mathrm{sign}\Big(\alpha_{R_1,M_--1,M_+,+M_+}(\xi)\Big)\stackrel{\eqref{Eq_Prp_RLRPCS_s_FPLIR_ss_SICFPLR_001_001a}}{=}&(-1)^{M-1}\mathrm{sign}\Big(\alpha_{R_1,M_-,M_+-1,-M_-}(\xi)\Big)
                                                                                                       \notag\\
                                                        \stackrel{\eqref{Eq_Lem_RLRPCS_s_RCsSs_ss_C_002_002f}}{=}&(-1)^{M-1+M_-}=(-1)^{M_+-1}
                                                     &\quad\forall\xi\in I_{\tsc{c}_{R_1}(\frac{1}{2}),M_-,M_+,1}
                                                                                                       \label{Eq_Lem_RLRPCS_s_RCsSs_ss_C_002_002j}
\end{alignat}
whence
\begin{alignat}{6}
\sigma_{R_1,M_-,M_+,1,0}(\xi)\stackrel{\eqrefsabc{Eq_Lem_RLRPCS_s_RCsSs_ss_Ks1_001_001a}
                                                 {Eq_Lem_RLRPCS_s_RCsSs_ss_C_002_002g}
                                                 {Eq_Lem_RLRPCS_s_RCsSs_ss_C_002_002i}}{>}0\qquad\forall\xi\in I_{\tsc{c}_{R_1}(\tfrac{1}{2}),M_-,M_+,1}
                                                                                                       \label{Eq_Lem_RLRPCS_s_RCsSs_ss_C_002_002k}\\
\sigma_{R_1,M_-,M_+,1,1}(\xi)\stackrel{\eqrefsabc{Eq_Lem_RLRPCS_s_RCsSs_ss_Ks1_001_001b}
                                                 {Eq_Lem_RLRPCS_s_RCsSs_ss_C_002_002h}
                                                 {Eq_Lem_RLRPCS_s_RCsSs_ss_C_002_002j}}{>}0\qquad\forall\xi\in I_{\tsc{c}_{R_1}(\tfrac{1}{2}),M_-,M_+,1}
                                                                                                       \label{Eq_Lem_RLRPCS_s_RCsSs_ss_C_002_002l}
\end{alignat}
Because of the consistency condition \eqref{Eq_Lem_RLRPCS_s_RCsSs_ss_Ks1_001_001c}, positivity of the weight-functions implies convexity \rmkref{Rmk_RLRPCS_s_RCsSs_ss_C_001},
so that \eqrefsab{Eq_Lem_RLRPCS_s_RCsSs_ss_C_002_002k}{Eq_Lem_RLRPCS_s_RCsSs_ss_C_002_002l} prove \eqref{Eq_Lem_RLRPCS_s_RCsSs_ss_C_002_001a}.
Notice that, by \prprefnp{Prp_RLRPCS_s_FPLIR_ss_RFPs_001}, \eqref{Eq_Prp_RLRPCS_s_FPLIR_ss_RFPs_001_001a} implies that $\xi_{R_1,M_-,M_+,\ell\neq0,0}<\tfrac{1}{2}$
and $\xi_{R_1,M_-,M_+,\ell\neq1,1}>\tfrac{1}{2}$, $\forall M_\pm\in\mathbb{Z}:M:=M_-+M_+>2$, satisfying the conditions of \crlrefnp{Crl_RLRPCS_s_RCsSs_ss_C_001},
so that the length of $I_{\tsc{c}_{R_1}(\frac{1}{2}),M_-,M_+,1}$ \eqrefsatob{Eq_Lem_RLRPCS_s_RCsSs_ss_C_002_001a}
                                                                             {Eq_Lem_RLRPCS_s_RCsSs_ss_C_002_001c}
is $>0$.
\end{subequations}
\qed
\end{proof}
%
%
\begin{theorem}[{\rm Convexity of a positive subdivision in the neighborhood of $i+\tfrac{1}{2}$}]
\label{Thm_RLRPCS_s_RCsSs_ss_C_001}
\begin{subequations}
                                                                                                       \label{Eq_Thm_RLRPCS_s_RCsSs_ss_C_001_001}
Assume that the subdivision level $K_\mathrm{s}\geq1$ of $\tsc{s}_{i,M_-,M_+}$ \defref{Def_RLRPCS_s_I_002} satisfies the conditions of
\lemrefnp{Lem_RLRPCS_s_RCsSs_ss_C_001} (positive subdivision), {\em viz}
\begin{alignat}{6}
&M:=M_-+M_+\geq2&
                                                                                                       \tag{\text{\ref{Eq_Def_RLRPCS_s_I_002_001a}}}\\
&-M_-\leq0<1\leq M_+&
                                                                                                       \tag{\text{\ref{Eq_Lem_RLRPCS_s_RCsSs_ss_C_001_001a}}}\\
&1\leq K_\mathrm{s}\leq\min(M_-+1,M_+)&
                                                                                                       \tag{\text{\ref{Eq_Lem_RLRPCS_s_RCsSs_ss_C_001_001b}}}
\end{alignat}
implying \lemref{Lem_RLRPCS_s_RCsSs_ss_C_001} that all substencils contain either point $i$ or point $i+1$. Define the interval
\begin{alignat}{6}
I_{\tsc{c}_{R_1}(\frac{1}{2}),M_-,M_+,K_\mathrm{s}}:= \bigcap_{   L_\mathrm{s}=0}^{K_\mathrm{s}-1}\displaylimits
                                                      \bigcap_{\ell_\mathrm{s}=0}^{L_\mathrm{s}  }I_{\tsc{c}_{R_1}(\frac{1}{2}),M_--\ell_\mathrm{s},M_+-L_\mathrm{s}+\ell_\mathrm{s},1}
                                                                                                       \label{Eq_Thm_RLRPCS_s_RCsSs_ss_C_001_001a}
\end{alignat}
recursively using convexity intervals $I_{\tsc{c}_{R_1}(\frac{1}{2}),M_-,M_+,1}$ \eqrefsatob{Eq_Lem_RLRPCS_s_RCsSs_ss_C_002_001a}{Eq_Lem_RLRPCS_s_RCsSs_ss_C_002_001c}
of $(K_\mathrm{s}=1)$-level positive subdivisions \lemref{Lem_RLRPCS_s_RCsSs_ss_C_002}.
Then the rational weight-functions $\sigma_{R_1,M_-,M_+,K_{\rm s},k_{\rm s}}(\xi)$ \eqref{Eq_Prp_RLRPCS_s_RCsSs_ss_Ks_001_001b} satisfy
\begin{equation}
0<\sigma_{R_1,M_-,M_+,K_{\rm s},k_{\rm s}}(\xi)<1\qquad\left\{\begin{array}{l}\forall \xi\in I_{\tsc{c}_{R_1}(\tfrac{1}{2}),M_-,M_+,K_\mathrm{s}}\\
                                                                              \forall K_\mathrm{s}\in\{1,\cdots,M-1\}                            \\
                                                                              \forall k_{\rm s}\in\{0,K_\mathrm{s}\}                             \\\end{array}\right.
                                                                                                       \label{Eq_Thm_RLRPCS_s_RCsSs_ss_C_001_001b}
\end{equation}
implying convexity of the combination \eqref{Eq_Prp_RLRPCS_s_RCsSs_ss_Ks_001_001a}.
\end{subequations}
\end{theorem}
%
%
\begin{proof}
\begin{subequations}
                                                                                                       \label{Eq_Thm_RLRPCS_s_RCsSs_ss_C_001_002}
The validity of \eqrefsab{Eq_Thm_RLRPCS_s_RCsSs_ss_C_001_001a}{Eq_Thm_RLRPCS_s_RCsSs_ss_C_001_001b} for $K_\mathrm{s}=1$ was proven in \lemrefnp{Lem_RLRPCS_s_RCsSs_ss_C_002}.
Assume $\min(M_-+1,M_+)\geq2$ so that the ($K_\mathrm{s}=2$)-level subdivision be a positive subdivision \lemref{Lem_RLRPCS_s_RCsSs_ss_C_001}.
Then, by \eqref{Eq_Prp_RLRPCS_s_RCsSs_ss_Ks_001_001b}, we have
\begin{alignat}{6}
\sigma_{R_1,M_-,M_+,2,0}(\xi)\stackrel{\eqref{Eq_Prp_RLRPCS_s_RCsSs_ss_Ks_001_001b}}{=}&\sigma_{R_1,M_-  ,M_+  ,1,0}(\xi)\;
                                                                                        \sigma_{R_1,M_-  ,M_+-1,1,0}(\xi)
                                                                                                       \label{Eq_Thm_RLRPCS_s_RCsSs_ss_C_001_002a}\\
\sigma_{R_1,M_-,M_+,2,1}(\xi)\stackrel{\eqref{Eq_Prp_RLRPCS_s_RCsSs_ss_Ks_001_001b}}{=}&\sigma_{R_1,M_-  ,M_+  ,1,0}(\xi)\;
                                                                                        \sigma_{R_1,M_-  ,M_+-1,1,1}(\xi)
                                                                                                       \notag\\
                                                                                     + &\sigma_{R_1,M_-  ,M_+  ,1,1}(\xi)\;
                                                                                        \sigma_{R_1,M_--1,M_+  ,1,0}(\xi)
                                                                                                       \label{Eq_Thm_RLRPCS_s_RCsSs_ss_C_001_002b}\\
\sigma_{R_1,M_-,M_+,2,2}(\xi)\stackrel{\eqref{Eq_Prp_RLRPCS_s_RCsSs_ss_Ks_001_001b}}{=}&\sigma_{R_1,M_-  ,M_+  ,1,1}(\xi)\;
                                                                                        \sigma_{R_1,M_--1,M_+  ,1,1}(\xi)
                                                                                                       \label{Eq_Thm_RLRPCS_s_RCsSs_ss_C_001_002c}
\end{alignat}
Having assumed that the ($K_\mathrm{s}=2$)-level subdivision is a positive subdivision \lemref{Lem_RLRPCS_s_RCsSs_ss_C_001}, we have
\begin{alignat}{6}
\min(M_-+1,M_+)\geq2\Longrightarrow\left[\begin{array}{l}M_-+1\geq2\stackrel{\eqref{Eq_Def_RLRPCS_s_I_002_001a}}{\Longrightarrow}-(M_--1)\leq0<1\leq M_+\\
                                                         M_+\geq2\stackrel{\eqref{Eq_Def_RLRPCS_s_I_002_001a}}{\Longrightarrow}-M_-\leq0<1\leq (M_+-1)\\\end{array}\right]
                                                                                                       \label{Eq_Thm_RLRPCS_s_RCsSs_ss_C_001_002d}
\end{alignat}
implying that the $1$-level subdivisions of the stencils $\tsc{s}_{i,M_--1,M_+}$ and $\tsc{s}_{i,M_-,M_+-1}$ are positive \crlref{Crl_RLRPCS_s_RCsSs_ss_C_001}. Therefore all of the $1$-level
weight-functions on the \tsc{rhs} of \eqrefsatob{Eq_Thm_RLRPCS_s_RCsSs_ss_C_001_002a}{Eq_Thm_RLRPCS_s_RCsSs_ss_C_001_002c} are positive in the neighborhood of $\xi=\tfrac{1}{2}$,
because of \lemrefnp{Lem_RLRPCS_s_RCsSs_ss_C_002}, and we have
\begin{alignat}{6}
0<\sigma_{R_1,M_-,M_+,2,0}(\xi)&\quad\forall \xi\in\Big(I_{\tsc{c}_{R_1}(\tfrac{1}{2}),M_-,M_+,1}\cap I_{\tsc{c}_{R_1}(\tfrac{1}{2}),M_-,M_+-1,1}\Big)
                                                                                                       \label{Eq_Thm_RLRPCS_s_RCsSs_ss_C_001_002e}\\
0<\sigma_{R_1,M_-,M_+,2,1}(\xi)&\quad\forall \xi\in\Big(I_{\tsc{c}_{R_1}(\tfrac{1}{2}),M_-,M_+,1}\cap I_{\tsc{c}_{R_1}(\tfrac{1}{2}),M_-,M_+-1,1}\cap I_{\tsc{c}_{R_1}(\tfrac{1}{2}),M_--1,M_+,1}\Big)
                                                                                                       \label{Eq_Thm_RLRPCS_s_RCsSs_ss_C_001_002f}\\
0<\sigma_{R_1,M_-,M_+,2,2}(\xi)&\quad\forall \xi\in\Big(I_{\tsc{c}_{R_1}(\tfrac{1}{2}),M_-,M_+,1}                                                \cap I_{\tsc{c}_{R_1}(\tfrac{1}{2}),M_--1,M_+,1}\Big)
                                                                                                       \label{Eq_Thm_RLRPCS_s_RCsSs_ss_C_001_002g}
\end{alignat}
Defining
\begin{alignat}{6}
I_{\tsc{c}_{R_1}(\frac{1}{2}),M_-,M_+,2}\stackrel{\eqref{Eq_Thm_RLRPCS_s_RCsSs_ss_C_001_001a}}{:=} \bigcap_{   L_\mathrm{s}=0}^{1             }\displaylimits
                                                                                                   \bigcap_{\ell_\mathrm{s}=0}^{L_\mathrm{s}  }I_{\tsc{c}_{R_1}(\frac{1}{2}),M_--\ell_\mathrm{s},M_+-L_\mathrm{s}+\ell_\mathrm{s},1}
                                                                                                =I_{\tsc{c}_{R_1}(\tfrac{1}{2}),M_-,M_+,1}\cap I_{\tsc{c}_{R_1}(\tfrac{1}{2}),M_-,M_+-1,1}\cap I_{\tsc{c}_{R_1}(\tfrac{1}{2}),M_--1,M_+,1}
                                                                                                       \label{Eq_Thm_RLRPCS_s_RCsSs_ss_C_001_002h}
\end{alignat}
we have that all of the 3 ($K_\mathrm{s}=2$)-level weight-functions \eqrefsatob{Eq_Thm_RLRPCS_s_RCsSs_ss_C_001_002e}{Eq_Thm_RLRPCS_s_RCsSs_ss_C_001_002g} are simultaneously positive
$\forall\xi\in I_{\tsc{c}_{R_1}(\frac{1}{2}),M_-,M_+,2}$, which \rmkref{Rmk_RLRPCS_s_RCsSs_ss_C_001}, because of the consistency condition \eqref{Eq_Prp_RLRPCS_s_RCsSs_ss_Ks_001_001b},
proves \eqrefsab{Eq_Thm_RLRPCS_s_RCsSs_ss_C_001_001a}{Eq_Thm_RLRPCS_s_RCsSs_ss_C_001_001b} for $K_\mathrm{s}=2$.

It is straightforward to complete the proof by induction. Since we have already proved \eqrefsab{Eq_Thm_RLRPCS_s_RCsSs_ss_C_001_001a}{Eq_Thm_RLRPCS_s_RCsSs_ss_C_001_001b} for $K_\mathrm{s}=2$,
assume $K_\mathrm{s}-1\geq2\iff K_\mathrm{s}\geq3$. By \prprefnp{Prp_RLRPCS_s_RCsSs_ss_Ks_001}
\begin{alignat}{6}
\sigma_{R_1,M_-,M_+,K_\mathrm{s},k_\mathrm{s}}(\xi)\stackrel{\eqref{Eq_Prp_RLRPCS_s_RCsSs_ss_Ks_001_001b}}{=}
                                                   \sum_{\ell_\mathrm{s}=\max(0,k_\mathrm{s}-1)}^{\min(K_\mathrm{s}-1,k_\mathrm{s})}
                                                   \sigma_{R_1,M_-             ,M_+                                 ,K_\mathrm{s}-1,             \ell_\mathrm{s}}(\xi)\;
                                                   \sigma_{R_1,M_--\ell_{\rm s},M_+-(K_\mathrm{s}-1)+\ell_\mathrm{s},1             ,k_\mathrm{s}-\ell_\mathrm{s}}(\xi)
                                                   \quad\forall k_\mathrm{s}\in\{0,\cdots,K_\mathrm{s}\}
                                                                                                       \label{Eq_Thm_RLRPCS_s_RCsSs_ss_C_001_002i}
\end{alignat}
Assume that \eqrefsab{Eq_Thm_RLRPCS_s_RCsSs_ss_C_001_001a}{Eq_Thm_RLRPCS_s_RCsSs_ss_C_001_001b} are valid for $K_\mathrm{s}-1\geq2$
\begin{equation}
0<\sigma_{R_1,M_-,M_+,K_{\rm s}-1,\ell_{\rm s}}(\xi)<1\qquad\left\{\begin{array}{l}\displaystyle\forall \xi\in I_{\tsc{c}_{R_1}(\tfrac{1}{2}),M_-,M_+,K_\mathrm{s}-1}
                                                                                               =\bigcap_{   L_\mathrm{s}=0}^{K_\mathrm{s}-2}\displaylimits
                                                                                                \bigcap_{\ell_\mathrm{s}=0}^{L_\mathrm{s}  }
                                                                                                I_{\tsc{c}_{R_1}(\frac{1}{2}),M_--\ell_\mathrm{s},M_+-L_\mathrm{s}+\ell_\mathrm{s},1}\\
                                                                                                                                                          \\
                                                                                  \forall \ell_{\rm s}\in\{0,K_\mathrm{s}-1\}                             \\\end{array}\right.
                                                                                                       \label{Eq_Thm_RLRPCS_s_RCsSs_ss_C_001_002j}
\end{equation}
and that $K_\mathrm{s}$ satisfies \eqref{Eq_Lem_RLRPCS_s_RCsSs_ss_C_001_001b}, {\em ie}
\begin{alignat}{6}
\min(M_-+1,M_+)\geq K_\mathrm{s}\Longrightarrow\left[\begin{array}{l}M_-+1\geq K_\mathrm{s}\\
                                                                     M_+  \geq K_\mathrm{s}\\\end{array}\right]\stackrel{\eqref{Eq_Def_RLRPCS_s_I_002_001a}}{\Longrightarrow}
                                               -(M_--\ell_\mathrm{s})\leq0<1\leq(M_+-(K_\mathrm{s}-1)+\ell_\mathrm{s})
                                                                                                       \notag\\
\quad\forall \ell_\mathrm{s}\in\{0,\cdots,K_\mathrm{s}-1\}
                                                                                                       \label{Eq_Thm_RLRPCS_s_RCsSs_ss_C_001_002k}
\end{alignat}
so that the substencils $\tsc{s}_{i,M_--\ell_{\rm s},M_+-(K_\mathrm{s}-1)+\ell_\mathrm{s}}$ satisfy the conditions of \crlrefnp{Crl_RLRPCS_s_RCsSs_ss_C_001}, implying by \lemrefnp{Lem_RLRPCS_s_RCsSs_ss_C_002} that
\begin{alignat}{6}
0<\sigma_{R_1,M_--\ell_\mathrm{s},M_+-(K_\mathrm{s}-1)+\ell_\mathrm{s},1,m_s}(\xi)\quad\forall \xi\in I_{\tsc{c}_{R_1}(\tfrac{1}{2}),M_--\ell_\mathrm{s},M_+-(K_\mathrm{s}-1)+\ell_\mathrm{s},1}
\qquad\left\{\begin{array}{l}\forall    m_\mathrm{s}\in\{0,1\}\\
                             \forall \ell_\mathrm{s}\in\{0,\cdots,K_\mathrm{s}-1\}\\\end{array}\right.
                                                                                                       \label{Eq_Thm_RLRPCS_s_RCsSs_ss_C_001_002l}
\end{alignat}
Combining \eqrefsabc{Eq_Thm_RLRPCS_s_RCsSs_ss_C_001_002i}{Eq_Thm_RLRPCS_s_RCsSs_ss_C_001_002j}{Eq_Thm_RLRPCS_s_RCsSs_ss_C_001_002l} yields
\begin{equation}
0<\sigma_{R_1,M_-,M_+,K_{\rm s},k_{\rm s}}(\xi)  \qquad\left\{\begin{array}{l}\displaystyle\forall \xi\in I_{\tsc{c}_{R_1}(\tfrac{1}{2}),M_-,M_+,K_\mathrm{s}-1}\cap
                                                                                          \left(\bigcap_{\ell_\mathrm{s}=0}^{K_\mathrm{s}-1}
                                                                                          I_{\tsc{c}_{R_1}(\frac{1}{2}),M_--\ell_\mathrm{s},M_+-(K_\mathrm{s}-1)+\ell_\mathrm{s},1}\right)\\
                                                                                                                                                                                          \\
                                                                                          \forall k_{\rm s}\in\{0,K_\mathrm{s}\}                                                          \\\end{array}\right.
                                                                                                       \label{Eq_Thm_RLRPCS_s_RCsSs_ss_C_001_002m}
\end{equation}
which \rmkref{Rmk_RLRPCS_s_RCsSs_ss_C_001}, because of the consistency condition \eqref{Eq_Prp_RLRPCS_s_RCsSs_ss_Ks_001_001b},
proves \eqrefsab{Eq_Thm_RLRPCS_s_RCsSs_ss_C_001_001a}{Eq_Thm_RLRPCS_s_RCsSs_ss_C_001_001b} $\forall K_\mathrm{s}$ satisfying \eqref{Eq_Lem_RLRPCS_s_RCsSs_ss_C_001_001b}.
\end{subequations}
\qed
\end{proof}
%
%
\begin{figure}[ht!]
\begin{picture}(500,200)
\put(50,-10){\includegraphics[angle=0,width=300pt]{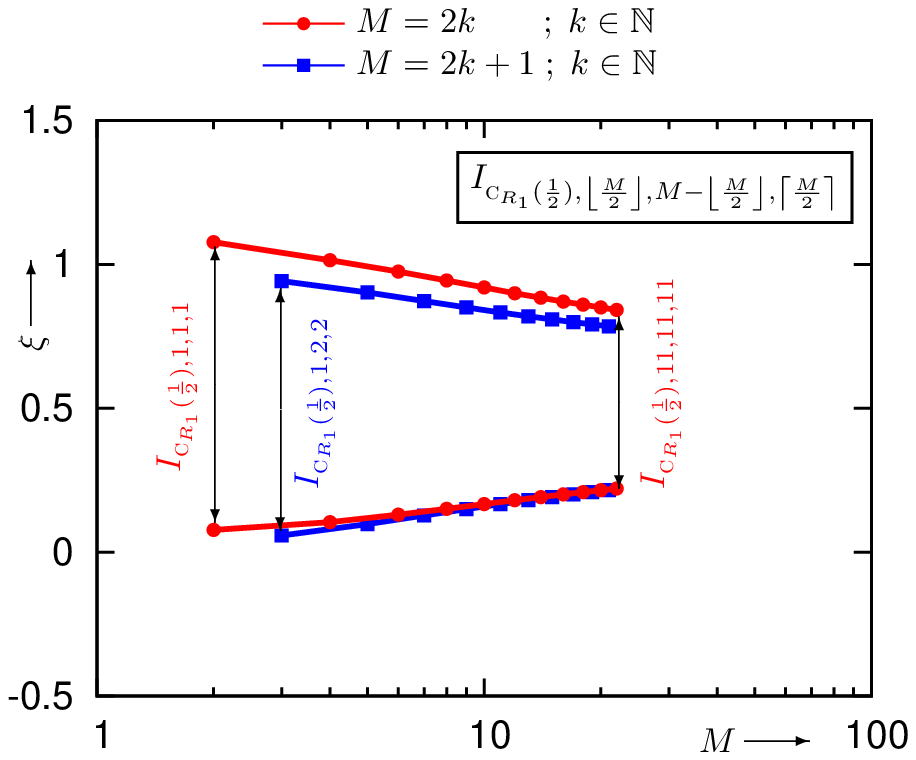}}
\end{picture}
\caption{ Interval of convexity $I_{\tsc{c}_{R_1}(\frac{1}{2}),\left\lfloor\frac{M}{2}\right\rfloor,M-\left\lfloor\frac{M}{2}\right\rfloor,\left\lceil\frac{M}{2}\right\rceil}$ around
$i+\tfrac{1}{2}$ \thmref{Thm_RLRPCS_s_RCsSs_ss_C_001}, of the maximum positive subdivision-level \lemref{Lem_RLRPCS_s_RCsSs_ss_C_001} $K_\mathrm{s}=\left\lceil\frac{M}{2}\right\rceil$,
of the usual \tsc{weno} stencils $\tsc{s}_{i,\left\lfloor\frac{M}{2}\right\rfloor,M-\left\lfloor\frac{M}{2}\right\rfloor}$ \defref{Def_AELRP_s_EPR_ss_PR_001},
as a function of stencil width $M$ (in logscale).}
\label{Fig_Xmp_RLRPCS_s_RCsSs_ss_C_001_001}
\end{figure}
%
%
\begin{example}[{\rm Convexity around $i+\tfrac{1}{2}$ of usual \tsc{weno} discretizations}]
\label{Xmp_RLRPCS_s_RCsSs_ss_C_001}
The usual \tsc{weno} discretizations for the numerical approximation of $f'(x)$ \cite{Shu_1998a,
                                                                                      Shu_2009a,
                                                                                      Liu_Shu_Zhang_2009a}
use \cite[p. 298]{Gerolymos_2011a} the ($K_\mathrm{s}=\left\lceil\frac{M}{2}\right\rceil$)-level subdivision \defref{Def_RLRPCS_s_I_002}
of the general family of stencils $\tsc{s}_{i,\left\lfloor\frac{M}{2}\right\rfloor,M-\left\lfloor\frac{M}{2}\right\rfloor}$,
which \defref{Def_AELRP_s_EPR_ss_PR_001} contains $M+1$ points.
If $M=2k$ ($k\in\mathbb{N}_{>0}$) is even,
then the stencil $\tsc{s}_{i,\left\lfloor\frac{M}{2}\right\rfloor,M-\left\lfloor\frac{M}{2}\right\rfloor}=\tsc{s}_{i,k,k}$ is symmetric around point $i$,
and upwind-biased with respect to the cell-interface $i+\tfrac{1}{2}$ ({\em eg} $\tsc{s}_{i,3,3}$; \figrefnp{Fig_Xmp_RLRPCS_s_RCsSs_ss_Ks_001_001}),
corresponding to the family of \tscn{weno}{$(2r-1)$} ($r:=k+1$) upwind-biased schemes \cite{Jiang_Shu_1996a,
                                                                                            Balsara_Shu_2000a,
                                                                                            Gerolymos_Senechal_Vallet_2009a}.
If $M=2k+1$ ($k\in\mathbb{N}_{>0}$) is odd,
then the stencil $\tsc{s}_{i,\left\lfloor\frac{M}{2}\right\rfloor,M-\left\lfloor\frac{M}{2}\right\rfloor}=\tsc{s}_{i,k,k+1}$ is symmetric around the cell-interface $i+\tfrac{1}{2}$,
and downwind-biased with respect to the point $i$ ({\em eg} $\tsc{s}_{i,3,4}$; \figrefnp{Fig_Xmp_RLRPCS_s_RCsSs_ss_Ks_001_002}),
corresponding to centered (central) \tsc{weno} schemes \cite{Qiu_Shu_2002a}.
For the family of stencils $\tsc{s}_{i,\left\lfloor\frac{M}{2}\right\rfloor,M-\left\lfloor\frac{M}{2}\right\rfloor}$,
we have
\begin{subequations}
                                                                                                       \label{Eq_Xmp_RLRPCS_s_RCsSs_ss_C_001_001}
\begin{alignat}{6}
M_-=&\left\lfloor\frac{M}{2}\right\rfloor
                                                                                                       \label{Eq_Xmp_RLRPCS_s_RCsSs_ss_C_001_001a}\\
M_+=&M-\left\lfloor\frac{M}{2}\right\rfloor
                                                                                                       \label{Eq_Xmp_RLRPCS_s_RCsSs_ss_C_001_001b}\\
\min(M_-+1,M_+)=&\left\{\begin{array}{l}\min(k+1,k~~~~~~\,)=k~~~~~~  \qquad \forall M=2k~~~~~~  \;;\;k\in\mathbb{N}_{>0}\\
                                        \min(k+1,k+1      )=k+1      \qquad \forall M=2k+1      \;;\;k\in\mathbb{N}_{>0}\\\end{array}\right\}=\left\lceil\frac{M}{2}\right\rceil\quad\forall M\in\mathbb{N}_{\geq2}
                                                                                                       \label{Eq_Xmp_RLRPCS_s_RCsSs_ss_C_001_001c}
\end{alignat}
\end{subequations}
so that, by \lemrefnp{Lem_RLRPCS_s_RCsSs_ss_C_001},
$K_\mathrm{s}=\min(\left\lfloor\frac{M}{2}\right\rfloor+1,M-\left\lfloor\frac{M}{2}\right\rfloor)\stackrel{\eqref{Eq_Xmp_RLRPCS_s_RCsSs_ss_C_001_001}}{=}\left\lceil\frac{M}{2}\right\rceil$
corresponds to the maximum level of positive subdivision of $\tsc{s}_{i,\left\lfloor\frac{M}{2}\right\rfloor,M-\left\lfloor\frac{M}{2}\right\rfloor}$.
Therefore, \thmrefnp{Thm_RLRPCS_s_RCsSs_ss_C_001} applies, and there exists an interval of convexity
$I_{\tsc{c}_{R_1}(\frac{1}{2}),\left\lfloor\frac{M}{2}\right\rfloor,M-\left\lfloor\frac{M}{2}\right\rfloor,\left\lceil\frac{M}{2}\right\rceil}$ \eqref{Eq_Thm_RLRPCS_s_RCsSs_ss_C_001_001a}
around $\xi=\tfrac{1}{2}$ \figref{Fig_Xmp_RLRPCS_s_RCsSs_ss_C_001_001}.

Notice that the interval of convexity $I_{\tsc{c}_{R_1}(\frac{1}{2}),\left\lfloor\frac{M}{2}\right\rfloor,M-\left\lfloor\frac{M}{2}\right\rfloor,\left\lceil\frac{M}{2}\right\rceil}$ \figref{Fig_Xmp_RLRPCS_s_RCsSs_ss_C_001_001},
is slightly larger for $M=2k$ ($k\in\mathbb{N}_{>0}$) even, compared to $M=2k+1$ ($k\in\mathbb{N}_{>0}$) odd,
and its length slightly decreases (quasi-logarithmically $\forall M\in\{2,\cdots,22\}$) with increasing number of cells in the stencil, $M$ \figref{Fig_Xmp_RLRPCS_s_RCsSs_ss_C_001_001}.
For stencils with $M=2k$ ($k\in\mathbb{N}_{>0}$) even, like $\tsc{s}_{i,3,3}$ \figref{Fig_Xmp_RLRPCS_s_RCsSs_ss_sSs_001_001}, because of symmetry with respect to point $i$,
it is straightforward to show that there is a symmetric interval of convexity around $i-\tfrac{1}{2}$ \figref{Fig_Xmp_RLRPCS_s_RCsSs_ss_Ks_001_001},
as was also observed in \cite[Tab. 3.2, p. 516]{Liu_Shu_Zhang_2009a}.
On the contrary, for stencils with $M=2k+1$ ($k\in\mathbb{N}_{>0}$) odd, like $\tsc{s}_{i,3,4}$ \figref{Fig_Xmp_RLRPCS_s_RCsSs_ss_sSs_001_002},
it turns out that positivity of the weight-functions does not hold at $\xi=-\tfrac{1}{2}$ (\xmprefnp{Xmp_RLRPCS_s_RCsSs_ss_Ks_001}; \figrefnp{Fig_Xmp_RLRPCS_s_RCsSs_ss_Ks_001_002}),
as was also observed in \cite[Tab. 3.5, p. 518]{Liu_Shu_Zhang_2009a}.
\qed
\end{example}
%

%
%
%
%
%
%
%
%
%
\section{Conclusions}\label{RLRPCS_s_C}
%
%
%
%
%
%
%
%
%

In the present work, we studied analytically the representation of the Lagrange reconstructing polynomial by combination of substencils, and in particular
the conditions under which this representation is convex, {\em ie} the weight-functions $\in[0,1]$.

We first formalized several results on the fundamental polynomials of Lagrange reconstruction \prpref{Prp_RLRPCS_s_RB_001},
$\alpha_{R_1,M_-,M_+,\ell}(\xi)$ \eqref{Eq_Prp_AELRP_s_EPR_ss_PR_001_001g}. Each of the polynomials $\alpha_{R_1,M_-,M_+,\ell}(\xi)$ is the reconstruction pair \defref{Def_AELRP_s_RPERR_ss_RP_001}
of the corresponding fundamental function of Lagrange interpolation $\alpha_{I,M_-,M_+,\ell}(\xi)$ \prpref{Prp_RLRPCS_s_FPLIR_ss_RPsFPs_001},
and for this reason all of its $M$ roots are real \prpref{Prp_RLRPCS_s_FPLIR_ss_RFPs_001},
distant $<\tfrac{1}{2}$ from the corresponding root of the fundamental function of Lagrange interpolation $\alpha_{I,M_-,M_+,\ell}(\xi)$ \eqref{Eq_Prp_AELRP_s_EPR_ss_PR_001_001h}.
This leads to a simple factorization of the fundamental polynomials of Lagrange reconstruction \prpref{Prp_RLRPCS_s_FPLIR_ss_RFPs_002}.

The leading $O(\Delta x^M)$ term of the approximation error of the Lagrange reconstructing polynomials on 2 overlapping stencils shifted by 1 cell, $\{i-M_-,\cdots,i+M_+-1\}$ and $\{i-M_-+1,\cdots,i+M_+\}$,
is different \prpref{Prp_RLRPCS_s_FPLIR_ss_SICFPLR_001}, and several identities hold between some of the fundamental polynomials on the 2 stencils. Based on these identities \prpref{Prp_RLRPCS_s_FPLIR_ss_SICFPLR_001},
we show that there exist unique rational weight-functions combining the Lagrange reconstructing polynomials on $\{i-M_-,\cdots,i+M_+-1\}$ and $\{i-M_-+1,\cdots,i+M_+\}$
into the Lagrange reconstructing polynomials on $\{i-M_-,\cdots,i+M_+\}$ \lemref{Lem_RLRPCS_s_RCsSs_ss_Ks1_001}, this representation failing at the poles of the weight-functions,
all of which are real and can be identified with roots of fundamental polynomials of Lagrange reconstruction. Having established this $1$-level subdivision rule,
the general recurrence relation for the weight-functions proven in \cite[Lemma 2.1]{Gerolymos_2011a_news} applies, and provides
the analytical expression of the weight-functions for a general level of subdivision \prpref{Prp_RLRPCS_s_RCsSs_ss_Ks_001}, and of the set of their poles, all of which are real.
These weight-functions are unique \prpref{Prp_RLRPCS_s_RCsSs_ss_Ks_002}.

Finally, we prove \thmref{Thm_RLRPCS_s_RCsSs_ss_C_001} that for any $K_\mathrm{s}$-level subdivision of $\{i-M_-,\cdots,i+M_+\}$ into $K_\mathrm{s}+1$ substencils
$\left\{i-M_-+k_\mathrm{s},\cdots,M_+-K_\mathrm{s}+k_\mathrm{s}\right\}$ ($k_\mathrm{s}\in\{0,\cdots,K_\mathrm{s}\}$), iff each of the substencils contains either
point $i$ or point $i+1$ (positive subdivision; \lemrefnp{Lem_RLRPCS_s_RCsSs_ss_C_001}), then there exists a neighborhood of $\xi=\tfrac{1}{2}$ ($x=x_i+\tfrac{1}{2}\Delta x$),
$I_{\tsc{c}_{R_1}(\frac{1}{2}),M_-,M_+,K_\mathrm{s}}\ni\tfrac{1}{2}$,
whose limits can be explicitly defined by roots of fundamental polynomials of Lagrange reconstruction,
where all of the weight-functions $\sigma_{R_1,M_-,M+,K_{\rm s},k_{\rm s}}(\xi)>0\;\forall\xi\in I_{\tsc{c}_{R_1}(\frac{1}{2}),M_-,M_+,K_\mathrm{s}}$, implying because of the consistency relation
$\sum_{k_\mathrm{s}=0}^{K_\mathrm{s}}\sigma_{R_1,M_-,M_+,K_\mathrm{s},k_\mathrm{s}}(\xi)=1\;\forall\xi\in{\mathbb R}$ \eqref{Eq_Prp_RLRPCS_s_RCsSs_ss_Ks_001_001c},
that the representation of the Lagrange reconstructing polynomial by combination of substencils is convex $\forall\xi\in I_{\tsc{c}_{R_1}(\frac{1}{2}),M_-,M_+,K_\mathrm{s}}\ni\tfrac{1}{2}$.
\thmrefnp{Thm_RLRPCS_s_RCsSs_ss_C_001} provides a formal proof of (and general conditions for) convexity in the neighborhood of $\xi=\tfrac{1}{2}$,
which had always been conjectured, on the basis of numerical evidence, all along the development of \tsc{weno} schemes \cite{Liu_Osher_Chan_1994a,
                                                                                                                             Jiang_Shu_1996a,
                                                                                                                             Shu_1998a,
                                                                                                                             Balsara_Shu_2000a,
                                                                                                                             Shu_2009a,
                                                                                                                             Liu_Shu_Zhang_2009a,
                                                                                                                             Gerolymos_Senechal_Vallet_2009a}.

%
%
%
%
%
%
%
%
%
\section*{Acknowledgments}
%
%
%
%
%
%
%
%
%

Computations were performed using \tsc{hpc} resources from \tsc{genci--idris} (Grants 2010--066327 and 2010--022139).
Symbolic calculations were performed using {\tt maxima} ({\tt http://sourceforge.net/projects/maxima}).
The corresponding package {\tt reconstruction.mac} is available at {\tt http://www.aerodynamics.fr}.

%
%
%
%
%
%
%
%
%
\footnotesize\bibliographystyle{elsarticle-num}\bibliography{Aerodynamics,GV,GV_news,Aerodynamics_in_press}
%
%
%
%
%
%
%
%
%
%
%
%
%

\end{document}